%% file: main.tex
\documentclass[a4paper,11pt]{amsart}
\usepackage{lmodern}
\usepackage{graphicx}
\usepackage[margin=3cm]{geometry}
\usepackage[utf8]{inputenc}
\usepackage[T1]{fontenc}
\usepackage{amssymb}
\usepackage{amsthm}
\usepackage{mathtools}
\usepackage{bbm}
\usepackage[mathscr]{eucal}
\usepackage[dvipsnames]{xcolor}
\usepackage[shortlabels]{enumitem}
\usepackage{tikz}
\usetikzlibrary{decorations.pathreplacing,calligraphy,patterns,calc}
\usepackage{tikz-cd}
\usepackage{hyperref}
\usepackage{ytableau}
\ytableausetup{smalltableaux,aligntableaux=center}
\usepackage{soul}
\usepackage{stmaryrd}

\newtheorem{cor}{Corollary}[section]
\newtheorem{deff}[cor]{Definition}
\newtheorem{nota}[cor]{Notation}
\newtheorem{ex}[cor]{Example}
\newtheorem{lem}[cor]{Lemma}
\newtheorem{prop}[cor]{Proposition}
\newtheorem{rem}[cor]{Remark}
\newtheorem{theo}[cor]{Theorem}

\newcommand{\Z}{\mathbb{Z}}
\newcommand{\Q}{\mathbb{Q}}
\newcommand{\F}{\mathbb{F}}

\newcommand{\kk}{\mathbbm{k}}
\newcommand{\Zm}{\Z[\delta]_{\mathfrak{m}}}

\newcommand{\p}{p}

\newcommand{\qn}[2][q]{\llbracket #2\rrbracket_{#1}}
\newcommand{\qnd}[2][\delta]{\llbracket #2\rrbracket_{#1}}
\newcommand{\qnn}[1]{\llbracket #1\rrbracket}

\DeclarePairedDelimiter\abs{\lvert}{\rvert}
\DeclarePairedDelimiter\ip{\langle}{\rangle}

\DeclareMathOperator{\Hom}{Hom}
\DeclareMathOperator{\End}{End}
\DeclareMathOperator{\Ima}{Im}

\DeclareMathOperator{\rad}{rad}
\DeclareMathOperator{\Std}{Std}
\DeclareMathOperator{\id}{\mathsf{id}}
\DeclareMathOperator{\spn}{span}

\newcommand{\bleq}{\mathrel{\mathpalette\bleqinn\relax}}
\newcommand{\bleqinn}[2]{%
  \ooalign{%
    \raisebox{.2ex}{$#1\blacktriangleleft$}\cr
    $#1\leq$\cr
  }%
}

\newcommand{\tlmod}[1]{\mathcal{#1}}
\newcommand{\ctlmod}[2][n]{\tlmod{S}(#1,#2)}
\newcommand{\stlmod}[2][n]{\tlmod{L}(#1,#2)}

\newcommand{\uqmod}[1]{\mathsf{#1}}
\newcommand{\wmod}[1]{\Delta(#1)}

\newcommand{\tmod}[1]{\uqmod{T}(#1)}

\newcommand{\supp}[1]{\text{supp}(#1)}
\newcommand{\mo}[2][]{\mathrm{m}^{#1}_{#2}}

\newcommand{\TL}{\mathsf{TL}}
\newcommand{\tlnkd}{\TL_n^\kk(\delta)}
\newcommand{\tlnzd}{\TL_n^{\Zm}(\delta)}
\newcommand{\uq}{\boldsymbol{\mathrm{U}}_q}

\newcommand{\tlcat}{\mathscr{TL}}
\newcommand{\sltwo}{\mathfrak{sl}_2}
\newcommand{\uqsl}{\uq(\sltwo)}

\newcommand{\td}{\mathbf{td}}
\newcommand{\trunc}[2][n]{\mathcal{T}_{#1}^{#2}}

\newcommand{\jwlp}{{\overline{\mathsf{JW}}}}
\newcommand{\jwlpz}{\mathsf{JW}}
\newcommand{\jw}{\mathsf{jw}}
\newcommand{\upr}{\overline{\mathsf{U}}}
\newcommand{\dor}{\overline{\mathsf{D}}}
\newcommand{\zupr}{\mathsf{U}}
\newcommand{\zdor}{\mathsf{D}}
\newcommand{\ssupr}{\mathsf{u}}
\newcommand{\ssdor}{\mathsf{d}}
\newcommand{\sslo}{\mathsf{x}}

\newcommand{\alg}[1]{\mathscr{#1}}

\newcommand{\cellposet}{\Lambda}
\newcommand{\cellindices}[1]{M(#1)}
\newcommand{\cellindex}[1]{\mathfrak{#1}}
\newcommand{\cellbasissym}{c}
\newcommand{\cellbasis}[3][\lambda]{\cellbasissym_{\cellindex{#2},\cellindex{#3}}^{#1}}
\newcommand{\cellhalf}[2][\lambda]{\cellbasissym_{\cellindex{#2}}^{#1}}
\newcommand{\rr}[4][\lambda]{r_{#2}^\lambda(\cellindex{#3},\cellindex{#4})}

\newcommand{\tlcellposet}[1][n]{\Lambda_{#1}}
\newcommand{\tlcellindices}[2][n]{M_{#1}(#2)}

\newcommand{\basis}{\mathscr{B}}

\newcommand{\vv}[2][\lambda]{v_{#1}}
\newcommand{\NN}[2]{N_{#1}}
\newcommand{\yy}[2][n]{y_{#1}^{#2}}
\newcommand{\zz}[2][n]{z_{#1}^{#2}}
\newcommand{\yz}[2][n]{e_{#1}^{#2}}

\newcommand{\cellmod}[1]{\mathcal{#1}}
\newcommand{\cmod}[2][n]{\cellmod{S}(#1,#2)}

\newcommand{\tllinewidth}{1.25}
\tikzset{
centered/.style={baseline={([yshift=-0.5ex]current bounding box.center)}},
mylabels/.style={font=\tiny},
minlabels/.style={font=\fontsize{4}{0}\selectfont},
st/.style={line width=\tllinewidth,color=black},
jw/.style={line width=\tllinewidth,color=black,fill=gray!30,rounded corners},
jwlp/.style={line width=\tllinewidth,color=black,fill=pink,rounded corners},
jwlpz/.style={line width=\tllinewidth,color=black,fill=Goldenrod,rounded corners
},
morw/.style={line width=\tllinewidth,color=black,rounded corners},
morg/.style={line width=\tllinewidth,color=black,fill=white,rounded corners},
empty/.style={color=black,densely dotted,rounded corners},
}
\newcommand{\tlcap}[3]{\draw[st] (#1) to[out=90,in=180] (#2) to[out=0,in=90] (#3);}
\newcommand{\tlcup}[3]{\draw[st] (#1) to[out=270,in=180] (#2) to[out=0,in=270] (#3);}

\newcommand{\jantzenopacity}{0.9}
\colorlet{top}{red!95!black}
\colorlet{bot}{blue!95!white}
\colorlet{colzero}{bot!100!top}
\colorlet{colone}{bot!80!top}
\colorlet{coltwo}{bot!60!top}
\colorlet{colthree}{bot!40!top}
\colorlet{colfour}{bot!20!top}
\colorlet{colfive}{bot!00!top}
\newcommand{\colzero}{colzero}
\newcommand{\colone}{colone}
\newcommand{\coltwo}{coltwo}
\newcommand{\colthree}{colthree}
\newcommand{\colfour}{colfour}
\newcommand{\colfive}{colfive}

\title{Cell modules for the Temperley--Lieb algebra in mixed characteristic}
\author[Stuart Martin]{Stuart Martin}
\address[Stuart Martin]{
    Department of Pure Mathematics and Mathematical Statistics;
    Centre for Mathematical Sciences;
    Wilberforce Road;
    Cambridge;
    CB3 0WB;
    United Kingdom
}
\email{sm137@cam.ac.uk}
\author[Charles Senécal]{Charles Senécal}
\address[Charles Senécal]{
    Department of Pure Mathematics and Mathematical Statistics;
    Centre for Mathematical Sciences;
    Wilberforce Road;
    Cambridge;
    CB3 0WB;
    United Kingdom
}
\email{cs2228@cam.ac.uk}
\author[Robert A. Spencer]{Robert A. Spencer}
\address[Robert A. Spencer]{}
\email{maths@robertandrewspencer.com}
\date{\today}

\begin{document}

\begin{abstract}
We study the representation theory of the Temperley--Lieb algebra $\tlnkd$ in mixed characteristic, i.e. over an arbitrary field $\kk$ of characteristic $\p$ and where $\delta$ satisfies some minimal polynomial $m_\delta$. In particular, we completely describe the submodule structure of cell modules for $\TL_n$ and give their Alperin diagrams. The proof is entirely diagrammatic and does not appeal to the role of $\TL_n$ as the endomorphism algebra of tensor powers of the fundamental representation of $\uqsl$. We also investigate two-dimensional Jantzen-like filtrations of the cell modules related to the mixed characteristic.
\end{abstract}

\maketitle

\tableofcontents

\section*{Introduction}
\input{intro.tex}

\section{Quantum numbers and modular combinatorics}\label{section:qnum}
\input{qnumbers.tex}

\section{The Temperley--Lieb category}\label{section:tl}
\input{temperley-lieb.tex}

\section{The submodule structure}\label{section:diag}
\input{diagrammatic_new.tex}

\section{Jantzen filtrations}\label{section:filt}
\input{jantzen.tex}

\section*{Acknowledgements}
\input{acknowledgements.tex}

\bibliographystyle{plain}
\bibliography{references}

\end{document}

%% file: intro.tex
The Temperley--Lieb algebras $\TL_n(\delta)$, dating back to the work of Temperley and Lieb on lattice models in statistical physics \cite{Temperley1971}, have, since their introduction, played a role in many different contexts, both in mathematical physics and in mathematics. Notably, they appear in Jones's work on polynomial knot invariants \cite{Jones1985}, serve as a canonical example of Graham and Lehrer's theory of cellular algebras \cite{Graham1996}, and, being a quotient of the Hecke algebras of type $A_n$, also play a central role in the representation theory of the Lie algebra $\mathfrak{sl}_2$ and its associated quantum group through (quantum) Schur--Weyl duality.

More recently, the Temperley--Lieb algebras have found applications to Soergel bimodule theory in characteristic zero, where the (two-coloured) Temperley--Lieb category appears as degree-zero morphisms between colour-alternating objects in the diagrammatic Hecke category for dihedral groups \cite{Elias2016}. Here, the parameter $\delta$ is chosen in a way that agrees with the choice of the realisation and depends on the order of the dihedral group $D_{2m}$ in the finite case. The well-known Jones--Wenzl projectors play a critical role, as they yield idempotents in the Hecke category, whose images correspond to indecomposable Soergel bimodules. An analogous result holds over fields of positive characteristic and the corresponding idempotents are the $p$-Jones--Wenzl projectors of \cite{Burrull2019}, whose coefficients can be computed from the $p$-canonical basis in type $\tilde{A}_1$.

Although the Temperley--Lieb algebras' representation theory in characteristic zero has been well understood for a long time \cite{Goodman1993,Ridout2014,Westbury1995} and is mostly derived from self-contained, diagrammatic and combinatorial techniques, a lot of the recent endeavours to understand its characteristic $\p$ representation theory \cite{Cox2003,Andersen2019,Burrull2019,Tubbenhauer2021,Sutton2023} have explicitly relied on $\TL_n$'s role as $\End_{\uqsl}(V^{\otimes n})$ through Schur--Weyl duality in order to apply the well-known theory of Weyl and tilting modules for $\uqsl$ (\cite{Andersen1991,Ringel1991,Donkin1998,Andersen2018}). 

Of central importance throughout the study of the representation theory of Temperley--Lieb algebras, and indeed of any cellular algebra, is the investigation of their cell modules. These modules carry a complete set of irreducible representations as their heads, and counting their composition factors also determines the decomposition numbers and the blocks of the algebra. More generally, structural information about cell modules gives a lot of information about the representation category of the algebra.

Lately, there have been some efforts to build the whole characteristic $\p$ theory from first principles, starting from the generators and relations definition of the Temperley--Lieb algebras.  Over a commutative ring $\kk$ with distinguished element $\delta$, the algebra $\tlnkd$ has generators $u_1,u_2,\ldots,u_{n-1}$ and unit $\id_n$, with relations
\begin{equation*}
\begin{aligned}
&u_i^2=\delta u_i && \mbox{for } i\geq 1,\\
&u_iu_{i\pm 1}u_i=u_i && \mbox{for } i,i\pm 1\in \{1,\ldots,n-1\},\\
&u_iu_j=u_ju_i && \mbox{if } \abs{i-j}\geq 2.
\end{aligned}
\end{equation*}
These relations are in direct correspondence with the usual Temperley--Lieb diagrams,
\begin{equation*}
u_i\mapsto\;
\begin{tikzpicture}[scale=0.25,centered]
\draw[st] (0,0) -- (0,3);
\draw[st] (3,0) -- (3,3);
\tlcap{5,0}{6,1}{7,0};
\tlcup{5,3}{6,2}{7,3};
\draw[st] (9,0) -- (9,3);
\draw[st] (12,0) -- (12,3);
\begin{scope}[overlay]
\end{scope}
\node at (1.5,1.5) {$\cdots$};
\node at (10.5,1.5) {$\cdots$};
\end{tikzpicture}\;,\quad
\id_n\mapsto\;
\begin{tikzpicture}[scale=0.25,centered]
\draw[st] (0,0) -- (0,3);
\draw[st] (2,0) -- (2,3);
\draw[st] (5,0) -- (5,3);
\draw[st] (7,0) -- (7,3);
\node at (3.5,1.5) {$\cdots$};
\end{tikzpicture}\;,
\end{equation*}
where the cup and the cap of $u_i$ link the sites $i$ and $i+1$. The Temperley--Lieb algebras can thus be studied purely diagrammatically, without resorting to Schur--Weyl duality and the representation theory of $\uqsl$. This allows for some more generality as it makes sense to study its representation theory over any pointed ring $(\kk,\delta)$, not requiring, in particular, $\delta$ to be of the form $q+q^{-1}$. This is the approach taken in \cite{Spencer2023} and in \cite{Martin2022}, where the decomposition numbers for $\TL_n$ in mixed characteristic and the dimensions of the corresponding simple modules are computed, and where the mixed characteristic analogues of the $\p$-Jones--Wenzl idempotents are constructed, respectively.

This paper closes the loop by completely describing the submodule structure of the cell modules $\ctlmod{m}$ of the Temperley--Lieb algebra $\TL_n$ in mixed characteristic, which is done by giving their (unambiguous) Alperin diagrams \cite{Alperin1980}. The proof is entirely diagrammatic and highlights the role played by the mixed Jones--Wenzl idempotents in the cellular structure of the Temperley--Lieb algebras.

These results also allow us to study two-dimensional Jantzen-like filtrations of the cell modules, obtained from the corresponding intersection forms, the computation of which can be reduced to taking partial traces of some mixed Jones--Wenzl idempotents using our results about the structure of cell modules. This has ties to Kazhdan--Lusztig theory, where there has been much interest into Jantzen filtrations (and the related Jantzen conjectures), as there is a deep interplay between these filtrations and the Kazhdan--Lusztig conjectures~\cite{Gabber1981}. Recently, counterexamples to the semi-semisimplicity of Jantzen filtrations for some non-degenerate deformation directions have been studied \cite{Williamson2016} and, more generally, non-diagonalisability of the canonical bilinear forms on Weyl modules (or on cell modules for the corresponding endomorphism algebra) seems to be linked with the failure of the Kazhdan--Lusztig conjecture. The results of the present paper may thus be useful to start the study of forms for cell modules over two-dimensional rings, the Jantzen filtrations of Weyl modules for quantum groups, and the consequences about the failure of the Kazhdan--Lusztig conjecture in certain cases.

\vspace{1em}\noindent
The paper is organised as follows. Section~\ref{section:qnum} recalls the definition of quantum numbers and explores some of their properties over prime characteristics. In Section~\ref{section:tl}, we recall the definitions of the Temperley--Lieb category and of cellular algebras, we recall the construction of several cellular bases of $\TL_n$ using the light-leaves formalism of \cite{Elias2015}, and we give the definitions of the classical and mixed Jones--Wenzl idempotents. In Section~\ref{section:diag}, we produce an entirely diagrammatic and self-contained proof of the submodule structure of the cell modules of $\TL_n$ in complete generality for the choice of $\kk$ and $\delta$. Finally, Section~\ref{section:filt} describes the two-dimensional Jantzen filtrations of cell modules with respect to a certain grid of ideals depending on the mixed characteristic.

%% file: qnumbers.tex
The representation theory of the Templerley--Lieb algebras is intimately connected to a $q$-analogue of the integers known as ``quantum numbers''. These polynomials are also variously known in the literature under the names ``Vieta-Fibonacci polynomials'' \cite{Horadam2002}, ``compressed Chebyshev polynomials'' \cite{Keri2002} and ``modified Chebyshev polynomials'' \cite{Witula2006}.

\begin{deff}\label{defqn}
The quantum numbers $\qnn{n}$, defined for any $n\in\Z$, are polynomials in $\Z[\delta]$. They are defined using the recurrence relation $\delta\qnn{n}=\qnn{n-1}+\qnn{n+1}$ subject to initial conditions $\qnn{0}=0$ and $\qnn{1}=1$.
\end{deff}

Under the identification $\delta\mapsto q+q^{-1}$ which induces a bijection between $\Z[\delta]$ and the ring of symmetric Laurent polynomials in $q$, the quantum numbers defined as above are equivalent to
\begin{equation}\label{eq:quantum_as_q}
\qn{n}=q^{n-1}+q^{n-3}+\ldots +q^{-n+1}=\frac{q^n-q^{-n}}{q-q^{-1}}.
\end{equation}
Note that this final form is expressed as an element of $\Q(q)$, thought it lives in the image of $\Z[q+q^{-1}] \hookrightarrow \Q(q)$.

Using this presentation, it is clear that under the specialisation $q\mapsto 1$ (or, equivalently, $\delta\mapsto 2$), the quantum numbers give back the integers: $\qn[q=1]{n}=\qn[\delta=2]{n} = n$. Even though the formulation in terms of $q$ and $q^{-1}$ is often useful, it is important to keep in mind that quantum numbers are polynomials in $\delta$ and thus give well-defined specialisations no matter the value of $\delta$ (or $q$).

\subsection{Factorisation of quantum numbers}

The following lemma follows from direct calculations using equation \eqref{eq:quantum_as_q}.

\begin{lem}\label{qnum} For any $\ell, m,n \in \Z$,
\begin{enumerate}
\item $\qn{n\ell}=\qn{\ell}\qn[q^\ell]{n}$, and hence
\item if $m\mid n$, then $\qn{m}\mid \qn{n}$.
\end{enumerate}
\end{lem}

These statements have their equivalent in $\delta$ notation, meaning that $m\mid n$ implies $\qnd{m}\mid \qnd{n}$ and that we have $\qnd{n\ell}=\qnd{\ell}\qn[\qnn{\ell+1}-\qnn{\ell-1}]{n}$, where the subscript indicates the value of $\delta$ to substitute in the polynomial.

As alluded to at the beginning of this section, the family of quantum numbers is closely related to the family of Chebyshev polynomials in the following way.
\begin{lem}
  Let $U_n$ denote the family of Chebyshev polynomials of the second kind, defined by the formula $U_{n+1}(\cos x)=(\sin nx )/(\sin x )$. Then $\qnd{n} = U_{n+1}(\delta/2)$.
\end{lem}
\begin{proof}
Using the classical recursion relation of the polynomials $U_n$, one can check that that the renormalisation $U_{n+1}(\delta/2)$ satisfies the same recursion relation and initial conditions as in Definition~\ref{defqn}.
\end{proof}

The factorisation of Chebyshev polynomials, or equivalently of quantum numbers, is intimately linked to cyclotomic polynomials: using equation~\eqref{eq:quantum_as_q}, one can write
\begin{equation*}
\qn{n}=\frac{q^n-q^{-n}}{q-q^{-1}}=q^{-n+1}\frac{q^{2n}-1}{q^2-1}=q^{-n+1}\prod_{\substack{k\mid 2n\\k\geq 3}} \Phi_k(q),
\end{equation*}
where $\Phi_k(q)$ is the $k^{th}$ cyclotomic polynomial. Since $\Phi_k$ is irreducible, symmetric, and of degree $\varphi(k)$, the polynomial $q^{-\frac{\varphi(k)}{2}}\Phi_k(q)$ can be written as an irreducible polynomial in $q+q^{-1}$ that we denote by
\begin{equation*}
\psi_k(q+q^{-1}):=q^{-\frac{\varphi(k)}{2}}\Phi_k(q).
\end{equation*}
Comparing degrees gives the following.

\begin{lem}\label{chebyfac}
The polynomials $\psi_k$ are irreducible over $\Q$ and have integer coefficients. Moreover, the quantum number $\qnd{n}$ factors as
\begin{equation*}
\qnd{n}=\prod_{\substack{k\mid 2n\\k\geq 3}}\psi_k(\delta).
\end{equation*}
\end{lem}

Of particular interest for later sections will be the reductions modulo $p$ of quantum numbers, and hence the following result will be useful.

\begin{lem}\label{lem:psioddeven}
If $\ell$ is odd then $\psi_\ell(x) = \psi_{2\ell}(x)$ when reduced mod $2$, but they are distinct in all other prime characteristics.
\end{lem}
\begin{proof}
  If $\ell$ is odd then $\phi(\ell) = \phi(2\ell)$ and $\Phi_\ell(x) = \Phi_{2\ell}(-x)$ so $\psi_\ell(x)$ and $\psi_{2\ell}(x)$ are equal modulo $2$. On the other hand for odd $p$, the cyclotomic polynomial has no repeated roots modulo $p$ so $\psi_n(x)$ is irreducible modulo $p$.
\end{proof}

Fix a field or domain $\kk$ of characteristic $\p$ and let $\overline{\delta}\in\kk$. The elements of $\Z[\delta]$ have a canonical image in the pointed ring $(\kk,\overline{\delta})$, so we may consider the quantum numbers $\qnn{n}$ as elements in $\kk$. Since $m\mid n$ implies $\qnn{m}\mid \qnn{n}$, if some quantum number vanishes in $\kk$, then there exists a smallest positive $\ell$ such that $\qnn{\ell}=0$ and if $[n] = 0$ then $\ell \mid n$. If no quantum number vanishes in $\kk$, we say $\ell=\infty$. The pair $(\ell,\p)$ is called the \textbf{mixed characteristic} of the pair $(\kk,\delta)$, and our main case of interest is the \textbf{strictly mixed characteristic} case, when both $p$ and $\ell$ are positive and finite.

\begin{rem}
While it is tempting to assume that any pair of prime $p$ and number $\ell$ defines a valid mixed characteristic, this is not the case.

Indeed, if $p \mid \ell$ and $p \neq \ell$, then it is not possible for $\qnn{\ell}$ to be the smallest vanishing quantum number. Indeed, suppose $\qnn{\ell}=0$ where $\ell = \ell'p$ for $\ell'\neq 1$. Then $\qn{\ell} = \qn{p}\qn[q^p]{\ell'}$. However, $\qn[q^p]{\ell'} = (\qn{\ell'})^p$ by the Freshman Lemma so that $0 = \qnn{\ell'}^{p}\qnn{p}$ and either $\qnn{\ell'}$ or $\qnn{p}$ must vanish, a contradiction on the minimality of $\ell$.
\end{rem}

Let $m_\delta$ be an irreducible factor of $\qnn{\ell}$ in $\Z[\delta]$ that does not divide any $\qnn{n}$ for $1\leq n < \ell$.
\begin{lem}\label{lem:irredfac}
  These irreducible factors are given by
  \begin{equation}
  m_\delta = 
      \begin{cases}
      \psi_{2n}(\delta)&n\text{ is even}\\
      \psi_n(\pm\delta)&n \text{ is odd}
      \end{cases}
  \end{equation}
\end{lem}
\begin{proof}
This follows directly from Lemmata \ref{chebyfac} and \ref{lem:psioddeven}.
\end{proof}

The reduction modulo $\p$ of $m_\delta$ which we will call $\overline{m_\delta}$ is itself irreducible by Lemma \ref{lem:psioddeven} and is a minimal polynomial for $\overline{\delta}\in\kk$. The ideal $\mathfrak{m}:=(\p,m_\delta)$ is maximal in $\Z[\delta]$, and there is a canonical specialisation morphism $\Zm\to\kk$ sending $\delta$ to $\overline{\delta}$.

\begin{lem}\label{lem:irreduciblemd}
  Let $(\ell,\p)$ define a valid mixed characteristic. If $p = 2$ and $\ell\geq 3$ then, as elements of the ring $\Z[\delta]/(\p)$, the multiplicity of $m_\delta$ in $\qnn{\ell}$ is 2. Otherwise, if $p$ is odd or if $\p=\ell=2$, it appears with multiplicity 1. Moreover, for any $\p$ and any $\ell$, it appears with multiplicity 1 in the ring $\Z[\delta]/(\p^i)$ for any $i\geq 2$.
\end{lem}
\begin{proof}
Choosing $m_\delta$ as in Lemma~\ref{lem:irredfac}, the result follows from the fact that $\psi_\ell(\delta)=\psi_{\ell}(-\delta)\pmod{\p^i}$ if and only if $\p=2$ and $i=1$, and that all $\psi_n$ are distinct modulo $\p^i$ for any other choice of $\p,i$.
\end{proof}

\subsection{$(\ell, p)$-digits and the family tree}
In much the same way that the $p$-adic digits of numbers control the $p$-modular representation theory of certain objects, we find that in mixed characteristic the ``$(\ell, p)$-digits'' of certain numbers are central.

\begin{deff}\label{def:lpexp}
For $n\in\Z_{\geq 0}$, define its \textbf{$(\ell,\p)$-adic expansion} by writing $n+1=\sum_{i=0}^{\infty} n_i\p^{(i)}$, where $\p^{(0)}:=1$ and $\p^{(i)}:=\ell\p^{i-1}$ for $i\geq 1$, and the coefficients are integers satisfying $0\leq n_0<\ell$ and $0\leq n_i<\p$ for $i\geq 1$. If $k$ is the largest integer such that $n_k\neq 0$, this expansion will be written as
$$ n+1=[n_k,n_{k-1},\ldots,n_0]_{\ell,\p}.$$
\end{deff}

\begin{ex}\label{ex:lpadic}
Let $\ell=5$ and $\p=3$, as is the case when $$\delta = 1 + \sqrt{-1} \in \F_9 \simeq \F_3[\delta]/(\delta^2+\delta-1).$$
Note that $\qnn{5} = \delta^4 - 3\delta^2 + 1 = (\delta^2+\delta-1)(\delta^2-\delta-1)$ and that $(\delta^2+\delta-1)$ is irreducible modulo 3 as per Lemma \ref{lem:irreduciblemd}. The $(5,3)$-adic extension associated to the integer 685 is given by
$$685+1=686=1\cdot 5\cdot 3^4+2\cdot 5\cdot 3^3+0\cdot 5\cdot 3^2+0\cdot5\cdot 3+ 2\cdot 5 +1=[1,2,0,0,2,1]_{5,3}.$$
\end{ex}

\begin{deff}\label{def:ancestors}
If $s$ is the smallest integer such that $n_s\neq 0$ and $s<k$, then define the \textbf{mother} $\mo{n}$ of $n$ as
\begin{equation*}
\mo{n}+1=[n_k,\ldots,n_{s+1},0,\ldots,0]_{\ell,\p}.
\end{equation*}
If $s=k$, then $n$ has no mother and is called \textbf{Eve}. The elements of the set $\{\mo{n},\mo[2]{n},\ldots\}$ are called \textbf{ancestors} of $n$. Given $n,s\in\Z_{\geq 0}$, write $a_{n,s}$ for the youngest ancestor of $n$ having its $s^{th}$ $(\ell,\p)$-digit equal to 0. We use the convention $a_{n,-1}=0$.
\end{deff}

\addtocounter{cor}{-2}
\begin{ex}[continued]
Keeping the same setup as above, the integer 685 has mother
$$\mo{685}=[1,2,0,0,2,0]_{5,3}-1=685-1=684.$$
Its complete list of ancestors is given by $\{684,674,404\}$, with $404$ being Eve. Note also that $a_{685,0}=684$, $a_{685,1}=a_{685,2}=a_{685,3}=674$, and $a_{685,4}=404$.
\end{ex}
\addtocounter{cor}{1}

\begin{deff}\label{def:support}
The \textbf{support} of the integer $n$ is the set
$$\supp{n}=\{n_k\p^{(k)}\pm n_{k-1}\p^{(k-1)}\pm\cdots\pm n_1\p^{(1)}\pm n_0\p^{(0)}-1\}.$$
\end{deff}

\addtocounter{cor}{-3}
\begin{ex}[continued]
Continuing with the same setup, we have
$$\supp{685}=\{685,683,665,663,145,143,125,123\}.$$
\end{ex}
\addtocounter{cor}{2}

\begin{figure}
\caption{Values of $\supp{n}$ for $\ell=5$ and $\p=3$. The $y$-axis is $n$ and a square at $(x,y)$ is coloured iff $x \in \supp{y}$.}
\includegraphics[scale=0.5]{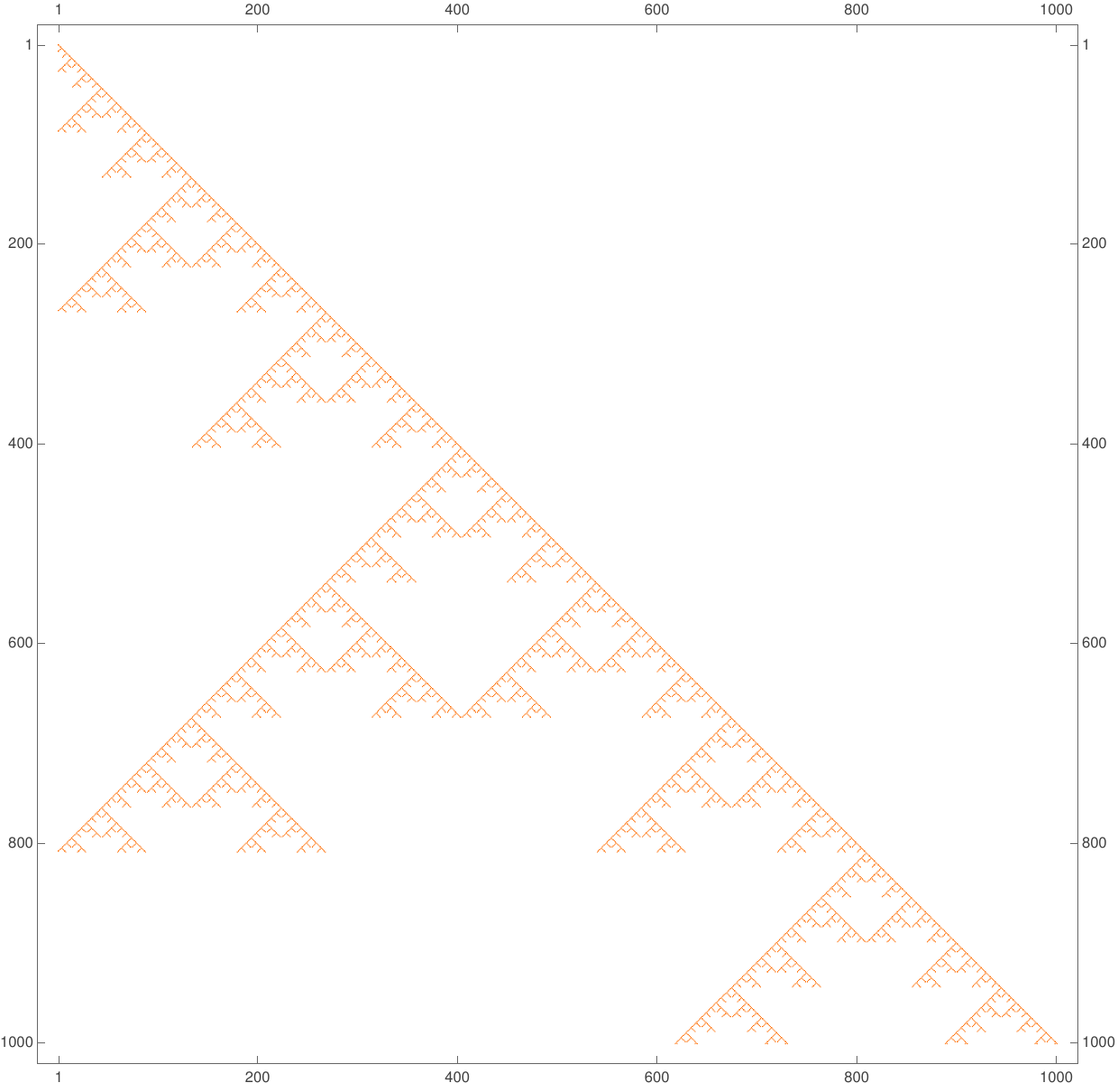}
\end{figure}

The following definition is taken from \cite{Sutton2023}, with a slight difference in conventions causing shifts by one in the integers considered and their reflections.

\begin{deff}[{\cite[Definition 2.7]{Sutton2023}}]\label{defadm}
Let $S\subset\Z_{\geq 0}$ be a finite set and write it as a disjoint union of \textbf{stretches} $S=\sqcup_i S_i$, where a stretch $S_i$ is a maximal subset of consecutive integers. Write
$$n+1=[n_k,n_{k-1},\ldots,n_0]_{\ell,\p}.$$
The set $S$ is called \textbf{down-admissible} for $n$ if
\begin{enumerate}
\item $n_{\min(S_i)}\neq 0$ for all $i$,
\item if $s\in S$ and $n_{s+1}=0$, then $s+1\in S$.
\end{enumerate}
For a down-admissible set $S$ for $n$, define its \textbf{downward reflection along $S$} as
\begin{equation*}
n[S]=[n_k,\epsilon_{k-1}n_{k-1},\ldots,\epsilon_0 n_0]_{\ell,\p}-1,\quad \epsilon_i=\begin{cases}
1 & \mbox{if } i\notin S,\\
-1 & \mbox{if } i\in S.
\end{cases}
\end{equation*}
The set $S$ is called \textbf{up-admissible} for $n$ if
\begin{enumerate}
\item $n_{\min(S_i)}\neq 0$ for all $i$,
\item if $s\in S$ and $n_{s+1}=\p-1$, then $s+1\in S$.
\end{enumerate}
For an up-admissible set $S$ for $n$, define its \textbf{upward reflection along $S$} as
\begin{equation*}
n(S)=[n'_{k(S)},n'_{k(S)-1},\ldots,n'_0]_{\ell,\p}-1,\quad n'_i=\begin{cases}
n_i & \mbox{if } i\notin S,i-1\notin S,\\
n_i+2 & \mbox{if } i\notin S,i-1\in S,\\
-n_i & \mbox{if } i\in S,
\end{cases}
\end{equation*}
where $n_j$ is taken to be $0$ if $j>k$ and $k(S)$ is the largest integer such that $n'_{k(S)}\neq 0$.
\end{deff}

\begin{rem} For every $n\in\Z_{\geq 0}$,
\begin{enumerate}
\item the set $S$ is down-admissible for $n$ if and only if it is up-admissible for $n[S]$, in which case $n[S](S)=n$. Similarly, $n(S)[S]=n$, and
\item the set $\supp{n}$ can be reinterpreted in terms of down-admissible sets as $\supp{n}=\{n[S]\mid S\mbox{ down-admissible for }n\}$, and
\item there is a finite number of down-admissible sets for $n$ but an infinite number of up-admissible ones.
\end{enumerate}
\end{rem}

\begin{ex} Let $(\ell,\p)=(5,3)$ as in Example~\ref{ex:lpadic}. Let $n=123$. Then, $$n+1=124=[2,2,0,4]_{5,3}.$$ The only up-admissible sets for $n$ of size 1 are $\{0\}$ and $\{3\}$, and the only ones of size 2 are $\{0,3\}$, $\{2,3\}$, and $\{3,4\}$. Any other up-admissible set for $n$ is obtained by adding an arbitrary number of consecutive integers after 3 to the up-admissible sets $\{3\},\{0,3\},\{2,3\},\{0,2,3\}$ or $\{0,1,2,3\}$.

\begin{figure}[h]\label{fig:bars}
  \begin{tikzpicture}[scale=0.35]
    \node[anchor=east, align=right] at (1.7,-1) {$[0,0,0,2,2,0,4]_{5,3}$};
    \draw[line width=2pt, line cap=round,color=black] (-1,-0) -- (0,-0);
    
    \foreach \i in {0,...,3} {\begin{scope}[shift={(0,\i*2.5)}]
    \pgfmathsetmacro{\t}{100-\i*25} 

    \draw[line width=2pt, line cap=round,color=black!\t] (-4-\i,0.5) -- (-3,0.5);
    
    \draw[line width=2pt, line cap=round,color=black!\t] (-4-\i,1) -- (-3,1);
    \draw[line width=2pt, line cap=round,color=black!\t] (-1,1) -- (0,1);
    
    \draw[line width=2pt, line cap=round,color=black!\t] (-4-\i,1.5) -- (-2,1.5);
    
    \draw[line width=2pt, line cap=round,color=black!\t] (-4-\i,2) -- (-2,2);
    \draw[line width=2pt, line cap=round,color=black!\t] (-1,2) -- (0,2);
    
    \draw[line width=2pt, line cap=round,color=black!\t] (-4-\i,2.5) -- (0,2.5);
    \end{scope}}
  
  \begin{scope}[shift={(8,8)}]
    \node[anchor=east, align=right] at (1.65,1) {$[1,2,0,0,2,1]_{5,3}$};
    \draw[line width=2pt, line cap=round,color=black] (-1,-0) -- (0,-0);

    \draw[line width=2pt, line cap=round,color=black] (-5,-0.5) -- (-4,-0.5);
    
    \draw[line width=2pt, line cap=round,color=black] (-5,-1) -- (-4,-1);
    \draw[line width=2pt, line cap=round,color=black!] (-1,-1) -- (0,-1);
    
    \draw[line width=2pt, line cap=round,color=black] (-4,-1.5) -- (-1,-1.5);
    
    \draw[line width=2pt, line cap=round,color=black] (-4,-2) -- (-0,-2);

    \draw[line width=2pt, line cap=round,color=black] (-5,-2.5) -- (-1,-2.5);

    \draw[line width=2pt, line cap=round,color=black] (-5,-3) -- (-0,-3);
  \end{scope}
  \end{tikzpicture}
  \caption{The down-admissible sets for $685=[1,2,0,0,2,1]_{5,3}-1$ (right) and some of the up-admissible sets for $123 = [2,2,0,4]_{5,3}-1$ (left). The bars are aligned at the set ${0,1,2,3,4}$ which is common to both.}
\end{figure}
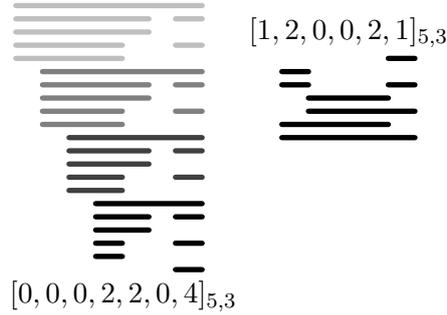

Consider also $n'=685$. Then, $n'+1=686=[1,2,0,0,2,1]_{5,3}$ and the down-admissible sets for $n'$ are $\{0\}, \{4\},
\{0,4\},
\{1,2,3\},
\{0,1,2,3\},\{1,2,3,4\},$ and $\{0,1,2,3,4\}.$
These are illustrated in figure \ref{fig:bars}.

Let $S=\{0,1,2,3,4\}$. Then, $n(S)=685=n'$, and $n'[S]=123=n$, so that $n(S)[S]=n$ as expected.

Finally,
\begin{gather*}
n'[\{0\}]=683, n'[\{4\}]=145,
n'[\{0,4\}]=143,
n'[\{1,2,3\}]=665,\\
n'[\{0,1,2,3\}]=663,n'[\{1,2,3,4\}]=125,\mbox{ and }n'[\{0,1,2,3,4\}]=123,
\end{gather*}
and all these values agree with the elements of $\supp{685}$ computed in Example~\ref{ex:lpadic}.
\end{ex}

For every up- or down-admissible set $S$, there is a unique finest partition of $S$ into up- or down-admissible stretches, called \textbf{minimal up- or down-admissible stretches}. It is not too hard to see that, if
$$n+1=[n_k,n_{k-1},\ldots,n_0]_{\ell,\p},$$
then a stretch $\{i,i-1,\ldots,j+1,j\}$ is minimal down-admissible if and only if
$$(n_{i+1},n_{i},\ldots,n_{j+1},n_j)=(n_{i+1},0,\ldots,0,n_j),$$
with $n_{i+1},n_j\neq 0$, and is minimal up-admissible if and only if
$$(n_{i+1},n_{i},\ldots,n_{j+1},n_j)=(n_{i+1},\p-1,\ldots,\p-1,n_j),$$
with $n_{i+1}\neq \p-1,n_j\neq 0$.

The following lemma exhibits a property of up-admissible sets that will be useful for inductive arguments.
\begin{lem}\label{indadmsets}
Suppose $S,S'$ are up-admissible sets for an integer $m\in\Z_{\geq 0}$ such that $S\subset S'$. Then there exist up-admissible sets
\begin{equation*}
S=S_1\subset S_2\subset \cdots\subset S_k=S'
\end{equation*}
such that $\abs{S_i}=\abs{S_{i-1}}+1$ for all $1<i\leq k$.
\end{lem}
\begin{proof}
The proof will proceed by induction on $k:=\abs{S'\setminus S}$, the base case $k=1$ being trivial.

Suppose $\abs{S'\setminus S}=k$. Then write $S'\setminus S=\sqcup_{i=1}^l T_i$ as a disjoint union of stretches with $T_1<T_2<\cdots <T_l$. With $m+1 = [m_t, m_{t-1}, \ldots, m_0]_{\ell,p}$, let $A:=\{a\in T_l\mid m_a\neq 0\}$. If $A\neq\varnothing$, set $\alpha=\max A$. Otherwise, set $\alpha=\min T_l$. Then $S'':=S\cup\{\alpha\}$ is up-admissible, $S\subset S''\subset S'$, $\abs{S''}=\abs{S}+1$, and $\abs{S'\setminus S''}=k-1$. The desired sequence of up-admissible sets then follows by applying the induction hypothesis to $S''$ and $S'$.
\end{proof}

%% file: temperley-lieb.tex
We define the Temperley--Lieb algebras and category and start discussing their representation theory.

\subsection{Diagrams}
The most versatile presentation of Temperley-Lieb objects is through planar pairings known as ``diagrams''. 

\begin{deff}
Let $m,n\in\Z_{\geq 0}$ be of the same parity. An $(m,n)$-diagram is a picture consisting of two rows of sites aligned horizontally, $m$ on the bottom and $n$ on top. Those sites are paired by non-intersecting lines living inside the strip between the two rows. These pictures are considered up to isotopy fixing the sites.
\end{deff}

For example, here are a $(5,5)$-diagram, a $(5,3)$-diagram, a $(2,6)$-diagram and a $(4,0)$-diagram :
\begin{equation*}
d_1=\;
\begin{tikzpicture}[scale=0.25,mylabels,centered]
\draw[st] (0,0) -- (0,3);
\draw[st] (2,0) to[out=90,in=180] (3,1) to[out=0,in=90] (4,0);
\draw[st] (6,0)..controls{(6,1.7) and (2,1.3)}..(2,3);
\draw[st] (4,3) to[out=270,in=180] (5,2) to[out=0,in=270] (6,3);
\draw[st] (8,0) -- (8,3);
\end{tikzpicture}
\;,\quad
d_2=\;
\begin{tikzpicture}[scale=0.25,mylabels,centered]
\draw[st] (0,0)..controls{(0,2.2) and (6,1.8)}..(6,4);
\draw[st] (2,0) to[out=90,in=180] (5,2) to[out=0,in=90] (8,0);
\draw[st] (4,0) to[out=90,in=180] (5,1) to[out=0,in=90] (6,0);
\draw[st] (2,4) to[out=270,in=180] (3,3) to[out=0,in=270] (4,4);
\end{tikzpicture}
\;,\quad
d_3=\;
\begin{tikzpicture}[scale=0.25,mylabels,centered]
\draw[st] (4,0)..controls{(4,1.7) and (0,1.3)}..(0,3);
\draw[st] (6,0) -- (6,3);
\draw[st] (2,3) to[out=270,in=180] (3,2) to[out=0,in=270] (4,3);
\draw[st] (8,3) to[out=270,in=180] (9,2) to[out=0,in=270] (10,3);
\end{tikzpicture}
\;,\quad
d_4=\;
\begin{tikzpicture}[scale=0.25,mylabels,centered]
\draw[st] (2,0) to[out=90,in=180] (5,2) to[out=0,in=90] (8,0);
\draw[st] (4,0) to[out=90,in=180] (5,1) to[out=0,in=90] (6,0);
\end{tikzpicture}
\;.
\end{equation*}

A line that pairs a bottom site to a top site is called a \textbf{propagating line}. An $(m,n)$-diagram where the number of propagating lines is maximal is called \textbf{epic} if $m\geq n$, and \textbf{monic} if $m\leq n$. More generally, the number of propagating lines in a diagram is called its \textbf{through degree}. For example, the diagram $d_1$ above has through degree 3, written $\td(d_1)=3$. The through degree of a linear combination of diagrams is defined to be the highest through degree of its diagrams.

\begin{deff}
The Temperley--Lieb category $\tlcat^\kk(\delta)$ is the monoidal category where
\begin{itemize}
\item objects are the non-negative integers $\Z_{\geq 0}$;
\item for $m,n\in\Z_{\geq 0}$, $\Hom_{\tlcat^\kk(\delta)}(m,n)$ is the free $\kk$-module having basis all $(m,n)$-diagrams;
\item composition of morphisms is done on basis elements by vertically concatenating the two diagrams and replacing every closed loop by a factor of $\delta$;
\item tensor product on objects is given by $m\otimes n= m+n$ and on morphisms is given by horizontal concatenation.
\end{itemize}
\end{deff}

When the context is clear, we sometimes drop the $\kk$ and the $\delta$ from the notation and simply write $\tlcat$. The definition of composition of morphisms begs for an example: taking the diagrams $d_1\in\Hom_\tlcat(5,5)$ and $d_2\in\Hom_\tlcat(5,3)$ from the previous example, we get
\begin{equation*}
d_2\circ d_1=\;
\begin{tikzpicture}[scale=0.25,mylabels,centered]
\draw[st] (0,0)..controls{(0,2.2) and (6,1.8)}..(6,4);
\draw[st] (2,0) to[out=90,in=180] (5,2) to[out=0,in=90] (8,0);
\draw[st] (4,0) to[out=90,in=180] (5,1) to[out=0,in=90] (6,0);
\draw[st] (2,4) to[out=270,in=180] (3,3) to[out=0,in=270] (4,4);
\end{tikzpicture}
\;\circ\;
\begin{tikzpicture}[scale=0.25,mylabels,centered]
\draw[st] (0,0) -- (0,3);
\draw[st] (2,0) to[out=90,in=180] (3,1) to[out=0,in=90] (4,0);
\draw[st] (6,0)..controls{(6,1.7) and (2,1.3)}..(2,3);
\draw[st] (4,3) to[out=270,in=180] (5,2) to[out=0,in=270] (6,3);
\draw[st] (8,0) -- (8,3);
\end{tikzpicture}
\;=\;
\begin{tikzpicture}[scale=0.25,mylabels,centered]
\draw[st] (0,3)..controls{(0,5.2) and (6,4.8)}..(6,7);
\draw[st] (2,3) to[out=90,in=180] (5,5) to[out=0,in=90] (8,3);
\draw[st] (4,3) to[out=90,in=180] (5,4) to[out=0,in=90] (6,3);
\draw[st] (2,7) to[out=270,in=180] (3,6) to[out=0,in=270] (4,7);
\draw[st] (0,0) -- (0,3);
\draw[st] (2,0) to[out=90,in=180] (3,1) to[out=0,in=90] (4,0);
\draw[st] (6,0)..controls{(6,1.7) and (2,1.3)}..(2,3);
\draw[st] (4,3) to[out=270,in=180] (5,2) to[out=0,in=270] (6,3);
\draw[st] (8,0) -- (8,3);
\end{tikzpicture}
\;
=\delta\;
\begin{tikzpicture}[scale=0.25,mylabels,centered]
\draw[st] (0,0)..controls{(0,2.2) and (6,1.8)}..(6,4);
\draw[st] (2,0) to[out=90,in=180] (3,1) to[out=0,in=90] (4,0);
\draw[st] (6,0) to[out=90,in=180] (7,1) to[out=0,in=90] (8,0);
\draw[st] (2,4) to[out=270,in=180] (3,3) to[out=0,in=270] (4,4);
\end{tikzpicture}\;.
\end{equation*}

\begin{deff}
Let $n\in\Z_{\geq 0}$. The Temperley--Lieb algebra $\TL_n^\kk(\delta)$ is defined to be $\End_{\tlcat^\kk(\delta)}(n)$.
\end{deff}

Morphisms in the Temperley--Lieb category will often be drawn as grey boxes, representing some linear combination of diagrams: for example,
\begin{equation*}
\begin{tikzpicture}[scale=0.25,mylabels,centered]
\draw[st] (1,-1) -- (1,3);
\draw[st] (3,-1) -- (3,3);
\draw[st] (5,-1) -- (5,3);
\draw[morg] (0,0) rectangle (6,2);
\node at (3,1) {$F$};
\end{tikzpicture}\;
\end{equation*}
denotes some arbitrary morphism $F\in\Hom_{\tlcat}(m,n)$, where the domain and codomain will be made clear from the context. We will sometimes use trapezes such as
\begin{equation*}
\begin{tikzpicture}[scale=0.25,mylabels,centered]
\draw[st] (1,2) -- (1,3);
\draw[st] (3,-1) -- (3,3);
\draw[st] (5,2) -- (5,3);
\draw[morg] (2,0) -- (4,0) -- (6,2) -- (0,2) -- cycle;
\node at (3,1) {$F$};
\end{tikzpicture}\;,\quad
\begin{tikzpicture}[scale=0.25,mylabels,centered]
\draw[st] (1,-1) -- (1,0);
\draw[st] (3,-1) -- (3,3);
\draw[st] (5,-1) -- (5,0);
\draw[morg] (0,0) -- (6,0) -- (4,2) -- (2,2) -- cycle;
\node at (3,1) {$F$};
\end{tikzpicture}\;,\quad
\begin{tikzpicture}[scale=0.25,mylabels,centered]
\draw[st] (1,-1) -- (1,3);
\draw[st] (3,-1) -- (3,3);
\draw[st] (5,2) -- (5,3);
\draw[morg] (0,0) -- (4,0) -- (6,2) -- (0,2) -- cycle;
\node at (2,1) {$F$};
\end{tikzpicture}\;\text{, and}\quad
\begin{tikzpicture}[scale=0.25,mylabels,centered]
\draw[st] (1,-1) -- (1,3);
\draw[st] (3,-1) -- (3,3);
\draw[st] (5,-1) -- (5,0);
\draw[morg] (0,0) -- (6,0) -- (4,2) -- (0,2) -- cycle;
\node at (2,1) {$F$};
\end{tikzpicture}
\end{equation*}
to emphasise that $F$ goes from a lower $m$ to a greater $n$, or vice versa.

\subsection{Light leaves and cellular structures}

The Temperley--Lieb algebra is a fundamental example of a particular class of algebras called \textbf{cellular} algebras.

\begin{deff}[\cite{Graham1996}]\label{cellular}
Let $R$ be a commutative unital ring. A cellular $R$-algebra is an associative unital algebra $\alg{A}$ with cellular data $(\cellposet, M,\cellbasissym,*)$, where
\begin{enumerate}
\item $(\cellposet,\preceq)$ is a partially ordered set and for all $\lambda\in\cellposet$, $\cellindices{\lambda}$ is a finite set such that $\cellbasissym:\sqcup_{\lambda\in\cellposet} \cellindices{\lambda}\times \cellindices{\lambda}\to \alg{A}$ is an injective map having an $R$-basis of $\alg{A}$ as its image.
\item If $\lambda\in\cellposet$ and $\cellindex{s},\cellindex{t}\in \cellindices{\lambda}$, we write $\cellbasis{s}{t}:=\cellbasissym(\cellindex{s},\cellindex{t})$. Then $*$ is an $R$-linear anti-involution of $\alg{A}$ such that $(\cellbasis{s}{t})^*=\cellbasis{t}{s}$.
\item \label{cell3} If $\lambda\in\cellposet$ and $\cellindex{s},\cellindex{t}\in \cellindices{\lambda}$, we have, for all $a\in\alg{A}$,
\begin{equation}\label{multcell}
a \cellbasis{s}{t}\equiv\sum_{\cellindex{u}\in \cellindices{\lambda}}\rr{a}{u}{s} \cellbasis{u}{t} \pmod{\alg{A}^{\prec \lambda}},
\end{equation}
where $\rr{a}{u}{s}\in R$ doesn't depend on $\cellindex{t}$ and where $\alg{A}^{\prec \lambda}$ is the $R$-submodule of $\alg{A}$ generated by $\{\cellbasis[\mu]{v}{w} : \mu\prec\lambda,\cellindex{v},\cellindex{w}\in \cellindices{\mu}\}$.
\end{enumerate}
\end{deff}

\begin{rem}\label{rightcell}
Note that applying the anti-involution $*$ to equation \eqref{multcell}, one gets
\begin{equation}\label{cellmult}
\cellbasis{t}{s} a^*\equiv \sum_{\cellindex{u}\in \cellindices{\lambda}} \rr{a}{u}{s}\cellbasis{t}{u}\pmod{\alg{A}^{\prec \lambda}}.
\end{equation}
\end{rem}

\begin{deff}\label{modcell}
Given a cellular algebra $\alg{A}$ with cellular datum as above, for all $\lambda\in\cellposet$, we define its (left) cell module $W(\lambda)$ as the free $R$-module with basis $\{\cellhalf{s} : \cellindex{s}\in \cellindices{\lambda}\}$ and left $\alg{A}$-action given by
\begin{equation*}
a \cdot \cellhalf{s}=\sum_{\cellindex{t}\in \cellindices{\lambda}}\rr{a}{t}{s}\cellhalf{t}
\end{equation*}
for all $a\in\alg{A},\cellindex{s}\in \cellindices{\lambda}$.
\end{deff}

The third property in Definition~\ref{cellular} ensures that the left action above is well-defined. Using Remark~\ref{rightcell}, is is thus possible to define a right action on the same space, giving a right cell module denoted $W(\lambda)^*$.

\begin{lem}[\cite{Graham1996}]\label{celllayers}
There is an isomorphism of $(\alg{A},\alg{A})$-bimodules $$\alpha^\lambda: W(\lambda)\otimes W(\lambda)^*\to \alg{A}^{\preceq \lambda}/\alg{A}^{\prec \lambda}$$ given by $\cellhalf{s}\otimes (\cellhalf{t})^*\mapsto \cellbasis{s}{t}+\alg{A}^{\prec \lambda}$. 
\end{lem}

The study of the structure of the cell modules is closely related to that of a bilinear form defined on them, the existence of which relies on the properties from Definition~\ref{cellular} and the following lemma.

\begin{lem}[\cite{Graham1996}]\label{lembil}
Let $\lambda\in\cellposet$ and $a\in\alg{A}$. For all $\cellindex{s},\cellindex{t},\cellindex{u},\cellindex{v}\in\cellindices{\lambda}$, we have
\begin{equation*}
\cellbasis{s}{t} a \cellbasis{u}{v} \equiv \phi_a(\cellindex{t},\cellindex{u}) \cellbasis{s}{v}\pmod{\alg{A}^{\prec \lambda}},
\end{equation*}
where $\phi_a(\cellindex{t},\cellindex{u})\in R$ is independent of $\cellindex{s}$ and $\cellindex{v}$.
\end{lem}

The scalars appearing in the above lemma allow us to define, for $\lambda\in\cellposet$,
\begin{equation}
\begin{aligned}
\langle\cdot ,\cdot\rangle:W(\lambda)\times W(\lambda)&\to\kk\\
(\cellhalf{s},\cellhalf{t})&\mapsto\ip{\cellhalf{s},\cellhalf{t}}=\phi_{\mathbf{1}}(\cellindex{s},\cellindex{t}),
\end{aligned}
\end{equation}
and extended bilinearly. Here, $\mathbf{1}$ is the unit of $\alg{A}$.

\begin{prop}[\cite{Graham1996}]
For all $\lambda\in\cellposet$, the bilinear form $\langle\cdot,\cdot\rangle$ is symmetric and
$
\langle ax,y\rangle=\langle x,a^* y\rangle
$
for all $x,y\in W(\lambda)$, $a\in\alg{A}$.
\end{prop}

This bilinear form provides the following powerful identification of the irreducible modules of $\alg{A}$.

\begin{theo}[\cite{Graham1996}]
Suppose $R$ is a field and set, for all $\lambda\in\cellposet$, $$\rad{(\lambda)}=\{x\in W(\lambda) : \ip{x,y}=0,\; \forall y\in W(\lambda)\}.$$
Also set $\cellposet_0=\{\lambda\in\Lambda : \rad{(\lambda)}\neq W(\lambda)\}$.
\begin{enumerate}
\item $\alg{A}$ is semisimple if and only if $\rad{(\lambda)}=0$ for all $\lambda\in\cellposet$.
\item For $\lambda\in\cellposet_0$, the head of the cell module $L(\lambda):=W(\lambda)/\rad{(\lambda)}$ is absolutely irreducible.
\item The set $\{L(\lambda) : \lambda\in\cellposet_0\}$ is a complete set of non-isomorphic irreducible modules.
\end{enumerate}
\end{theo}

Note that a cellular basis is not unique, nor is the cellular datum in general. In particular, a few different cellular bases for the Temperley--Lieb algebra may be constructed using the light ladder strategy of \cite{Elias2015}, as done in \cite{Sutton2023} and recalled below. Although these bases are different, the corresponding cell layers (in the sense of \cite{Konig1996}) stay the same and hence the cell modules are isomorphic.

\begin{deff}\label{lightladders}
Let $F\in\Hom_{\tlcat}(m,n)$ be a morphism and fix a family of morphisms $\{G_n\in\End_{\tlcat}(n)\}_{n\geq 0}$. 
Define new morphisms $\varepsilon_{-1}(F)\in\Hom_{\tlcat}(m+1,n-1)$, provided $n-1\geq 0$, and $\varepsilon_{+1}(F)\in\Hom_{\tlcat}(m+1,n+1)$ by
\begin{equation*}
\varepsilon_{-1}(F):=\begin{tikzpicture}[scale=0.25,mylabels,centered]
\draw[st] (0,-1) -- (0, 0);
\draw[st] (2,-1) -- (2, 0);
\draw[st] (4,-1) -- (4, 0);
\draw[st] (6,-1) -- (6, 0);
\draw[morg] (-1,0) -- (1,2) -- (7,2) -- (7,0) -- cycle;
\node at (4,1) {$F$};
\tlcap{6,2}{7,3}{8,2};
\draw[st] (8,-1) -- (8,2);
\draw[st] (2,2) -- (2,5);
\draw[st] (4,2) -- (4,5);
\draw[morg] (1,2) rectangle (5,4);
\node at (3,3) {$G_{n-1}$};
\end{tikzpicture}\;
,\quad
\varepsilon_{+1}(F):=
\begin{tikzpicture}[scale=0.25,mylabels,centered]
\draw[st] (0,-1) -- (0, 0);
\draw[st] (2,-1) -- (2, 0);
\draw[st] (4,-1) -- (4, 0);
\draw[st] (6,-1) -- (6, 0);
\draw[morg] (-1,0) -- (1,2) -- (7,2) -- (7,0) -- cycle;
\node at (4,1) {$F$};
\draw[st] (8,-1) -- (8,2);
\draw[st] (2,2) -- (2,5);
\draw[st] (4,2) -- (4,5);
\draw[st] (6,2) -- (6,5);
\draw[st] (8,2) -- (8,5);
\draw[morg] (1,2) rectangle (9,4);
\node at (5.5,3) {$G_{n+1}$};
\end{tikzpicture}
\;
.
\end{equation*}
Then, given a sequence $\underline{s}=(s_1,s_2,\ldots,s_k)\in\{\pm 1\}^k$, define $\varepsilon_{\underline{s}}(F)\in\Hom_{\tlcat}(m+k,n+l)$, where $l=\sum_{i=1}^k s_i$, by $\varepsilon_{\underline{s}}(F):=\varepsilon_{s_k}\circ\ldots\circ\varepsilon_{s_2}\circ\varepsilon_{s_1}(F)$. Note that this is well-defined if and only if all the partial sums of the sequence $\underline{s}$ are non-negative.

The \textbf{down morphism} associated to the sequence $\underline{s}$ and defined by applying $\varepsilon_{\underline{s}}$ to the empty diagram is denoted by $\delta^G_{\underline{s}}\in\Hom_{\tlcat}(k,l)$. The corresponding \textbf{up morphism}, obtained by taking the vertical reflection of $\delta^G_{\underline{s}}$, is denoted by $\upsilon^G_{\underline{s}}\in\Hom_{\tlcat}(l,k)$.
\end{deff}

These morphisms will shortly form ``rungs'' of a ladder-like construction.

\begin{rem}
More generally, light ladders in the $\mathfrak{sl}_n$ web category are indexed by paths in the associated positive Weyl chamber. In the present case, for $\mathfrak{sl}_2$, this is equivalent to Dyck paths, or sequences of $\pm 1$'s with non-negative partial sums as described above.
\end{rem}

To define the cellular bases for $\tlnkd$, start with the poset $\tlcellposet=\{\lambda_1-\lambda_2 : (\lambda_1,\lambda_2)\vdash n\}$ of the differences in length between the rows of partitions of $n$ having at most two parts, with the usual order on the integers. Note that the elements of $\tlcellposet$ can be written in simpler terms as $\{m\in\Z : 0\leq m\leq n\mbox{ and }m\equiv_2 n\}$, but the associated partitions actually carry useful information. For $\lambda\in\tlcellposet$, let $\tlcellindices{\lambda}=\Std((\lambda_1,\lambda_2))$ be the set of standard Young tableaux of shape $(\lambda_1,\lambda_2)$. From a standard Young tableau $\cellindex{t}\in\tlcellindices{\lambda}$, define an associated sequence $\underline{s}_{\cellindex{t}}=(s_1,\ldots,s_n)\in\{\pm 1\}^n$ by setting $s_i=+1$ if $i$ lies in the first row of $\cellindex{t}$ and $s_i=-1$ if it lies in the second row. Note that this is a valid Dyck path due to the fact the Young tableau is standard.

Then, given two standard Young tableaux $\cellindex{t_1},\cellindex{t_2}\in\tlcellindices{\lambda}$ and their associated sequences $\underline{s}_{\cellindex{t_1}},\underline{s}_{\cellindex{t_2}}$, write $\upsilon^G_{\cellindex{t_1}}:=\upsilon^G_{\underline{s}_{\cellindex{t_1}}}$ and $\delta^G_{\cellindex{t_2}}:=\delta^G_{\underline{s}_{\cellindex{t_2}}}$, and define $\cellbasis{t_1}{t_2}:=\upsilon^G_{\cellindex{t_1}} \cdot\delta^G_{\cellindex{t_2}}$ and $\cellhalf{t_1}=\upsilon^G_{\cellindex{t_1}}$. Finally, the anti-involution $*$ is given by reflecting the diagrams across the horizontal axis. As will be seen below, some particular choices for the family $\{G_n\}$ then lead to different cellular bases, denoted by $\basis(G)=\{\cellbasis{s}{t} : \lambda\in\cellposet,\cellindex{s},\cellindex{t}\in\tlcellindices{\lambda}\}$.

\begin{theo}[\cite{Graham1996}]
Taking $G_n=\id_n$ for all $n$, the above construction gives the diagram basis $\basis(\id)$ of the Temperley--Lieb algebra $\TL_n^\kk(\delta)$, which is cellular with the datum described above.
\end{theo}

We can now apply the definition of cell modules to the Temperley--Lieb algebras.

\begin{deff}
The cell modules for the cellular algebra $\TL_n(\delta)$ are denoted by $\ctlmod{k}$, for $k\in\cellposet$. A basis of $\ctlmod{k}$ is given by the set $\{\cellhalf[k]{t} : \cellindex{t}\in \cellindices{k}\}$ obtained from the light ladder construction using the family $G_n=\id_n$ for all $n$, and the action is given by the usual left multiplication modulo terms of lower cellular order. This amounts to saying that the basis is given by all epic $(n,k)$-diagrams, and that the action is the usual multiplication with resulting diagrams sent to zero if the number of propagating lines decreases. The simple head of the cell module $\ctlmod{k}$ will be denoted by $\stlmod{k}$.
\end{deff}

More generally, the half-diagrams $\upsilon^{\id}_{\cellindex{t}}$ coming from the light-ladder construction using the family $G_n=\id_n$ give a way to write down a basis of any morphism space within the Temperley--Lieb category.

\begin{lem}\label{lemma:homspace}
Let $a,b\in\Z_{\geq 0}$ of the same parity. Then, a basis of the space $\Hom_{\tlcat}(a,b)$ is given by the set
\begin{equation*}
\bigsqcup_{k=0}^{\min{(a,b)}}\left\{\upsilon^{\id}_{\cellindex{s}}\cdot\delta^{\id}_{\cellindex{t}} =\;
\begin{tikzpicture}[scale=0.25,mylabels,centered]
\draw[st] (1,-1) -- (1,0);
\draw[st] (1,4) -- (1,5);
\draw[st] (3,-1) -- (3,5);
\draw[st] (5,-1) -- (5,5);
\draw[st] (7,-1) -- (7,0);
\draw[st] (7,4) -- (7,5);
\draw[morg] (0,0) -- (2,2) -- (6,2) -- (8,0) -- cycle;
\node at (4,1) {$\delta^{\id}_{\cellindex{t}}$};
\draw[morg] (2,2) -- (6,2) -- (8,4) -- (0,4) -- cycle;
\node at (4,3) {$\upsilon^{\id}_{\cellindex{s}}$};
\end{tikzpicture}\;
: \cellindex{s}\in\tlcellindices[b]{k}, \cellindex{t}\in\tlcellindices[a]{k} \right\},
\end{equation*}
where $\tlcellindices[a]{k}$ and $\tlcellindices[b]{k}$ are non-empty only when $k$ is of the same parity as both $a$ and $b$.
\end{lem}

However, other choices for the family $\{G_n\}$ will be of crucial importance for the rest of this paper. A sufficient condition for the corresponding light ladders to lead to bases of morphism spaces, cellular bases of $\TL_n$ algebras, and isomorphic corresponding cell modules is in the statement of Proposition~\ref{unitriangular}. For completeness, we provide a proof, which will make use of the following order on standard tableaux.

\begin{deff}[{\cite[Definition 3.4]{Bowman2017}}]
Let $\mathfrak{s},\mathfrak{t}$ be two standard tableaux of shape $\lambda\vdash n$ and let $\trianglelefteq$ denote the usual dominance order on Young diagrams. For any $1\leq i\leq n$, write $\mathfrak{t}(i)$ for the Young diagram of the subtableau of $\mathfrak{t}$ consisting of the boxes containing the integers $1,2,\ldots,i$. We denote by $\bleq$ the \textbf{reverse lexicographic order} on standard tableaux, defined by $\mathfrak{s}\bleq \mathfrak{t}$ if either $\mathfrak{s}=\mathfrak{t}$ or $\mathfrak{s}(i)\trianglelefteq\mathfrak{t}(i)$ for the largest $i$ such that $\mathfrak{s}(i)\neq \mathfrak{t}(i)$. We write $\mathfrak{s}\blacktriangleleft\mathfrak{t}$ if $\mathfrak{s}\bleq\mathfrak{t}$ and $\mathfrak{s}\neq\mathfrak{t}$.
\end{deff}

\begin{ex}
Consider the standard tableaux $\mathfrak{s},\mathfrak{t}$ of shape $(4,2,1)\vdash 7$ given by
\begin{equation*}
\mathfrak{s}=\ytableaushort{1246,35,7}\quad\mbox{and}\quad
\mathfrak{t}=\ytableaushort{1235,46,7}.
\end{equation*}
Then, $\mathfrak{s}(7)=\mathfrak{t}(7)$ and $\mathfrak{s}(6)=\mathfrak{t}(6)$, but $\mathfrak{s}(5)=(3,2)\neq (4,1)=\mathfrak{t}(5)$. Since $(3,2)\triangleleft (4,1)$, this gives $\mathfrak{s}\blacktriangleleft \mathfrak{t}$.
\end{ex}

When choosing the family $\{G_n\}$ as in the following proposition, one can show that the change of basis between $\basis(\id)$ and $\basis(G)$ is unitriangular with respect to the reverse lexicographic order. This proves cellularity for the new basis.

\begin{prop}\label{unitriangular}
If $G_n$ is a family of morphisms such that $G_n\equiv \id_n \pmod{\td(<n)}$ for all $n$, then the corresponding light ladders lead to bases of the morphism spaces as in Lemma~\ref{lemma:homspace} and $\basis(G)$ is a cellular basis of $\tlnkd$.
\end{prop}
\begin{proof}
The first step is to show that changing from $\id_n$ to $G_n$ in the light ladder construction induces a unitriangular change of basis in every cell module.

Note that $\tlnkd^{<k}=\td(<n)$ is the ideal generated by diagrams in $\Hom_{\tlcat}(k,n)$ that do not have maximal through degree. For every $k\in\cellposet$ and $\cellindex{t}\in\tlcellindices{k}$, since the set $\{\upsilon^{\id}_{\cellindex{t}}:\cellindex{t}\in\tlcellindices{k}\}$ is a basis of the cell module $\ctlmod{k}$, we can write
\begin{equation*}
\upsilon^G_{\cellindex{t}}= \upsilon^{\id}_{\cellindex{t}}+\sum_{\cellindex{u}\in\tlcellindices{k}}f(\cellindex{t},\cellindex{u})\upsilon^{\id}_{\cellindex{u}}+\td(<k),
\end{equation*}
for some $f(\cellindex{t},\cellindex{u})\in\kk$, where the sole copy of $\upsilon^{\id}_{\cellindex{t}}$ is obtained by taking the identity term in the expansion of every $G_i$ appearing in $\upsilon^G_{\cellindex{t}}$.

{\renewcommand{\qedsymbol}{\ensuremath{\triangle}}
\noindent\textbf{\underline{Claim:}} $f(\cellindex{t},\cellindex{u})=0$ unless $\cellindex{u}\blacktriangleleft\cellindex{t}$.
\begin{proof}[Proof of the claim]
If $n=0$ or $1$, there is nothing to show. Let $n\geq 2$ and suppose the result is true for non-negative integers less than $n$. Let $\upsilon^G_{\cellindex{t}}\in\Hom_{\tlcat}(k,n)$, and focus on the last two steps in its light ladder construction. There are four possibilities, drawn below with their corresponding Young tableaux. Remember that the conventions here make the pictures upside down compared to Definition~\ref{lightladders}.
\ytableausetup{nosmalltableaux}
\begin{equation*}
\setlength{\tabcolsep}{10pt}
\renewcommand{\arraystretch}{5}
\scalebox{0.8}{
\begin{tabular}{cccc}
\begin{tikzpicture}[scale=0.25,mylabels,centered]
    \draw[st] (1,1) -- (1,9);
    \draw[st] (3,1) -- (3,9);
    \draw[st] (5,1) -- (5,9);
    \draw[st] (7,1) -- (7,9);
    \draw[st] (9,8) -- (9,9);
    \draw[morg] (0,2) rectangle (8,4);
    \draw[morg] (2,4) rectangle (8,6);
    \draw[morg] (4,6) -- (8,6) -- (10,8) -- (4,8) -- cycle;
    \node at (4,3) {$G_k$};
    \node at (5,5) {$G_{k-1}$};
    \node at (6,7) {$F$};
\end{tikzpicture}
&
\begin{tikzpicture}[scale=0.25,mylabels,centered]
    \draw[st] (1,1) -- (1,2);
    \draw[st] (1,6) -- (1,9);
    \draw[st] (1,4)..controls{(1,5) and (-3,4.6)}..(-3,6);
    \draw[st] (-3,6) -- (-3,9);
    \draw[st] (3,1) -- (3,9);
    \draw[st] (5,1) -- (5,9);
    \draw[st] (7,1) -- (7,9);
    \draw[st] (9,8) -- (9,9);
    \tlcup{-1,6}{0,5.1}{1,6};
    \draw[st] (-1,6) -- (-1,9);
    \draw[morg] (0,2) rectangle (8,4);
    \draw[morg] (2,4) rectangle (8,6);
    \draw[morg] (0,6) -- (8,6) -- (10,8) -- (0,8) -- cycle;
    \node at (4,3) {$G_k$};
    \node at (5,5) {$G_{k-1}$};
    \node at (4,7) {$F$};
\end{tikzpicture}
&
\begin{tikzpicture}[scale=0.25,mylabels,centered]
    \draw[st] (-7,4) -- (-7,9);
    \draw[st] (-5,6) -- (-5,9);
    \draw[st] (-3,6) -- (-3,9);
    \draw[st] (-1,4) -- (-1,9);
    \draw[st] (1,1) -- (1,9);
    \draw[st] (3,1) -- (3,9);
    \draw[st] (5,1) -- (5,9);
    \draw[st] (7,1) -- (7,9);
    \draw[st] (9,8) -- (9,9);
    \tlcup{-5,6}{-4,5}{-3,6};
    \tlcup{-7,4}{-4,2}{-1,4};
    \draw[morg] (0,2) rectangle (8,4);
    \draw[morg] (-2,4) rectangle (8,6);
    \draw[morg] (-4,6) -- (8,6) -- (10,8) -- (-4,8) -- cycle;
    \node at (4,3) {$G_k$};
    \node at (3,5) {$G_{k+1}$};
    \node at (2,7) {$F$};
\end{tikzpicture}
&
\begin{tikzpicture}[scale=0.25,mylabels,centered]
    \draw[st] (1,1) -- (1,9);
    \draw[st] (-3,4) -- (-3,9);
    \draw[st] (3,1) -- (3,9);
    \draw[st] (5,1) -- (5,9);
    \draw[st] (7,1) -- (7,9);
    \draw[st] (9,8) -- (9,9);
    \tlcup{-3,4}{-2,3}{-1,4};
    \draw[st] (-1,4) -- (-1,9);
    \draw[morg] (0,2) rectangle (8,4);
    \draw[morg] (-2,4) rectangle (8,6);
    \draw[morg] (0,6) -- (8,6) -- (10,8) -- (0,8) -- cycle;
    \node at (4,3) {$G_k$};
    \node at (3,5) {$G_{k+1}$};
    \node at (4,7) {$F$};
\end{tikzpicture}\\
$\substack{\scalebox{0.9}{\begin{tikzpicture}[scale=0.35,centered,mylabels]
\draw[densely dotted] (0,0) -- (6,0) -- (6,2) -- (8,2) -- (8,4) -- (0,4) -- cycle;
\draw (8,2) rectangle (10,4);
\draw (10,2) rectangle (12,4);
\node at (9,3) {$n-1$};
\node at (11,3) {$n$};
\end{tikzpicture}}\vspace{0.5em}\\
\cellindex{t_1}}$
&
$\substack{
\scalebox{0.9}{\begin{tikzpicture}[scale=0.35,centered,mylabels]
\draw[densely dotted] (0,0) -- (4,0) -- (4,2) -- (10,2) -- (10,4) -- (0,4) -- cycle;
\draw (10,2) rectangle (12,4);
\draw (4,0) rectangle (6,2);
\node at (11,3) {$n$};
\node at (5,1) {$n-1$};
\end{tikzpicture}}\vspace{0.5em}\\
\cellindex{t_2}}$
&
$\substack{
\scalebox{0.9}{\begin{tikzpicture}[scale=0.35,centered,mylabels]
\draw[densely dotted] (0,0) -- (4,0) -- (4,2) -- (12,2) -- (12,4) -- (0,4) -- cycle;
\draw (2,0) rectangle (4,2);
\draw (4,0) rectangle (6,2);
\node at (5,1) {$n$};
\node at (3,1) {$n-1$};
\end{tikzpicture}}\vspace{0.5em}\\
\cellindex{t_3}}$
&
$\substack{
\scalebox{0.9}{\begin{tikzpicture}[scale=0.35,centered,mylabels]
\draw[densely dotted] (0,0) -- (4,0) -- (4,2) -- (10,2) -- (10,4) -- (0,4) -- cycle;
\draw (10,2) rectangle (12,4);
\draw (4,0) rectangle (6,2);
\node at (11,3) {$n-1$};
\node at (5,1) {$n$};
\end{tikzpicture}}\vspace{0.5em}\\
\cellindex{t_4}}$
\end{tabular}
}
\end{equation*}
In each case, $F$ is a morphism in $\Hom_{\tlcat}(m,n-2)$, for some $m\in\{k,k-2,k+2\}$. In the first two cases, any choice of diagram other than the identity in the expansion of $G_k$ or $G_{k-1}$ leads to terms that are not monic, so the result follows by the inductive hypothesis applied to $F$. In the last two cases, there is exactly one other possible monic term coming from
\begin{equation*}
\begin{tikzpicture}[scale=0.25,mylabels,centered]
    \draw[st] (-7,4) -- (-7,9);
    \draw[st] (-5,6) -- (-5,9);
    \draw[st] (-3,6) -- (-3,9);
    \draw[st] (-1,6) -- (-1,9);
    \draw[st] (1,1) -- (1,4);
    \draw[st] (1,6) -- (1,9);
    \draw[st] (3,1) -- (3,9);
    \draw[st] (5,1) -- (5,9);
    \draw[st] (7,1) -- (7,9);
    \draw[st] (9,8) -- (9,9);
    \tlcup{-5,6}{-4,5}{-3,6};
    \tlcup{-7,4}{-4,2}{-1,4};
    \draw[empty] (0,2) rectangle (8,4);
    \draw[empty] (-2,4) rectangle (8,6);
    \draw[morg] (-4,6) -- (8,6) -- (10,8) -- (-4,8) -- cycle;
    \tlcup{-1,6}{0,5.2}{1,6};
    \tlcap{-1,4}{0,4.8}{1,4};
    \node at (2,7) {$F$};
\end{tikzpicture}\quad\quad\mbox{and}\quad\quad
\begin{tikzpicture}[scale=0.25,mylabels,centered]
    \draw[st] (1,1) -- (1,4);
    \draw[st] (1,6) -- (1,9);
    \draw[st] (-3,4) -- (-3,9);
    \draw[st] (3,1) -- (3,9);
    \draw[st] (5,1) -- (5,9);
    \draw[st] (7,1) -- (7,9);
    \draw[st] (9,8) -- (9,9);
    \tlcup{-3,4}{-2,3}{-1,4};
    \draw[st] (-1,6) -- (-1,9);
    \draw[empty] (0,2) rectangle (8,4);
    \draw[empty] (-2,4) rectangle (8,6);
    \draw[morg] (0,6) -- (8,6) -- (10,8) -- (0,8) -- cycle;
    \node at (4,7) {$F$};
    \tlcup{-1,6}{0,5.2}{1,6};
    \tlcap{-1,4}{0,4.8}{1,4};
\end{tikzpicture}\;,
\end{equation*}
both of which reduce to the second case above. Since $\cellindex{t_2}\blacktriangleleft\cellindex{t_3}$ and $\cellindex{t_2}\blacktriangleleft\cellindex{t_4}$, the claim follows.
\end{proof}
}

The claim allows us to write
\begin{equation*}
\upsilon^G_{\cellindex{t}}= \upsilon^{\id}_{\cellindex{t}}+\sum_{\substack{\cellindex{u}\in\tlcellindices{\lambda}\\\cellindex{u}\blacktriangleleft\cellindex{t}}}f(\cellindex{t},\cellindex{u})\upsilon^{\id}_{\cellindex{u}}+\td(<k),
\end{equation*}
indicating that there is indeed a unitriangular change of basis between the sets $\{\upsilon^{\id}_{\cellindex{t}} : \cellindex{t}\in\tlcellindices{k}\}$ and $\{\upsilon^G_{\cellindex{t}} : \cellindex{t}\in\tlcellindices{k}\}$. Since the former is a basis of $\ctlmod{k}$, so is the latter.

These bases of cell modules can be globalised (see Section 2.3 of \cite{Goodman2011} for details) into a cellular basis of $\tlnkd$ by taking $\{\alpha^k(\upsilon^G_{\cellindex{s}}\otimes(\upsilon^G_{\cellindex{t}})^*)\}$, where $\alpha^k$ is the isomorphism of Lemma~\ref{celllayers}. This is precisely the definition of $\basis(G)$.

Finally, the fact that the light ladders corresponding to the family $\{G_n\}$ can be used to construct bases of morphism spaces as in Lemma~\ref{lemma:homspace} also follows from unitriangularity with respect to the cellular order.
\end{proof}

Our main use for Proposition~\ref{unitriangular} is to apply it for the families of Jones--Wenzl idempotents $\{\jw_n\}$, semisimple $(\ell,\p)$-Jones--Wenzl idempotents $\{\jwlpz_n\}$ and their specialisations $\{\jwlp_n\}$, all of which will be introduced in the next section. Since all these families satisfy the hypothesis of Proposition~\ref{unitriangular}, $\basis(\jw)$, $\basis(\jwlpz)$, and $\basis(\jwlp)$ are cellular bases of $\TL_n$ over the appropriate field and all lead to isomorphic cell modules.

\begin{ex}
Here are the 3 down morphisms of $\TL_3^\kk(\delta)$, along with the corresponding Young tableaux.
\setlength{\tabcolsep}{20pt}
\renewcommand{\arraystretch}{3}
\begin{equation*}
\begin{tabular}{c|cc}
\begin{tikzpicture}[scale=0.25,mylabels,centered]
    \draw[st] (1,-1) -- (1,3);
    \draw[st] (3,-1) -- (3,3);
    \draw[st] (5,-1) -- (5,3);
    \draw[morg] (0,0) rectangle (6,2);
    \node at (3,1) {$G_3$};
\end{tikzpicture}
&
\begin{tikzpicture}[scale=0.25,mylabels,centered]
    \draw[st] (1,-1) -- (1,3);
    \draw[st] (3,-1) -- (3,2);
    \draw[st] (5,-1) -- (5,2);
    \draw[morg] (0,0) rectangle (4,2);
    \draw[st] (3,2) to[out=90,in=180] (4,3) to[out=0,in=90] (5,2);
    \node at (2,1) {$G_2$};
\end{tikzpicture}
&
\begin{tikzpicture}[scale=0.25,mylabels,centered]
    \draw[st] (1,1) to[out=90,in=180] (2,2) to[out=0,in=90] (3,1);
    \draw[st] (5,1) -- (5,3);
\end{tikzpicture}\\
$\cellindex{s}=\ytableaushort{123}$
&
$\cellindex{t}=\ytableaushort{12,3}$
&
$\cellindex{u}=\ytableaushort{13,2}$
\end{tabular}
\end{equation*}
Note that there are hidden copies of $G_1=\id_1$ on the strands going up in the second and third diagrams. The corresponding cellular basis of $\TL_3^\kk(\delta)$ is
\begin{equation*}
\cellbasis[\hspace{1pt} \resizebox{10pt}{!}{\ydiagram{3}}]{s}{s}=
\begin{tikzpicture}[scale=0.25,mylabels,centered]
    \draw[st] (1,-1) -- (1,3);
    \draw[st] (3,-1) -- (3,3);
    \draw[st] (5,-1) -- (5,3);
    \draw[morg] (0,0) rectangle (6,2);
    \node at (3,1) {$G_3$};
\end{tikzpicture}\;,\quad
\cellbasis[\hspace{1pt} \resizebox{6pt}{!}{\ydiagram{2,1}}]{t}{t}=\;
\begin{tikzpicture}[scale=0.25,mylabels,centered]
    \draw[st] (1,-1) -- (1,7.5);
    \draw[st] (3,-1) -- (3,2);
    \draw[st] (5,-1) -- (5,2);
    \draw[morg] (0,0) rectangle (4,2);
    \draw[st] (3,2) to[out=90,in=180] (4,3) to[out=0,in=90] (5,2);
    \draw[st] (3,4.5) to[out=270,in=180] (4,3.5) to[out=0,in=270] (5,4.5);
    \draw[st] (3,4.5) -- (3,7.5);
    \draw[st] (5,4.5) -- (5,7.5);
    \draw[morg] (0,4.5) rectangle (4,6.5);
    \node at (2,1) {$G_2$};
    \node at (2,5.5) {$G_2$};
\end{tikzpicture}\;,\quad
\cellbasis[\hspace{1pt} \resizebox{6pt}{!}{\ydiagram{2,1}}]{u}{t}=\;
\begin{tikzpicture}[scale=0.25,mylabels,centered]
    \draw[st] (1,-1) -- (1,3);
    \draw[st] (3,-1) -- (3,2);
    \draw[st] (5,-1) -- (5,2);
    \draw[morg] (0,0) rectangle (4,2);
    \draw[st] (3,2) to[out=90,in=180] (4,3) to[out=0,in=90] (5,2);
    \draw[st] (1,6.5) to[out=270,in=180] (2,5.5) to[out=0,in=270] (3,6.5);
    \draw[st] (5,5.5) -- (5,6.5);
    \draw[st] (1,3)..controls{(1,4.5) and (5,4)}..(5,5.5);
    \node at (2,1) {$G_2$};
\end{tikzpicture}\;,\quad
\cellbasis[\hspace{1pt} \resizebox{6pt}{!}{\ydiagram{2,1}}]{t}{u}=\;
\begin{tikzpicture}[scale=0.25,mylabels,centered]
    \draw[st] (1,0) to[out=90,in=180] (2,1) to[out=0,in=90] (3,0);
    \draw[st] (5,0) -- (5,1);
    \draw[st] (3,4.5) to[out=270,in=180] (4,3.5) to[out=0,in=270] (5,4.5);
    \draw[st] (3,4.5) -- (3,7.5);
    \draw[st] (5,4.5) -- (5,7.5);
    \draw[st] (1,3.5) -- (1,7.5);
    \draw[morg] (0,4.5) rectangle (4,6.5);
    \draw[st] (1,3.5)..controls{(1,2) and (5,2.5)}..(5,1);
    \node at (2,5.5) {$G_2$};
\end{tikzpicture}\;,\quad
\cellbasis[\hspace{1pt} \resizebox{6pt}{!}{\ydiagram{2,1}}]{u}{u}=\;
\begin{tikzpicture}[scale=0.25,mylabels,centered]
    \draw[st] (1,0) to[out=90,in=180] (2,1) to[out=0,in=90] (3,0);
    \draw[st] (1,3) to[out=270,in=180] (2,2) to[out=0,in=270] (3,3);
    \draw[st] (5,0) -- (5,3);
\end{tikzpicture}\;.
\end{equation*}
\end{ex}

The problem of counting the composition factors of cell modules in mixed characteristic has been done in \cite{Spencer2023}. Their result may be stated in terms of up-admissible sets as follows.

\begin{theo}[\cite{Spencer2023}]\label{compfac}
The composition factors of $\ctlmod{m}$ are given by the set
\begin{equation*}
\left\{\stlmod{m(S)} : m(S)\leq n, S\mbox{ is up-admissible for } m\right\}.
\end{equation*}
\end{theo}
\begin{proof}
The simple module $\stlmod{r}$ is known to be a composition factor of $\ctlmod{m}$ if and only if $m\in\supp{r}$, in which case it appears with multiplicity one (see \cite{Spencer2023}, Theorem~8.4). The condition $m\in\supp{r}$ is equivalent to the fact that there exists a down-admissible set $S$ for $r$ such that $r[S]=m$. But then $S$ is up-admissible for $m$ and $m(S)=r[S](S)=r$, which gives the desired result.
\end{proof}

\subsection{Up-morphisms, down-morphisms, and Jones--Wenzl idempotents} For any integer $n\geq 1$, the well-known (semisimple) Jones--Wenzl idempotent $\jw_n$ can be defined as the unique element of $\TL_n^{\Q(\delta)}$ projecting on the trivial module. They will be denoted by a grey rectangle:
\begin{equation*}
\jw_n=\begin{tikzpicture}[scale=0.25,mylabels,centered]
\draw[st] (1,-1) -- (1,3);
\draw[st] (3,-1) -- (3,3);
\draw[st] (5,-1) -- (5,3);
\draw[jw] (0,0) rectangle (6,2);
\node at (3,1) {$n$};
\end{tikzpicture}\;.
\end{equation*}
They have an explicit recursive definition given by
\begin{equation*}
\jw_1=\id_1,\quad\text{and}\quad
\begin{tikzpicture}[scale=0.25,mylabels,centered]
\draw[st] (1,-1) -- (1,3);
\draw[st] (3,-1) -- (3,3);
\draw[st] (5,-1) -- (5,3);
\draw[st] (7,-1) -- (7,3);
\draw[jw] (0,0) rectangle (8,2);
\node at (4,1) {$n+1$};
\end{tikzpicture}
\;=\;
\begin{tikzpicture}[scale=0.25,mylabels,centered]
\draw[st] (1,-1) -- (1,3);
\draw[st] (3,-1) -- (3,3);
\draw[st] (5,-1) -- (5,3);
\draw[jw] (0,0) rectangle (6,2);
\node at (3,1) {$n$};
\draw[st] (7,-1) -- (7,3);
\end{tikzpicture}\;
-\frac{\qnn{n}}{\qnn{n+1}}\;
\begin{tikzpicture}[scale=0.25,mylabels,centered]
\draw[st] (1,-1) -- (1,8);
\draw[st] (3,-1) -- (3,8);
\draw[st] (5,-1) -- (5,2);
\draw[st] (5,5) -- (5,8);
\draw[st] (7,-1) -- (7,2);
\draw[st] (7,5) -- (7,8);
\draw[jw] (0,0) rectangle (6,2);
\node at (3,1) {$n$};
\tlcap{5,2}{6,3}{7,2};
\tlcup{5,5}{6,4}{7,5};
\draw[jw] (0,5) rectangle (6,7);
\node at (3,6) {$n$};
\end{tikzpicture}\;.
\end{equation*}
\begin{lem}\label{jwproperties}
The semisimple Jones--Wenzl idempotents satisfy $(\jw_n)^*=\jw_n$, as well as
\begin{gather*}
\begin{tikzpicture}[scale=0.25,mylabels,centered]
\draw[st] (1,-1) -- (1,5);
\draw[st] (3,-1) -- (3,5);
\draw[st] (5,-1) -- (5,5);
\draw[st] (7,-1) -- (7,5);
\draw[jw] (0,0) rectangle (8,2);
\node at (4,1) {$n$};
\draw[jw] (2,2) rectangle (6,4);
\node at (4,3) {$m$};
\end{tikzpicture}\;
=\;
\begin{tikzpicture}[scale=0.25,mylabels,centered]
\draw[st] (1,-1) -- (1,3);
\draw[st] (3,-1) -- (3,3);
\draw[st] (5,-1) -- (5,3);
\draw[st] (7,-1) -- (7,3);
\draw[jw] (0,0) rectangle (8,2);
\node at (4,1) {$n$};
\end{tikzpicture}\;,\quad
\begin{tikzpicture}[scale=0.25,mylabels,centered]
\draw[st] (1,-1) -- (1,3);
\draw[st] (3,-1) -- (3,2);
\draw[st] (5,-1) -- (5,2);
\draw[st] (7,-1) -- (7,3);
\draw[jw] (0,0) rectangle (8,2);
\node at (4,1) {$n$};
\tlcap{3,2}{4,3}{5,2};
\end{tikzpicture}\;
=0,\\
\begin{tikzpicture}[scale=0.25,mylabels,centered]
\draw[st] (1,-1) -- (1,3);
\draw[st] (3,-1) -- (3,3);
\draw[st] (5,-1) -- (5,3);
\draw[jw] (0,0) rectangle (8,2);
\node at (4,1) {$n$};
\tlcap{7,2}{8,3}{9,2};
\tlcup{7,0}{8,-1}{9,0};
\draw[st] (9,0) -- (9,2);
\node at (9.8,1) {$k$};
\end{tikzpicture}\;
=\;
\frac{\qnn{n+1}}{\qnn{n+1-k}}\;
\begin{tikzpicture}[scale=0.25,mylabels,centered]
\draw[st] (1,-1) -- (1,3);
\draw[st] (3,-1) -- (3,3);
\draw[st] (5,-1) -- (5,3);
\draw[jw] (0,0) rectangle (6,2);
\node at (3,1) {$n-k$};
\end{tikzpicture}\;,
\end{gather*}
where the idempotent $\jw_m$ and the cap can be placed anywhere on top of $\jw_n$. The invariance under $*$ implies that the same properties hold with $\jw_m$ placed underneath instead, and with the cap replaced by a cup on the bottom of $\jw_n$.
\end{lem}

When $(\kk,\delta)$ is chosen so that $\tlnkd$ is not semisimple, the Jones--Wenzl idempotents are no longer well-defined in general. However, they have analogues, called the $(\ell,\p)$-Jones--Wenzl idempotents and denoted by $\jwlp_n$, that project onto the \emph{projective cover} of the trivial module instead. They were first defined in \cite{Burrull2019} for the characteristic $\p$ case, and in \cite{Martin2022} and \cite{Sutton2023} independently for the mixed case.

\begin{deff}[{\cite[Definition 3.7]{Sutton2023}}]
Fix a family of morphisms $\{G_n\in\End_{\tlcat}(n)\}_{n\geq 0}$. Let $n\in\Z_{\geq 0}$ and write $n+1=[n_k,n_{k-1},\ldots,n_0]_{\ell,\p}$, and let $0\leq i<k$ such that $n_i\neq 0$. Let $x=[n_k,\ldots,n_{i+1},0,\ldots,0]_{\ell,\p}-1$, and define a morphism $\delta_i:n\to n[\{i\}]$ by
\begin{equation*}
\delta_i=\;
\begin{tikzpicture}[scale=0.25,mylabels,centered]
\draw[st] (1,-1) -- (1,4);
\draw[st] (3,-1) -- (3,2);
\draw[st] (5,-1) -- (5,2);
\draw[st] (7,-1) -- (7,4);
\draw[morg] (0,0) rectangle (4,2);
\node at (2,1) {$G_x$};
\tlcap{3,2}{4,3}{5,2};
\node at (4.2,3.9) {$n_i\p^{(i)}$};
\end{tikzpicture}\;.
\end{equation*}
If $S=\{s_0<s_1<\cdots <s_j\}$ is a down-admissible set for $n$, then define the down-morphism $\delta_S:n\to n[S]$ by $\delta_S:=\delta_{s_0}\delta_{s_1}\cdots\delta_{s_j}$, and the corresponding up-morphism $\upsilon_S:n[S]\to n$ by $\upsilon_S=(\delta_S)^*$.
\end{deff}

\begin{rem}
A family of morphisms $\{G_n\in\End_{\tlcat}(n)\}_{n\geq 0}$ is said to be \textbf{left-aligned} if
\begin{equation*}
\begin{tikzpicture}[scale=0.25,mylabels,centered]
\draw[st] (1,-1) -- (1,5);
\draw[st] (3,-1) -- (3,5);
\draw[st] (5,-1) -- (5,5);
\draw[morg] (0,0) rectangle (6,2);
\node at (3,1) {$G_n$};
\draw[morg] (0,2) rectangle (4,4);
\node at (2,3) {$G_m$};
\end{tikzpicture}\;
=\;
\begin{tikzpicture}[scale=0.25,mylabels,centered]
\draw[st] (1,-3) -- (1,3);
\draw[st] (3,-3) -- (3,3);
\draw[st] (5,-3) -- (5,3);
\draw[morg] (0,0) rectangle (6,2);
\node at (3,1) {$G_n$};
\draw[morg] (0,0) rectangle (4,-2);
\node at (2,-1) {$G_m$};
\end{tikzpicture}\;
=\;
\begin{tikzpicture}[scale=0.25,mylabels,centered]
\draw[st] (1,-1) -- (1,3);
\draw[st] (3,-1) -- (3,3);
\draw[st] (5,-1) -- (5,3);
\draw[morg] (0,0) rectangle (6,2);
\node at (3,1) {$G_n$};
\end{tikzpicture}\;,
\end{equation*}
for all $1\leq m\leq n$. If $\{G_n\}$ is left aligned and $S=\{s_0,\ldots,s_j\}$ is an up-admissible stretch for some $n$ with $n+1=[n_k,n_{k-1},\ldots,n_0]_{\ell,\p}$ and writing $x=[n_k,\ldots,n_{s_k+1},0,\ldots,0]_{\ell,\p}-1$, the down-morphism $\delta_{S}=\delta_{s_0}\cdots\delta_{s_j}$ simplifies as
\begin{equation*}
\delta_S=\;
\begin{tikzpicture}[scale=0.25,mylabels,centered]
\draw[st] (1,-1) -- (1,4);
\draw[st] (3,-1) -- (3,2);
\draw[st] (5,-1) -- (5,2);
\draw[st] (7,-1) -- (7,4);
\draw[morg] (0,0) rectangle (4,2);
\node at (2,1) {$G_x$};
\tlcap{3,2}{4,3}{5,2};
\node at (4,3.7) {$S$};
\end{tikzpicture}\;,
\end{equation*}
where the label $S$ signifies that the cap contains $\sum_{s\in S}n_s\p^{(s)}$ strands.
\end{rem}

When $G_n=\jw_n$ for all $n$, the corresponding down- and up-morphisms will be denoted by $\ssdor_S$ and $\ssupr_S$, respectively. If $S$ is a down-admissible set for $n$, define
\begin{equation*}
\sslo_n^S:=\ssupr_S\,\jw_{n[S]}\,\ssdor_S=\;
\begin{tikzpicture}[scale=0.25,mylabels,centered]
\draw[st] (1,-1) -- (1,0);
\draw[st] (1,6) -- (1,7);
\draw[st] (3,-1) -- (3,7);
\draw[st] (5,-1) -- (5,7);
\draw[st] (7,-1) -- (7,0);
\draw[st] (7,6) -- (7,7);
\draw[morg] (0,0) -- (2,2) -- (6,2) -- (8,0) -- cycle;
\node at (4,1) {$\ssdor_S$};
\draw[jw] (2,2) rectangle (6,4);
\node at (4,3) {$n[S]$};
\draw[morg] (2,4) -- (6,4) -- (8,6) -- (0,6) -- cycle;
\node at (4,5) {$\ssupr_S$};
\end{tikzpicture}\;
\in\End_{\tlcat}(n).
\end{equation*}
These elements are orthogonal idempotents over $\Q(\delta)$. A careful selection of several of these idempotents allows us to define the following element, which actually has coefficients in $\Zm$ (see Remark~\ref{rem:specialisation}). Recall the definition of the \emph{ancestors} $a_{n,s}$ from Defintion~\ref{def:ancestors}.

\begin{deff}[{\cite[Definition 3.12]{Sutton2023}}]\label{jwlpzdef}
For any $n\geq 1$, the semisimple $(\ell,\p)$-Jones--Wenzl idempotent is an element of $\tlnzd$ defined by
\begin{equation*}
\jwlpz_n:=\sum_{m\in\supp{n}}\lambda_n^S\sslo_n^S,\quad\mbox{with}\quad
\lambda_n^S:=\prod_{s\in S}\frac{\qnd{a_{n,s-1}[S]+1}}{\qnd{a_{n,s}[S]+1}}\;,
\end{equation*}
where in each term $S$ is the unique down-admissible set such that $n[S]=m$. It is denoted by a yellow rectangle:
\begin{equation*}
\jwlpz_n=\begin{tikzpicture}[scale=0.25,mylabels,centered]
\draw[st] (1,-1) -- (1,3);
\draw[st] (3,-1) -- (3,3);
\draw[st] (5,-1) -- (5,3);
\draw[jwlpz] (0,0) rectangle (6,2);
\node at (3,1) {$n$};
\end{tikzpicture}\;.
\end{equation*}
\end{deff}

The following lemma from \cite{Sutton2023} gives another useful expression of the semisimple $(\ell,\p)$-Jones--Wenzl idempotent.

\begin{lem}\label{jwlpzexp}
The semisimple $(\ell,\p)$-Jones--Wenzl idempotent can be written as
\begin{equation*}
\jwlpz_n=\sum_{m\in\supp{\mo{n}}}\lambda_{\mo{n}}^{S'}\left(
\begin{tikzpicture}[scale=0.25,mylabels,centered]
\draw[st] (1,2) -- (1,3);
\draw[st] (3,2) -- (3,3);
\draw[st] (5,2) -- (5,3);
\draw[st] (7,2) -- (7,3);
\draw[st] (9,2) -- (9,3);
\draw[st] (11,2) -- (11,10);
\draw[morg] (0,3) -- (10,3) -- (8,5) -- (2,5) -- cycle;
\node at (5,4) {$\ssdor_{S'}$};
\draw[jw] (2,5) rectangle (12,7);
\node at (7,6) {$n[S']$};
\draw[morg] (2,7) -- (8,7) -- (10,9) -- (0,9) -- cycle;
\node at (5,8) {$\ssupr_{S'}$};
\draw[st] (1,9) -- (1,10);
\draw[st] (3,9) -- (3,10);
\draw[st] (5,9) -- (5,10);
\draw[st] (7,9) -- (7,10);
\draw[st] (9,9) -- (9,10);
\node at (13.2,3) {$a_t\p^{(t)}$};
\end{tikzpicture}
+\frac{\qnd{n[S'][t]+1}}{\qnd{\mo{n}[S']+1}}\;
\begin{tikzpicture}[scale=0.25,mylabels,centered]
\draw[st] (1,2) -- (1,3);
\draw[st] (3,2) -- (3,11);
\draw[st] (5,2) -- (5,11);
\draw[st] (7,2) -- (7,3);
\draw[st] (9,2) -- (9,3);
\draw[st] (11,2) -- (11,5);
\draw[morg] (0,3) -- (10,3) -- (8,5) -- (2,5) -- cycle;
\node at (5,4) {$\ssdor_{S'}$};
\draw[jw] (2,7) rectangle (6,9);
\node at (4,8) {$n[S'][t]$};
\tlcap{7,5}{9,7}{11,5};
\tlcup{7,11}{9,9}{11,11};
\begin{scope}[shift={(0,4)}]
\draw[morg] (2,7) -- (8,7) -- (10,9) -- (0,9) -- cycle;
\node at (5,8) {$\ssupr_{S'}$};
\draw[st] (1,9) -- (1,10);
\draw[st] (3,9) -- (3,10);
\draw[st] (5,9) -- (5,10);
\draw[st] (7,9) -- (7,10);
\draw[st] (9,9) -- (9,10);
\end{scope}
\draw[st] (11,11) -- (11,14);
\node at (13.2,4) {$a_t\p^{(t)}$};
\end{tikzpicture}\;
\right),
\end{equation*}
where $S'$ is the up-admissible set such that $m=\mo{n}[S']$ and $a_t$ is the first non-zero coefficient in the $(\ell,\p)$-expansion of $n$.
\end{lem}

\begin{rem}\label{rem:specialisation}
The coefficients $\lambda_n^S$, as well as many of the coefficients inside $\sslo_n^S$, are elements of $\Q(\delta)$ and might not descend to well-defined elements of $\kk$, nor do the idempotents $\lambda_n^S\sslo_n^S$ in general. It is nontrivial to show (see \cite[Theorem 3.18]{Sutton2023} or \cite[Proposition 3.5]{Martin2022}) that taking the whole sum actually gives an element of $\tlnzd$ and that the resulting coefficients can thus can be specialised to $\kk$. This is the crucial fact that allows us to define the following.
\end{rem}

\begin{deff}
For any integer $n\geq 1$, the $(\ell,\p)$-Jones--Wenzl idempotent $\jwlp_n$ is an element of $\tlnkd$ defined as the image of $\jwlpz_n$ under the specialisation morphism $\tlnzd\to\tlnkd$. It is denoted by a pink rectangle:
\begin{equation*}
\jwlp_n=\begin{tikzpicture}[scale=0.25,mylabels,centered]
\draw[st] (1,-1) -- (1,3);
\draw[st] (3,-1) -- (3,3);
\draw[st] (5,-1) -- (5,3);
\draw[jwlp] (0,0) rectangle (6,2);
\node at (3,1) {$n$};
\end{tikzpicture}\;.
\end{equation*}
\end{deff}

Keeping up with the notation set up above, when taking the family $G_n=\jwlpz_n$ (resp. $G_n=\jwlp_n$) for all $n$, the corresponding down- and up-morphisms will be denoted by $\zdor_S$ and $\zupr_S$ (resp. $\dor_S$ and $\upr_S$). We close this section with the following surprising behaviours of $(\ell,\p)$-Jones--Wenzl elements, respectively called \emph{non-classical absorption} and \emph{shortening} in \cite{Sutton2023}. They will be used throughout sections \ref{section:diag} and \ref{section:filt}.

\begin{prop}[{\cite[Proposition 3.19]{Sutton2023}}]\label{shortening}
Let $S$ be a down-admissible stretch for $m$. Then,
\begin{equation*}
\begin{tikzpicture}[scale=0.25,mylabels,centered]
\draw[jwlpz] (0,0) rectangle (8,2);
\node at (4,1) {$\jwlpz_m$};
\draw[st] (1,2)..controls{(1,3.3) and (3,2.7)}..(3,4);
\draw[st] (7,2)..controls{(7,3.3) and (5,2.7)}..(5,4);
\tlcap{3,2}{4,3}{5,2};
\node at (4,3.5) {$S$};
\draw[jwlpz] (2,4) rectangle (6,6);
\node at (4,5) {$\jwlpz_{m[S]}$};
\end{tikzpicture}\;
=\jwlpz_{m[S]}\zdor_S\jwlpz_m=\zdor_S\jwlpz_m=\;
\begin{tikzpicture}[scale=0.25,mylabels,centered]
\draw[jwlpz] (0,0) rectangle (8,2);
\node at (4,1) {$\jwlpz_m$};
\draw[st] (1,2)..controls{(1,3.3) and (3,2.7)}..(3,5);
\draw[st] (7,2)..controls{(7,3.3) and (5,2.7)}..(5,5);
\tlcap{3,2}{4,3}{5,2};
\node at (4,3.7) {$S$};
\end{tikzpicture}\;,
\end{equation*}
and
\begin{equation*}
\begin{tikzpicture}[scale=0.25,mylabels,centered]
\draw[jwlpz] (0,0) rectangle (4,2);
\node at (2,1) {$\jwlpz_a$};
\draw[st] (1,2)..controls{(1,3.3) and (3,2.7)}..(3,4);
\draw[st] (7,2)..controls{(7,3.3) and (5,2.7)}..(5,4);
\tlcap{3,2}{4,3}{5,2};
\draw[st] (5,0) -- (5,2);
\draw[st] (7,0) -- (7,2);
\node at (5.7,1.5) {$S$};
\draw[jwlpz] (2,4) rectangle (6,6);
\node at (4,5) {$\jwlpz_{m[S]}$};
\end{tikzpicture}\;
=\jwlpz_{m[S]}\zdor_S(\jwlpz_a\otimes \id_{m-a})=\zdor_S\jwlpz_m=\;
\begin{tikzpicture}[scale=0.25,mylabels,centered]
\draw[jwlpz] (0,0) rectangle (8,2);
\node at (4,1) {$\jwlpz_m$};
\draw[st] (1,2)..controls{(1,3.3) and (3,2.7)}..(3,5);
\draw[st] (7,2)..controls{(7,3.3) and (5,2.7)}..(5,5);
\tlcap{3,2}{4,3}{5,2};
\node at (4,3.7) {$S$};
\end{tikzpicture}\;,
\end{equation*}
where $a$ is the youngest ancestor of $m$ for which all digits indexed by elements of $S$ are zero. Applying the involution $*$ also gives the identities
\begin{equation*}
\jwlpz_{m(S)}\zupr_S\jwlpz_m=\jwlpz_{m(S)}\zupr_S,\quad (\jwlpz_a\otimes \id_{m(S)-a})\zupr_S\jwlpz_m=\jwlpz_{m(S)}\zupr_S,
\end{equation*}
where $S$ is an up-admissible stretch for $m$ and $a$ is the youngest ancestor of $m(S)$ for which all digits indexed by elements of $S$ are zero.
\end{prop}

Of course, the properties of Proposition~\ref{shortening} also hold when everything is specialised to $\tlnkd$.

\subsection{Truncation}\label{sect:truncation}
The Temperley--Lieb algebras nest within the Temperley--Lieb category and it is useful for computations, as will be clear in section~\ref{section:filt}, to study how cell modules compare over various $\TL_n$. This will be done using truncation functors, whose definition relies on the following elements.

\begin{deff}
Let $n,m\in\Z_{\geq 0}$ of the same parity with $m\leq n$. If $m\neq 0$, define morphisms in $\Hom_{\tlcat}(m,n)$ by
\begin{equation*}
\yy{m}:=\;
\begin{tikzpicture}[scale=0.25,mylabels,centered]
\tlcup{0,6}{1,5}{2,6};
\tlcup{4,6}{5,5}{6,6};
\node at (8,5.5) {$\cdots$};
\tlcup{10,6}{11,5}{12,6};
\draw[st] (14,6)..controls{(14,3.8) and (7,4.2)}..(7,1);
\draw[st] (16,6)..controls{(16,3.8) and (9,4.2)}..(9,1);
\draw[st] (20,6)..controls{(20,3.8) and (13,4.2)}..(13,1);
\node at (18,5.5) {$\cdots$};
\node at (11,1.5) {$\cdots$};
\end{tikzpicture}
\quad\mbox{and}\quad
\zz{m}:=\;
\begin{tikzpicture}[scale=0.25,mylabels,centered]
\tlcup{2,6}{3,5}{4,6};
\tlcup{6,6}{7,5}{8,6};
\node at (10,5.5) {$\cdots$};
\tlcup{12,6}{13,5}{14,6};
\draw[st] (0,6)..controls{(0,3.8) and (7,4.2)}..(7,1);
\draw[st] (16,6)..controls{(16,3.8) and (9,4.2)}..(9,1);
\draw[st] (20,6)..controls{(20,3.8) and (13,4.2)}..(13,1);
\node at (18,5.5) {$\cdots$};
\node at (11,1.5) {$\cdots$};
\end{tikzpicture}\;.
\end{equation*}
Note that both are elements of $\ctlmod{m}$, that $(\zz{m})^*\yy{m}=\id_m$ and thus $\ip{\zz{m},\yy{m}}=1$, and that $\yz{m}:=\yy{m}(\zz{m})^*$ is an idempotent.
If $m=0$ and $\delta\neq 0$, define
\begin{equation*}
\yy{0}:=\;
\begin{tikzpicture}[scale=0.25,mylabels,centered]
\tlcup{0,6}{1,5}{2,6};
\tlcup{4,6}{5,5}{6,6};
\node at (8,5.5) {$\cdots$};
\tlcup{10,6}{11,5}{12,6};
\end{tikzpicture}
\quad\mbox{and}\quad
\yz{0}:=\delta^{-n/2}\yy{0}(\yy{0})^*=\frac{1}{\delta^{n/2}}\;
\begin{tikzpicture}[scale=0.25,mylabels,centered]
\tlcup{0,6}{1,5}{2,6};
\tlcup{4,6}{5,5}{6,6};
\node at (8,5.5) {$\cdots$};
\tlcup{10,6}{11,5}{12,6};
\begin{scope}[shift={(0,9)},yscale=-1]
\tlcup{0,6}{1,5}{2,6};
\tlcup{4,6}{5,5}{6,6};
\node at (8,5.5) {$\cdots$};
\tlcup{10,6}{11,5}{12,6};
\end{scope}
\end{tikzpicture}\;,
\end{equation*}
so that $\yz{0}$ is again an idempotent:
\begin{equation*}
\yz{0}\cdot\yz{0}=\frac{1}{\delta^n}\;
\begin{tikzpicture}[scale=0.25,mylabels,centered]
\tlcup{0,6}{1,5}{2,6};
\tlcup{4,6}{5,5}{6,6};
\node at (8,5.5) {$\cdots$};
\tlcup{10,6}{11,5}{12,6};
\node at (8,3) {$\cdots$};
\begin{scope}[shift={(0,9)},yscale=-1]
\tlcup{0,6}{1,5}{2,6};
\tlcup{4,6}{5,5}{6,6};
\tlcup{10,6}{11,5}{12,6};
\end{scope}
\begin{scope}[shift={(0,6)},yscale=-1]
\tlcup{0,6}{1,5}{2,6};
\tlcup{4,6}{5,5}{6,6};
\node at (8,5.5) {$\cdots$};
\tlcup{10,6}{11,5}{12,6};
\begin{scope}[shift={(0,9)},yscale=-1]
\tlcup{0,6}{1,5}{2,6};
\tlcup{4,6}{5,5}{6,6};
\tlcup{10,6}{11,5}{12,6};
\end{scope}
\end{scope}
\end{tikzpicture}\;
=\frac{\delta^{n/2}}{\delta^n}\;
\begin{tikzpicture}[scale=0.25,mylabels,centered]
\tlcup{0,6}{1,5}{2,6};
\tlcup{4,6}{5,5}{6,6};
\node at (8,5.5) {$\cdots$};
\tlcup{10,6}{11,5}{12,6};
\begin{scope}[shift={(0,9)},yscale=-1]
\tlcup{0,6}{1,5}{2,6};
\tlcup{4,6}{5,5}{6,6};
\node at (8,5.5) {$\cdots$};
\tlcup{10,6}{11,5}{12,6};
\end{scope}
\end{tikzpicture}\;=\yz{0}.
\end{equation*}
\end{deff}

\begin{rem}
Note that if $\delta=0$, then $\rad{\ctlmod{0}}=\ctlmod{0}$, so that $0\notin\cellposet_0$. It follows that the idempotent $\yz{m}$ is well-defined for any $m\in\cellposet_0$.
\end{rem}

The idempotents $\yz{m}$ then allow us to construct isomorphisms of algebras by ``truncating'' a bigger Temperley--Lieb algebra to a smaller one.

\begin{lem}[{\cite[Lemma~5.1]{Spencer2023}}]
The idempotent $\yz{m}$ induces an isomorphism of algebras $\TL_m\cong\yz{m}\TL_n\yz{m}$ defined by $u\mapsto \yy{m}u(\zz{m})^*$.
\end{lem}

Letting $\yz{m}$ act on $\TL_n$-modules leads to the following definition of a truncation functor between categories of modules over different Temperley--Lieb algebras.

\begin{deff}[\cite{Spencer2023}]
Given integers $n,m\in\Z_{\geq 0}$ of the same parity with $m\leq n$, let $\trunc{m}$ be the functor from the category of $\TL_n$-modules to the category of $\TL_m$-modules defined on objects by $\trunc{m}(M)=\yz{m}M$, with $\TL_m$-action on an element $\yz{m}x\in\yz{m}M$ given by $u\cdot \yz{m}x=\yy{m}u(\zz{m})^*x$, and defined on morphisms by restriction.
\end{deff}

\begin{prop}[{\cite[Lemmas~5.2~and~5.3]{Spencer2023}}]\label{truncation}
The functor $\trunc{m}$ is exact, it preserves cell modules and simple modules, and it refines composition series. This means that $\trunc{m}(\ctlmod{r})=\ctlmod[m]{r}$ and $\trunc{m}(\stlmod{r})=\stlmod[m]{r}$, where it is understood that if $m<r$, $\ctlmod[m]{r}=\stlmod[m]{r}=0$.
\end{prop}

In particular, this shows that the submodule structure of cell modules $\ctlmod{m}$ does not depend on $n$ in the following sense: if $n<n'$, the list of composition factors $\stlmod{m(S)}$ appearing in $\ctlmod[n]{m}$ and $\ctlmod[n']{m}$ may be different (since they must satisfy $m(S)\leq n$ or $m(S)\leq n'$ respectively), but the lattice of submodules of $\ctlmod[n]{m}$ is obtained from that of $\ctlmod[n']{m}$ simply by removing the composition factors that are too large, keeping everything else intact.

\begin{ex}
It will be shown in Section~\ref{section:diag} that the Alperin diagram of $\ctlmod[267]{21}$ is as pictured on the left below. Applying the truncation functor $\trunc[267]{71}$ gives the Alperin diagram of $\ctlmod[71]{21}$ on the right.

\begin{equation*}
\begin{tikzpicture}[centered]
\node[inner sep=5pt, red] (21) at (0,0) {$21$};
\node[inner sep=5pt, red] (41) at (-2,-1) {$41$};
\node[inner sep=5pt, red] (27) at (0,-1) {$27$};
\node[inner sep=5pt] (81) at (2,-1) {$81$};
\node[inner sep=5pt, red] (37) at (-2,-2) {$37$};
\node[inner sep=5pt, red] (71) at (0,-2) {$71$};
\node[inner sep=5pt] (87) at (2,-2) {$87$};
\node[inner sep=5pt] (261) at (4,-2) {$261$};
\node[inner sep=5pt, red] (67) at (0,-3) {$67$};
\node[inner sep=5pt] (251) at (2,-3) {$251$};
\node[inner sep=5pt] (267) at (4,-3) {$267$};
\node[inner sep=5pt] (247) at (2,-4) {$247$};

\draw[thick, red] (21) -- (41);
\draw[thick, red] (21) -- (27);
\draw[thick] (21) -- (81);
\draw[thick, red] (41) -- (37);
\draw[thick, red] (41) -- (71);
\draw[thick, red] (27) -- (37);
\draw[thick] (27) -- (87);
\draw[thick] (81) -- (71);
\draw[thick] (81) -- (87);
\draw[thick] (81) -- (261);
\draw[thick, red] (37) -- (67);
\draw[thick, red] (71) -- (67);
\draw[thick] (71) -- (251);
\draw[thick] (87) -- (67);
\draw[thick] (87) -- (267);
\draw[thick] (261) -- (251);
\draw[thick] (261) -- (267);
\draw[thick] (67) -- (247);
\draw[thick] (251) -- (247);
\draw[thick] (267) -- (247);

\node at (1,-5) {$\ctlmod[267]{21}$};
\end{tikzpicture}\quad\rightsquigarrow\quad
\begin{tikzpicture}[centered]
\node[inner sep=5pt] (21) at (0,0) {$21$};
\node[inner sep=5pt] (41) at (-2,-1) {$41$};
\node[inner sep=5pt] (27) at (0,-1) {$27$};
\node[inner sep=5pt] (37) at (-2,-2) {$37$};
\node[inner sep=5pt] (71) at (0,-2) {$71$};
\node[inner sep=5pt] (67) at (0,-3) {$67$};

\draw[thick] (21) -- (41);
\draw[thick] (21) -- (27);
\draw[thick] (41) -- (37);
\draw[thick] (41) -- (71);
\draw[thick] (27) -- (37);
\draw[thick] (37) -- (67);
\draw[thick] (71) -- (67);

\node at (-1,-4) {$\ctlmod[71]{21}$};
\end{tikzpicture}
\end{equation*}
\end{ex}

The isomorphism $\phi:\trunc{m}\ctlmod{r}\to\ctlmod[m]{r}$ is explicitly given by $\phi(\yz{m}u):=(\zz{m})^*u$ and it is not hard to see that the map $\psi:\ctlmod[m]{r}\to\trunc{m}\ctlmod{r}:v\mapsto\yy{m}v$ is its inverse. These explicit maps can be used to show that truncation functors preserve certain Jantzen-like filtrations of cell modules defined using their corresponding bilinear forms. This will be done in Proposition~\ref{propsublattice}. Of crucial importance will be the following generalised version of contravariance.

\begin{lem}\label{contravariance}
Let $u\in\ctlmod{r}$, $v\in\ctlmod[m]{r}$, and $x$ be any morphism in $\Hom_{\tlcat}(n,m)$. Then, $\ip{xu,v}=\ip{u,x^*v}$, where the left side is computed using the bilinear form in $\ctlmod{r}$, while the right side takes place in $\ctlmod[m]{r}$.
\end{lem}
\begin{proof}
Both values are defined as the coefficient in front of the identity diagram $\id_r$ in the same element $u^*x^*v\in\Hom_{\tlcat}(r,r)$.
\end{proof}

%% file: diagrammatic_new.tex
Suppose that $\kk$ is any field of characteristic $\p$ and let $\delta\in\kk$. This section describes the lattice of submodules of the cell modules $\ctlmod{m}$ in mixed characteristic. The proof is entirely diagrammatic and does not rely on Schur--Weyl duality, but some key steps are motivated by statements about the category of tilting modules for $\uq$, the quantum group of $\sltwo$. In fact, in the case where $\kk$ is algebraically closed and $\delta$ is of the form $q+q^{-1}$ for some $q\in\kk^\times$, the results of this section apply to the corresponding endomorphism algebras of the tilting modules $\tmod{1}^{\otimes n}$ and their cell modules coming from the cellular structure developped in \cite{Andersen2015}. Throughout, we will highlight the analogies between the two settings.

Recall the idempotents $\yz{m}$ from Section~\ref{sect:truncation}, constructed using the elements $\yy{m}$ and $\zz{m}$. The homomorphism criterion from \cite[Proposition 4.8]{Andersen2015} can be rewritten in Temperley--Lieb language in the following way.
\begin{prop}[Homomorphism criterion, diagrammatic version]\label{homcrit}
Let $m\in\cellposet_0$. There is an isomorphism of $\kk$-vector spaces given by
\begin{equation*}
  \Hom_{\TL_n^\kk(\delta)}(\cmod{m},M)\cong\{v\in M\mid \td(<m)\cdot v=0 \mbox{ and }\yz{m}v=v\}.
\end{equation*}
\end{prop}
\begin{proof}
Let $\phi:\cmod{m}\to M$ be a morphism of $\TL_n^\kk(\delta)$-modules. Then, $\td(<m)\cdot\phi(\yy{m})=0$ because $\td(<m)\cdot\yy{m}=0$ in $\cmod{m}$. Also, $\yz{m}\cdot\phi(\yy{m})=\phi(\yz{m}\yy{m})=\phi(\yy{m})$, so $\phi(\yy{m})$ is an element of the right-hand side.

In the other direction, given an element $v$ of the right-hand side, define $\phi_v:\cmod{m}\to M$ by $\phi_v(c)=c(\zz{m})^*v$. This is clearly a $\tlnkd$-morphism, since $x\phi_v(c)=xc(\zz{m})^*v=\phi_v(xc)$ for any $x\in\tlnkd$.

Finally, note that $\phi_v(\yy{m})=\yy{m}(\zz{m})^*v=\yz{m}v=v$, so the map $v\mapsto\phi_v$ is a $\kk$-linear inverse of $\phi\mapsto\phi(\yy{m})$ and the result follows.
\end{proof}

Using this analogue of the homomorphism criterion, we can explicitly identify elements of the Temperley--Lieb category generating a submodule of a cell module with any given simple head.

\begin{prop}\label{genelem}
Fix $n,m\in\Z_{\geq 0}$ of the same parity with $n\geq m$ and consider the module $\ctlmod{m}$. Let $S$ be an up-admissible set for $m$, and define
\begin{equation*}
\vv[m(S)]{t}:=\yy{m(S)}\jwlp_{m(S)}\upr_S\jwlp_m=
\begin{tikzpicture}[scale=0.25,mylabels,centered]
\draw[st] (-1,4) -- (-1,5);
\draw[st] (1,4) -- (1,5);
\draw[st] (3,4) -- (3,5);
\draw[st] (5,4) -- (5,5);
\draw[st] (7,4) -- (7,5);
\draw[morg] (-2,4) -- (0,2) -- (6,2) -- (8,4) -- cycle;
\node at (3,3) {$\yy{m(S)}$};
\draw[jwlp] (0,0) rectangle (6,2);
\node at (3,1) {$m(S)$};
\draw[morg] (1,-2) -- (5,-2) -- (6,0) -- (0,0) -- cycle;
\node at (3,-1) {$\upr_S$};
\draw[jwlp] (1,-2) rectangle (5,-4);
\node at (3,-3) {$m$};
\draw[st] (2,-5) -- (2,-4);
\draw[st] (4,-5) -- (4,-4);
\end{tikzpicture}
\;
\in\Hom_{\tlcat}(m,n).
\end{equation*} 
Then, $\vv[m(S)]{}$ has a nonzero image in the cell module $\cmod{m}$, that we still write as $\vv[m(S)]{}$ by abuse of notation, and it generates a submodule having the composition factor $\stlmod{m(S)}$ as its head.
\end{prop}
\begin{proof}
Let $\NN{m(S)}{t}$ be the submodule $(\td(<m(S))\cdot \vv[m(S)]{t})<\ctlmod{m}$ and write $\overline{\vv[m(S)]{t}}$ for the image of $\vv[m(S)]{t}$ in the quotient $\ctlmod{m}/\NN{m(S)}{t}$. We have $\td(<m(S))\cdot \overline{\vv[m(S)]{t}}=0$ and
\begin{equation*}
\yz{m(S)}\overline{\vv[m(S)]{t}}=\ip{\zz{m(S)},\yy{m(S)}}\overline{\vv[m(S)]{t}}=\overline{\vv[m(S)]{t}},
\end{equation*}
so Proposition~\ref{homcrit} gives a $\tlnkd$-morphism $\phi:\ctlmod{m(S)}\to\ctlmod{m}/\NN{m(S)}{t}$. Note that
$\Ima{\phi}\cong\ip{\vv[m(S)]{t}}/\NN{m(S)}{t}$
has head isomorphic to $\stlmod{m(S)}$, hence the submodule of $\ctlmod{m}$ generated by $\vv[m(S)]{}$ contains $\stlmod{m(S)}$ in its head. This is the only composition factor of $\ctlmod{m}$ isomorphic to $\stlmod{m(S)}$, since every composition factor of appears with multiplicity one. This shows more generally that every composition factor $\stlmod{m(S)}$ of $\cmod{m}$ appears in the head of the submodule generated by a choice of some element $v$ with $\overline{v}=\overline{\vv[m(S)]{t}}$ in $\ctlmod{m}/\NN{m(S)}{t}$.
\end{proof}

The element $\vv[m(S)]{}$ should be seen as the diagrammatic incarnation of a $\uq$-morphism between tilting modules $\tmod{m}\to\tmod{m(S)}$, followed by an inclusion in $\tmod{1}^{\otimes n}$, and giving a non-zero morphism when pre-composed with the inclusion $\wmod{m}\xhookrightarrow{} \tmod{m}$.

Given two up-admissible sets $S,S'$ for $m$, Proposition~\ref{genelem} implies that the submodule with head isomorphic to $\stlmod{m(S)}$ is contained in the submodule with head isomorphic to $\stlmod{m(S')}$ if and only if $\vv[m(S)]{}=u\cdot\vv[m(S')]{}$ for some $0\neq u\in\tlnkd$. The study of the submodule structure of a given cell module $\ctlmod{m}$ can thus be reduced to studying morphisms, in the Temperley--Lieb category, between $(\ell,\p)$-Jones--Wenzl elements. These morphisms are the diagrammatic versions of $\uq$-morphisms between tilting modules, and have been investigated in \cite{Sutton2023}. Proposition~\ref{stwztheo} recalls some of their main results that will be needed for the rest of this section.

\begin{prop}[{\cite[Theorem 3.21]{Sutton2023}}]\label{stwztheo}
A basis for $$\Hom_{\tlcat}(\jwlp_m,\jwlp_l):=\{\jwlp_l\cdot x\cdot \jwlp_m\mid x\in\Hom_{\tlcat}(m,l)\}$$ is given by the set
\begin{equation*}
\{\jwlp_l \upr_{S'} \dor_{S} \jwlp_m \mid S, S'\mbox{ are down-admissible for } m,l\mbox{ resp. and } m[S]=l[S']\}.
\end{equation*}
Moreover, if $S,S'$ are minimal down- and up-admissible stretches for $m$, then
the following relations hold:
\begin{enumerate}
\item \textbf{Far-commutativity.} If $d(S,S')>1$,
\begin{gather*}
\dor_S\dor_{S'}\jwlp_{m}=\dor_{S'}\dor_S\jwlp_{m},\quad \dor_S\upr_{S'}\jwlp_m=\upr_{S'}\dor_S\jwlp_{m},\\ \upr_S\upr_{S'}\jwlp_m=\upr_{S'}\upr_S\jwlp_{m}.
\end{gather*}
\item \textbf{(Some) adjacency relations.} If $d(S,S')=1$ and $S'>S$,
\begin{gather*}
\dor_{S'}\upr_S\jwlp_m=\dor_{S\cup S'}\jwlp_m,\quad \dor_S\upr_{S'}\jwlp_m=\upr_{S'\cup S}\jwlp_m.
\end{gather*}
\end{enumerate}
\end{prop}
\begin{rem}
Note that the relations in Proposition~\ref{stwztheo} are just a small sample of the complete relations provided in \cite{Sutton2023}.
\end{rem}

These relations can be used to show that morphisms of the form $\jwlp_{m(S)}\upr_S\jwlp_m$ factor through $\jwlp_{m(S')}\upr_{S'}\jwlp_m$ in a way that respects containment of up-admissible sets $S,S'$. This provides the key inductive step to prove the main theorem about the submodule structure of cell modules.

\begin{prop}\label{fact}
Let $S_1,S_2$ be up-admissible sets for the integer $m$ such that $S_1\subset S_2$ and $\abs{S_2}=\abs{S_1}+1$. Then the morphism $\jwlp_{m(S_2)}\upr_{S_2}\jwlp_m$ factors through $\jwlp_{m(S_1)}\upr_{S_1}\jwlp_m$.
\end{prop}
\begin{proof}
Write $S_1$ as a disjoint union of minimal up-admissible stretches $S_1=\sqcup_{i=1}^l S_{1,i}$, where $S_{1,i}>S_{1,i-1}$. Let $S_2=S_1\cup \{k\}$. There are three possibilities:
\begin{enumerate}
\item If $k>S_1$, then, by definition,
\begin{align*}
\jwlp_{m(S_2)}\upr_{S_2}\jwlp_m&=\jwlp_{m(S_2)}\upr_{\{k\}}\upr_{S_1}\jwlp_m\\
&= \jwlp_{m(S_2)}\upr_{\{k\}}\jwlp_{m(S_1)}\upr_{S_1}\jwlp_m.
\end{align*}
\item If $i$ is the least index such that $k<S_{1,i}$ and $d(\{k\},S_{1,i})>1$, then far-commutativity gives
\begin{align*}
\jwlp_{m(S_2)}\upr_{S_2}\jwlp_m &= \jwlp_{m(S_2)}\upr_{S_{1,l}}\cdots\upr_{S_{1,i}}\upr_{\{k\}}\upr_{S_{1,i-1}}\cdots\upr_{S_{1,1}}\jwlp_m\\
&= \jwlp_{m(S_2)}\upr_{\{k\}}\upr_{S_{1,l}}\cdots\upr_{S_{1,i}}\upr_{S_{1,i-1}}\cdots\upr_{S_{1,1}}\jwlp_m\\
&= \jwlp_{m(S_2)}\upr_{\{k\}}\jwlp_{m(S_1)}\upr_{S_{1,l}}\cdots\upr_{S_{1,i}}\upr_{S_{1,i-1}}\cdots\upr_{S_{1,1}}\jwlp_m.
\end{align*}
\item If $i$ is the least index such that $k<S_{1,i}$ and $d(\{k\},S_{1,i})=1$, then
\begin{align*}
\jwlp_{m(S_2)}\upr_{S_2}\jwlp_m &= \jwlp_{m(S_2)}\upr_{S_{1,l}}\cdots\upr_{S_{1,i}\cup \{k\}}\upr_{S_{1,i-1}}\cdots\upr_{S_{1,1}}\jwlp_m\\
&= \jwlp_{m(S_2)}\upr_{S_{1,l}}\cdots\dor_{\{k\}}\upr_{S_{1,i}}\upr_{S_{1,i-1}}\cdots\upr_{S_{1,1}}\jwlp_m\\
&= \jwlp_{m(S_2)}\dor_{\{k\}}\upr_{S_{1,l}}\cdots\upr_{S_{1,i}}\upr_{S_{1,i-1}}\cdots\upr_{S_{1,1}}\jwlp_m\\
&= \jwlp_{m(S_2)}\dor_{\{k\}}\jwlp_{m(S_1)}\upr_{S_{1,l}}\cdots\upr_{S_{1,i}}\upr_{S_{1,i-1}}\cdots\upr_{S_{1,1}}\jwlp_m,
\end{align*}
where the second line follows from one of the adjacency relations and the third line follows from far-commutativity.
\end{enumerate}
\end{proof}

With that in hand, we are ready to prove the main theorem.

\begin{theo}\label{substruct}
The lattice of submodules of the cell module $\ctlmod{m}$ is given by the lattice of the poset $\{S\mid S\mbox{ is up-admissible for m and } m(S)\leq n\}$ with the reverse-inclusion order.
\end{theo}
\begin{proof}
It is already known that the elements of the poset correspond to the simple composition factors of $\ctlmod{m}$, as stated in Lemma~\ref{compfac}. Using Proposition~\ref{genelem} and the discussion after it, to prove the theorem it is then sufficient to show that whenever $S_1\subset S_2$, we can write $\vv[m(S_2)]{}=u\cdot \vv[m(S_1)]{}$ for some $0\neq u\in\tlnkd$.

First, suppose that $\abs{S_2}=\abs{S_1}+1$ and that $S_2=S_1\cup\{k\}$. Using the proof of Proposition~\ref{fact}, the element $\vv[m(S_2)]{}$ may be rewritten as 
\begin{equation*}
\vv[m(S_2)]{}=
\begin{cases}
\yy{m(S_2)}\jwlp_{m(S_2)}\upr_{\{k\}}\jwlp_{m(S_1)}\upr_{S_1}\jwlp_m &\mbox{if } m(S_1)<m(S_2),\\
\yy{m(S_2)}\jwlp_{m(S_2)}\dor_{\{k\}}\jwlp_{m(S_1)}\upr_{S_1}\jwlp_m &\mbox{if } m(S_1)>m(S_2),
\end{cases}
\end{equation*}
depending on where $k$ lies in $S_2$. In the former case, the result follows from the fact that
\begin{align*}
\vv[m(S_2)]{}&=
\yy{m(S_2)}\jwlp_{m(S_2)}\upr_{\{k\}}\jwlp_{m(S_1)}\upr_{S_1}\jwlp_m\\
&=\yy{m(S_2)}\jwlp_{m(S_2)}\upr_{\{k\}}\jwlp_{m(S_1)}(\zz{m(S_1)})^*\yy{m(S_1)}\jwlp_{m(S_1)}\upr_{S_1}\jwlp_m\\
&=(\yy{m(S_2)}\jwlp_{m(S_2)}\upr_{\{k\}}\jwlp_{m(S_1)}(\zz{m(S_1)})^*)\vv[m(S_1)]{},
\end{align*}
where $\yy{m(S_2)}\jwlp_{m(S_2)}\upr_{\{k\}}\jwlp_{m(S_1)}(\zz{m(S_1)})^*$ is an element of $\tlnkd$. Replacing the morphisms $\upr_{\{k\}}$ by $\dor_{\{k\}}$ in the above shows the result for the latter case. In pictures, this means
\begin{equation*}
\vv[m(S_2)]{}=\;
\begin{tikzpicture}[scale=0.25,mylabels,centered]
\draw[st] (-1,4) -- (-1,5);
\draw[st] (1,4) -- (1,5);
\draw[st] (3,4) -- (3,5);
\draw[st] (5,4) -- (5,5);
\draw[st] (7,4) -- (7,5);
\draw[morg] (-2,4) -- (0,2) -- (6,2) -- (8,4) -- cycle;
\node at (3,3) {$\yy{m(S_2)}$};
\draw[jwlp] (0,0) rectangle (6,2);
\node at (3,1) {$m(S_2)$};
\draw[morg] (1,-2) -- (5,-2) -- (6,0) -- (0,0) -- cycle;
\node at (3,-1) {$\upr_{S_2}$};
\draw[jwlp] (1,-2) rectangle (5,-4);
\node at (3,-3) {$m$};
\draw[st] (2,-5) -- (2,-4);
\draw[st] (4,-5) -- (4,-4);
\end{tikzpicture}\;
=\;
\begin{tikzpicture}[scale=0.25,mylabels,centered]
\draw[st] (-1,8.5) -- (-1,9.5);
\draw[st] (1,8.5) -- (1,9.5);
\draw[st] (3,8.5) -- (3,9.5);
\draw[st] (5,8.5) -- (5,9.5);
\draw[st] (7,8.5) -- (7,9.5);
\draw[morg] (-2,8.5) -- (0,6.5) -- (6,6.5) -- (8,8.5) -- cycle;
\node at (3,7.5) {$\yy{m(S_2)}$};
\draw[jwlp] (0,4.5) rectangle (6,6.5);
\draw[jwlp] (0.5,0) rectangle (5.5,2);
\node at (3,1) {$m(S_1)$};
\draw[st] (1.5,2)..controls{(1.5,3.45) and (0.5,3.05)}..(0.5,4.5);
\draw[st] (4.5,2)..controls{(4.5,3.45) and (5.5,3.05)}..(5.5,4.5);
\tlcup{2,4.5}{3,3.5}{4,4.5};
\node at (3,2.85) {$\{k\}$};
\node at (3,5.5) {$m(S_2)$};
\draw[morg] (1,-2) -- (5,-2) -- (5.5,0) -- (0.5,0) -- cycle;
\node at (3,-1) {$\upr_{S_1}$};
\draw[jwlp] (1,-2) rectangle (5,-4);
\node at (3,-3) {$m$};
\draw[st] (2,-5) -- (2,-4);
\draw[st] (4,-5) -- (4,-4);
\end{tikzpicture}\;
=\;
\begin{tikzpicture}[scale=0.25,mylabels,centered]
\draw[st] (-1,8.5) -- (-1,9.5);
\draw[st] (1,8.5) -- (1,9.5);
\draw[st] (3,8.5) -- (3,9.5);
\draw[st] (5,8.5) -- (5,9.5);
\draw[st] (7,8.5) -- (7,9.5);
\draw[morg] (-2,8.5) -- (0,6.5) -- (6,6.5) -- (8,8.5) -- cycle;
\node at (3,7.5) {$\yy{m(S_2)}$};
\draw[jwlp] (0,4.5) rectangle (6,6.5);
\draw[jwlp] (0.5,0) rectangle (5.5,2);
\node at (3,1) {$m(S_1)$};
\draw[morg] (1,-8) -- (5,-8) -- (5.5,-6) -- (0.5,-6) -- cycle;
\node at (3,-7) {$\upr_{S_1}$};
\draw[jwlp] (1,-8) rectangle (5,-10);
\node at (3,-9) {$m$};
\draw[st] (2,-11) -- (2,-10);
\draw[st] (4,-11) -- (4,-10);
\draw[st] (1.5,2)..controls{(1.5,3.45) and (0.5,3.05)}..(0.5,4.5);
\draw[st] (4.5,2)..controls{(4.5,3.45) and (5.5,3.05)}..(5.5,4.5);
\tlcup{2,4.5}{3,3.5}{4,4.5};
\node at (3,2.85) {$\{k\}$};
\node at (3,5.5) {$m(S_2)$};
\draw[jwlp] (0.5,-6) rectangle (5.5,-4);
\node at (3,-5) {$m(S_1)$};
\draw[morg] (0.5,-4) -- (5.5,-4) -- (8,-2) -- (-2,-2) -- cycle;
\draw[morg] (-2,-2) -- (8,-2) -- (5.5,0) -- (0.5,0) -- cycle;
\node at (3,-3) {$\yy{m(S_1)}$};
\node at (3,-1) {$(\zz{m(S_1)})^*$};
\draw[decorate,decoration={brace,mirror}] (9,-1.9) -- (9,9.5);
\draw[decorate,decoration={brace,mirror}] (9,-11) -- (9,-2.1);
\node at (12,-6.5) {$\vv[m(S_1)]{}$};
\node at (12,3.75) {$\in\tlnkd$};
\end{tikzpicture}\;,
\end{equation*}
and
\begin{equation*}
\vv[m(S_2)]{}=\;
\begin{tikzpicture}[scale=0.25,mylabels,centered]
\draw[st] (-1,4) -- (-1,5);
\draw[st] (1,4) -- (1,5);
\draw[st] (3,4) -- (3,5);
\draw[st] (5,4) -- (5,5);
\draw[st] (7,4) -- (7,5);
\draw[morg] (-2,4) -- (0,2) -- (6,2) -- (8,4) -- cycle;
\node at (3,3) {$\yy{m(S_2)}$};
\draw[jwlp] (0,0) rectangle (6,2);
\node at (3,1) {$m(S_2)$};
\draw[morg] (1,-2) -- (5,-2) -- (6,0) -- (0,0) -- cycle;
\node at (3,-1) {$\upr_{S_2}$};
\draw[jwlp] (1,-2) rectangle (5,-4);
\node at (3,-3) {$m$};
\draw[st] (2,-5) -- (2,-4);
\draw[st] (4,-5) -- (4,-4);
\end{tikzpicture}\;
=\;
\begin{tikzpicture}[scale=0.25,mylabels,centered]
\draw[st] (-1,8.5) -- (-1,9.5);
\draw[st] (1,8.5) -- (1,9.5);
\draw[st] (3,8.5) -- (3,9.5);
\draw[st] (5,8.5) -- (5,9.5);
\draw[st] (7,8.5) -- (7,9.5);
\draw[morg] (-2,8.5) -- (0,6.5) -- (6,6.5) -- (8,8.5) -- cycle;
\node at (3,7.5) {$\yy{m(S_2)}$};
\draw[jwlp] (0,4.5) rectangle (6,6.5);
\draw[jwlp] (-1.5,0) rectangle (7.5,2);
\node at (3,1) {$m(S_1)$};
\draw[morg] (1,-2) -- (5,-2) -- (7.5,0) -- (-1.5,0) -- cycle;
\node at (3,-1) {$\upr_{S_1}$};
\draw[jwlp] (1,-2) rectangle (5,-4);
\node at (3,-3) {$m$};
\draw[st] (2,-5) -- (2,-4);
\draw[st] (4,-5) -- (4,-4);
\draw[st] (-1,2)..controls{(-1,3.45) and (1,3.05)}..(1,4.5);
\draw[st] (1,2)..controls{(1,3.45) and (3,3.05)}..(3,4.5);
\draw[st] (7,2)..controls{(7,3.45) and (5,3.05)}..(5,4.5);
\tlcap{3,2}{4,3}{5,2};
\node at (4,3.6) {$\{k\}$};
\node at (3,5.5) {$m(S_2)$};
\end{tikzpicture}\;=\;
\begin{tikzpicture}[scale=0.25,mylabels,centered]
\draw[st] (-1,8.5) -- (-1,9.5);
\draw[st] (1,8.5) -- (1,9.5);
\draw[st] (3,8.5) -- (3,9.5);
\draw[st] (5,8.5) -- (5,9.5);
\draw[st] (7,8.5) -- (7,9.5);
\draw[morg] (-2,8.5) -- (0,6.5) -- (6,6.5) -- (8,8.5) -- cycle;
\node at (3,7.5) {$\yy{m(S_2)}$};
\draw[jwlp] (0,4.5) rectangle (6,6.5);
\draw[jwlp] (-1.5,0) rectangle (7.5,2);
\node at (3,1) {$m(S_1)$};
\draw[morg] (-1.5,0) -- (7.5,0) -- (8,-2) -- (-2,-2) -- cycle;
\draw[morg] (-2,-2) -- (8,-2) -- (7.5,-4) -- (-1.5,-4) -- cycle;
\node at (3,-1) {$(\zz{m(S_1)})^*$};
\node at (3,-3) {$\yy{m(S_1)}$};
\draw[jwlp] (-1.5,-4) rectangle (7.5,-6);
\node at (3,-5) {$m(S_1)$};
\draw[morg] (1,-8) -- (5,-8) -- (7.5,-6) -- (-1.5,-6) -- cycle;
\node at (3,-7) {$\upr_{S_1}$};
\draw[jwlp] (1,-8) rectangle (5,-10);
\node at (3,-9) {$m$};
\draw[st] (2,-11) -- (2,-10);
\draw[st] (4,-11) -- (4,-10);
\draw[st] (-1,2)..controls{(-1,3.45) and (1,3.05)}..(1,4.5);
\draw[st] (1,2)..controls{(1,3.45) and (3,3.05)}..(3,4.5);
\draw[st] (7,2)..controls{(7,3.45) and (5,3.05)}..(5,4.5);
\tlcap{3,2}{4,3}{5,2};
\node at (4,3.6) {$\{k\}$};
\node at (3,5.5) {$m(S_2)$};
\draw[decorate,decoration={brace,mirror}] (9,-1.9) -- (9,9.5);
\draw[decorate,decoration={brace,mirror}] (9,-11) -- (9,-2.1);
\node at (12,-6.5) {$\vv[m(S_1)]{}$};
\node at (12,3.75) {$\in\tlnkd$};
\end{tikzpicture}\;,
\end{equation*}
respectively.

For arbitrary up-admissible sets $S_1\subset S_2$, assuming that $n$ is large enough and using Lemma~\ref{indadmsets}, one can apply the above procedure for every intermediate step and the result follows. Since the submodule structure does not depend on $n$, as shown in Section~\ref{sect:truncation}, the theorem follows.
\end{proof}

\begin{ex}
Taking $\kk=\Z[\delta]/(\p,m_\delta)=\Z[\delta]/(3,\delta^2+\delta-1)$, so that $\ell=5$ and $\p=3$, consider the cell module $\ctlmod[267]{21}$ for $\TL_{267}$. Then,
$$21+1=[2,1,1]_{5,3},$$
and the up-admissible sets $S$ for $21$ such that $21(S)\leq 267$ are given in the Hasse diagram below.
\begin{equation*}
\begin{tikzpicture}[centered]
\node[inner sep=2pt] (21) at (0,0) {$\varnothing$};
\node[inner sep=2pt] (41) at (-2,-1) {$\{1\}$};
\node[inner sep=2pt] (27) at (0,-1) {$\{0\}$};
\node[inner sep=2pt] (81) at (2,-1) {$\{2\}$};
\node[inner sep=2pt] (37) at (-2,-2) {$\{0,1\}$};
\node[inner sep=2pt] (71) at (0,-2) {$\{1,2\}$};
\node[inner sep=2pt] (87) at (2,-2) {$\{0,2\}$};
\node[inner sep=2pt] (261) at (4,-2) {$\{2,3\}$};
\node[inner sep=2pt] (67) at (0,-3) {$\{0,1,2\}$};
\node[inner sep=2pt] (251) at (2,-3) {$\{1,2,3\}$};
\node[inner sep=2pt] (267) at (4,-3) {$\{0,2,3\}$};
\node[inner sep=2pt] (247) at (2,-4) {$\{0,1,2,3\}$};

\draw[thick] (21) -- (41);
\draw[thick] (21) -- (27);
\draw[thick] (21) -- (81);
\draw[thick] (41) -- (37);
\draw[thick] (41) -- (71);
\draw[thick] (27) -- (37);
\draw[thick] (27) -- (87);
\draw[thick] (81) -- (71);
\draw[thick] (81) -- (87);
\draw[thick] (81) -- (261);
\draw[thick] (37) -- (67);
\draw[thick] (71) -- (67);
\draw[thick] (71) -- (251);
\draw[thick] (87) -- (67);
\draw[thick] (87) -- (267);
\draw[thick] (261) -- (251);
\draw[thick] (261) -- (267);
\draw[thick] (67) -- (247);
\draw[thick] (251) -- (247);
\draw[thick] (267) -- (247);
\end{tikzpicture}
\end{equation*}

Then, Theorem~\ref{substruct} tells us that the lattice of submodules of $\ctlmod[267]{21}$ is equal to the above lattice of up-admissible sets, so that the Alperin diagram of $\ctlmod[267]{21}$ is given by
\begin{equation*}
\begin{tikzpicture}[centered]
\node[inner sep=5pt] (21) at (0,0) {$21$};
\node[inner sep=5pt] (41) at (-2,-1) {$41$};
\node[inner sep=5pt] (27) at (0,-1) {$27$};
\node[inner sep=5pt] (81) at (2,-1) {$81$};
\node[inner sep=5pt] (37) at (-2,-2) {$37$};
\node[inner sep=5pt] (71) at (0,-2) {$71$};
\node[inner sep=5pt] (87) at (2,-2) {$87$};
\node[inner sep=5pt] (261) at (4,-2) {$261$};
\node[inner sep=5pt] (67) at (0,-3) {$67$};
\node[inner sep=5pt] (251) at (2,-3) {$251$};
\node[inner sep=5pt] (267) at (4,-3) {$267$};
\node[inner sep=5pt] (247) at (2,-4) {$247$};

\draw[thick] (21) -- (41);
\draw[thick] (21) -- (27);
\draw[thick] (21) -- (81);
\draw[thick] (41) -- (37);
\draw[thick] (41) -- (71);
\draw[thick] (27) -- (37);
\draw[thick] (27) -- (87);
\draw[thick] (81) -- (71);
\draw[thick] (81) -- (87);
\draw[thick] (81) -- (261);
\draw[thick] (37) -- (67);
\draw[thick] (71) -- (67);
\draw[thick] (71) -- (251);
\draw[thick] (87) -- (67);
\draw[thick] (87) -- (267);
\draw[thick] (261) -- (251);
\draw[thick] (261) -- (267);
\draw[thick] (67) -- (247);
\draw[thick] (251) -- (247);
\draw[thick] (267) -- (247);
\end{tikzpicture}\;,
\end{equation*}
where each up-admissible set $S$ has been replaced by the corresponding upward reflection of $21$ along $S$.
\end{ex}

%% file: jantzen.tex
Historically, Jantzen filtrations \cite[\S 8]{jantzen1987} have been one of the primary tools used to probe the structure of cell modules. Though they only give a rough, layered view of the submodule structure, in some cases this has been sufficient. Since we already know the complete submodule structure of cell modules for Temperley--Lieb algebras from Theorem \ref{substruct}, it is somewhat superfluous to study their Jantzen filtrations. However, the Temperley-Lieb algebras admit a two-dimensional Jantzen-like filtration that is both illustrative and tractable.

\subsection{General framework}\label{section:framework}
In this section, let $R$ be a ring, $A$ an $R$-algebra and $M$ an $R$-module with a bilinear form.
We assume that $A$ has an anti-involution $\iota$ and that the form, $\langle\cdot, \cdot\rangle : M \times M \to R$ satisfies $\langle m_1, x \cdot m_2\rangle = \langle \iota x \cdot m_1, m_2\rangle$.
Further, suppose that $\F$ is a field which arises as some quotient of $R$.
\begin{deff}
  Suppose $M$ is an $A$-module. We say $N\subseteq M$ is \emph{$A$-closed} if $A \cdot N \subseteq N$.
\end{deff}

Note that if $N$ is also closed under addition, it becomes an $A$-module and that an alternative, equivalent definition is that $N$ is the union of $A$-submodules of $M$. As such we may talk about the composition factors of $N$ which are a subset of the composition factors of $M$.
\begin{deff}
  Let $I \subseteq R$ be a subset of elements of the base ring. Then we define $M(I) = \{m \in M \mid \langle m, M\rangle \subseteq I\}$.
\end{deff}

\begin{lem}
  $M(I)$ is $A$-closed. Further, if $I_1 \subseteq I_2$ then $M(I_1) \subseteq M(I_2)$ and if $I$ is an ideal of $R$, then $M(I)$ is an $A$-submodule of $M$.
\end{lem}
\begin{proof}
  Firstly let $m \in M(I)$ and $a\in A$. Then 
  \begin{equation}
    \langle am, M\rangle = \langle m, \iota a \cdot M\rangle \subseteq \langle m, M\rangle \subseteq I.
  \end{equation}
  The inclusion criterion follows immediately and if $I$ is an ideal, then it is closed under addition so that
  \begin{equation}
      \langle m_1 + m_2, M\rangle \subseteq \langle m_1 , M\rangle + \langle m_2, M\rangle = I + I \subseteq I
  \end{equation}
  for any $m_1$ and $m_2$ in $M(I)$. Hence $M(I)$ is closed under addition and so a submodule.
\end{proof}

\begin{rem}
  Note that, in fact, we only required that $I$ be closed under addition to show that $M(I)$ is a submodule. However, in most useful cases, $I$ will either be an ideal or the union of ideals, unless $M$ is not $R$-free.    
\end{rem}
Thus if we have a lattice of subsets of $R$, say $\mathbb S \subseteq \mathcal{P}(R)$, we invoke a lattice of $A$-closed subsets of $M$ which we write $M(\mathbb S)$.

Now, we can form $\overline{M(\mathbb S)}$, the lattice of $\overline A = A \otimes_R \F$-closed subsets of $\overline M = M\otimes_R\F$ given by $\{\overline {M(I)} = M(I)\otimes_R\F \mid I \in \mathbb{S}\}$.

Suppose $N$ is a composition factor of $\overline{M}$. Then let us write
$$\mathbb{S}_N(M) = \{I \in \mathbb S \mid N \text{ is a composition factor of } \overline{M(I)}\}.$$
Then $\mathbb{S}_N(M)$ is a sublattice of $\mathbb S$.

If $R$ is a principal ideal domain, $\mathfrak{p}\subset R$ is a maximal ideal, and $\mathbb{S}=\{\mathfrak{p}^k\mid k\geq 0\}$, then the lattice of submodules $\{\overline{M(\mathfrak{p}^k)}\mid k\geq 0\}$ precisely corresponds to the Jantzen filtration of $M$ with respect to the ideal $\mathfrak{p}$, and determining the sublattices $\mathbb{S}_N(M)$ for every composition factor $N$ of $\overline{M}$ is equivalent to computing the layers of the filtration.

Our aim is to apply this more general framework to the cell modules of the Temperley--Lieb algebra, and to study Jantzen-like filtrations of these modules arising when we take $\mathbb{S}$ to be a two-dimensional grid of ideals of the ring $\Zm$.

\subsection{Gram matrices of cell modules}
The first step is to study the Gram matrices of cell modules over the ring $\Zm$. The following lemma is a rewriting of some intermediary results in \cite{Sutton2023,Tubbenhauer2021}.

\begin{lem}\label{zbasis}
There is an isomorphism of $\Q(\delta)$-vector spaces
\begin{equation*}
\Hom_{\tlcat^{\Q(\delta)}}(\jwlpz_m,\jwlpz_l)\cong\spn_{\Q(\delta)}(\jwlpz_l\zupr_S\zdor_{S'}\jwlpz_m),
\end{equation*}
where the right side ranges over all pairs $S,S'$ such that $S$ is up-admissible for $l$, $S'$ is down-admissible for $m$ and $m[S']=l(S)$. Moreover, if an element
\begin{equation*}
\sum_{S,S'}r_{S,S'}\jwlpz_l\zupr_S\zdor_{S'}\jwlpz_m\in\Hom_{\tlcat^{\Q(\delta)}}(\jwlpz_m,\jwlpz_l)
\end{equation*}
descends to a well-defined element of $\Hom_{\tlcat^{\kk}}(\jwlp_m,\jwlp_l)$, then every coefficient $r_{S,S'}$ lies in $\Zm$.
\end{lem}
\begin{proof}
The first part is a mixed case analogue of \cite[Lemma~3.19(b)]{Tubbenhauer2021}, which is known to be valid by the results in Section~3 of \cite{Sutton2023}. Similarly, the second part is a mixed case analogue of \cite[Lemma~3.20]{Tubbenhauer2021}.
\end{proof}

This lemma can be used to show the following, which says that in the Gram matrix of a cell module $\ctlmod[m(S)]{m}$, the diagonal entry corresponding to a basis element generating the copy of the trivial submodule $\stlmod[m(S)]{m(S)}$ is the only non-zero entry in its row and column. This is even true when computing the Gram matrix in $\Zm$ rather than $\kk$. Recall the element $\vv[m(S)]{}$ from Section~\ref{section:diag}, and keep the same notation for the same element over $\Zm$.

\begin{cor}\label{gramblock}
Let $m\in\Z_{\geq 0}$ and let $S$ be an up-admissible set for $m$. Consider the cell module $\ctlmod[m(S)]{m}$ for $\TL_{m(S)}^{\Zm}(\delta)$ and let $x\in\ctlmod[m(S)]{m}$. Then, $\ip{\vv[m(S)]{},x}=0$ unless $x=r\cdot\vv[m(S)]{}$ for some $r\in\Zm$.
\end{cor}
\begin{proof}
Let $x\in\ctlmod[m(S)]{m}$ such that $\ip{\vv[m(S)]{},x}\neq 0$, i.e.
\begin{equation*}
\begin{tikzpicture}[scale=0.25,mylabels,centered]
\draw[st] (3,1) -- (3,2);
\draw[st] (5,1) -- (5,2);
\draw[st] (7,1) -- (7,2);
\draw[jwlpz] (2,2) rectangle (8,4);
\node at (5,3) {$m$};
\draw[morg] (2,4) -- (8,4) -- (10,6) -- (0,6) -- cycle;
\node at (5,5) {$\upr_S$};
\draw[jwlpz] (0,6) rectangle (10,8);
\node at (5,7) {$m(S)$};
\draw[morg] (0,8) -- (10,8) -- (8,10) -- (2,10) -- cycle;
\node at (5,9) {$x^*$};
\draw[jwlpz] (2,10) rectangle (8,12);
\node at (5,11) {$m$};
\draw[st] (3,12) -- (3,13);
\draw[st] (5,12) -- (5,13);
\draw[st] (7,12) -- (7,13);
\end{tikzpicture}\;
=a\;
\begin{tikzpicture}[scale=0.25,mylabels,centered]
\draw[st] (3,0) -- (3,4);
\draw[st] (5,0) -- (5,4);
\draw[st] (7,0) -- (7,4);
\draw[jwlpz] (2,1) rectangle (8,3);
\node at (5,2) {$m$};
\end{tikzpicture}\;\mod \td(< m),
\end{equation*}
for some $0\neq a \in\Zm$. Then, $x$ needs to be a nonzero element of $\Hom_{\tlcat^{\Zm}}(\jwlpz_m,\jwlpz_{m(S)})$ not factoring through any $\jwlpz_l$ for $l<m$. Using the basis from Lemma~\ref{zbasis}, it follows that the only possibilities are scalar multiples of $\jwlpz_{m(S)}\zupr_S\jwlpz_m=\vv[m(S)]{}$, as desired.
\end{proof}

Let $\mathbb{S}$ be a subset of ideals of $\Zm$. A consequence of Corollary~\ref{gramblock} is the following, which gives a way to study the lattice of submodules of a cell module $\ctlmod{m}$ corresponding to the elements of $\mathbb{S}$, as defined in Section~\ref{section:framework}.

\begin{prop}\label{propsublattice}
Let $S$ be an up-admissible set for $m$ and let $n\geq m(S)$. Then,
\begin{equation*}
\mathbb{S}_{\stlmod{m(S)}}(\ctlmod{m})=\{I\in\mathbb{S}\mid \ip{\vv[m(S)]{},\vv[m(S)]{}}\in I\}.
\end{equation*}
\end{prop}
\begin{proof}
This is where one can take full advantage of truncation functors and the fact that they preserve filtrations in the following sense: for any $r\geq 0$ and any ideal $I$,
\begin{equation*}
\trunc{r}(\ctlmod{m}(I))\cong\ctlmod[r]{m}(I)\cong(\trunc{r}\ctlmod{m})(I).
\end{equation*}
The second isomorphism follows directly from Proposition~\ref{truncation} and we will show that the first one is given by the restriction of the isomorphism $\phi:\trunc{r}\ctlmod{m}\to\ctlmod[r]{m}$, from Section~\ref{sect:truncation} to the submodule $\trunc{r}(\ctlmod{m}(I))$.

If $x\in\trunc{r}(\ctlmod{m}(I))$, then $x=\yz{r}u$ for some $u\in\ctlmod{m}(I)$ and $\phi(x)=(\zz{r})^*u\in\ctlmod[r]{m}$. Making use of contravariance from Lemma~\ref{contravariance}, this gives, for any $v\in\ctlmod[r]{m}$,
\begin{equation*}
\ip{\phi(x),v}=\ip{(\zz{r})^*u,v}=\ip{u,\zz{r}v}\in I,
\end{equation*}
since $u\in\ctlmod{m}(I)$. This shows that $\phi(\trunc{r}(\ctlmod{m}(I)))\subseteq \ctlmod[r]{m}(I)$.

Reciprocally, if $v\in\ctlmod[r]{m}(I)$, then $v=\id_r v=(\zz{r})^*\yy{r}v=\phi(\yz{r}\yy{r}v)=\phi(\yy{r} v)$, and for any $u\in\ctlmod{m}$,
\begin{equation*}
\ip{\yy{r}v,u}=\ip{v,(\yy{r})^*u}\in I,
\end{equation*}
since $(\yy{r})^*u\in\ctlmod[r]{m}$ and $v\in\ctlmod[r]{m}(I)$. This shows that $v\in\phi(\trunc{r}(\ctlmod{m}(I)))$, showing the reverse inclusion, and we conclude that $\phi(\trunc{r}(\ctlmod{m}(I)))=\ctlmod[r]{m}(I)$ as desired.

Thus, truncating if necessary, we may assume that $n=m(S)$ in $\mathbb{S}_{\stlmod{m(S)}}(\ctlmod{m})$. In that case, the copy of $\stlmod[m(S)]{m(S)}$ corresponds to the one-dimensional subspace spanned by (the image of) $v_{m(S)}$, and the result follows from Corollary~\ref{gramblock}.
\end{proof}

This says the the value of $\ip{\vv[m(S)]{},\vv[m(S)]{}}$ in the Gram matrix of $\ctlmod[m(S)]{m}$ determines the position of $\stlmod{m(S)}$ in Jantzen-like filtrations of $\ctlmod{m}$. The precise value of that entry will thus be of crucial importance in the next section.

\begin{prop}
Let $S$ be an up-admissible set for $m$. Then, in $\ctlmod[m(S)]{m}$ over $\Zm$,
\begin{equation*}
\ip{\vv[m(S)]{},\vv[m(S)]{}}=\prod_{s\in S}\frac{\qnd{a_{m(S),s}[S]+1}}{\qnd{a_{m(S),s-1}[S]+1}}\;.
\end{equation*}
\end{prop}
\begin{proof}
Write $S$ as a disjoint union of up-admissible stretches, $S=\sqcup_{i=1}^{l} S_i$, with $S_i>S_{i-1}$. We will the show the result by induction on $l$.

If $l=1$, then $S=\{s_1,\ldots,s_j\}$ is itself a stretch. Write $m(S)+1=[a_k,a_{k-1},\ldots,a_0]_{\ell,\p}$ and let $x=[a_k,\ldots,a_{s_j+1},0,\ldots,0]_{\ell,\p}-1$, $y=\sum_{s\in S}a_s\p^{(s)}$, and $z=m(S)-x-y$. Then, $\ip{\vv[m(S)]{},\vv[m(S)]{}}$ is the coefficient in front of the term of maximal through degree in
\begin{equation*}
\jwlpz_m\zdor_S\zupr_S\jwlpz_m=\;
\begin{tikzpicture}[scale=0.25,mylabels,centered]
\draw[st] (3,1) -- (3,2);
\draw[st] (5,1) -- (5,2);
\draw[st] (7,1) -- (7,2);
\draw[jwlpz] (2,2) rectangle (8,4);
\node at (5,3) {$m$};
\draw[morg] (2,4) -- (8,4) -- (10,6) -- (0,6) -- cycle;
\node at (5,5) {$\zupr_S$};
\draw[jwlpz] (0,6) rectangle (10,8);
\node at (5,7) {$m(S)$};
\draw[morg] (0,8) -- (10,8) -- (8,10) -- (2,10) -- cycle;
\node at (5,9) {$\zdor_S$};
\draw[jwlpz] (2,10) rectangle (8,12);
\node at (5,11) {$m$};
\draw[st] (3,12) -- (3,13);
\draw[st] (5,12) -- (5,13);
\draw[st] (7,12) -- (7,13);
\end{tikzpicture}\;
=\;
\begin{tikzpicture}[scale=0.25,mylabels,centered]
\draw[st] (3,1) -- (3,2);
\draw[st] (5,1) -- (5,2);
\draw[st] (7,1) -- (7,2);
\draw[jwlpz] (2,2) rectangle (8,4);
\node at (5,3) {$m$};
\tlcup{5,6}{6,5}{7,6};
\node at (6,4.5) {$S$};
\draw[st] (3,4)..controls{(3,5.2) and (1,4.8)}..(1,6);
\draw[st] (5,4)..controls{(5,5.2) and (3,4.8)}..(3,6);
\draw[st] (7,4)..controls{(7,5.2) and (9,4.8)}..(9,6);
\draw[jwlpz] (0,6) rectangle (10,8);
\node at (5,7) {$m(S)$};
\begin{scope}[yscale=-1,shift={(0,-14)}]
\tlcup{5,6}{6,5}{7,6};
\draw[st] (3,4)..controls{(3,5.2) and (1,4.8)}..(1,6);
\draw[st] (5,4)..controls{(5,5.2) and (3,4.8)}..(3,6);
\draw[st] (7,4)..controls{(7,5.2) and (9,4.8)}..(9,6);
\end{scope}
\node at (6,9.5) {$S$};
\draw[jwlpz] (2,10) rectangle (8,12);
\node at (5,11) {$m$};
\draw[st] (3,12) -- (3,13);
\draw[st] (5,12) -- (5,13);
\draw[st] (7,12) -- (7,13);
\end{tikzpicture}\;
=\;
\begin{tikzpicture}[scale=0.25,mylabels,centered]
\draw[st] (3,1) -- (3,2);
\draw[st] (5,1) -- (5,2);
\draw[st] (7,1) -- (7,2);
\draw[jwlpz] (2,2) rectangle (8,4);
\node at (5,3) {$m$};
\tlcup{5,6}{6,5}{7,6};
\node at (7.6,7) {$y$};
\draw[st] (3,4)..controls{(3,5.2) and (1,4.8)}..(1,6);
\draw[st] (5,4)..controls{(5,5.2) and (3,4.8)}..(3,6);
\draw[st] (7,4)..controls{(7,5.2) and (9,4.8)}..(9,6);
\draw[jwlpz] (0,6) rectangle (6,8);
\node at (3,7) {$x$};
\draw[st] (7,6) -- (7,8);
\draw[st] (9,6) -- (9,8);
\begin{scope}[yscale=-1,shift={(0,-14)}]
\tlcup{5,6}{6,5}{7,6};
\draw[st] (3,4)..controls{(3,5.2) and (1,4.8)}..(1,6);
\draw[st] (5,4)..controls{(5,5.2) and (3,4.8)}..(3,6);
\draw[st] (7,4)..controls{(7,5.2) and (9,4.8)}..(9,6);
\end{scope}
\node at (9.6,7) {$z$};
\draw[jwlpz] (2,10) rectangle (8,12);
\node at (5,11) {$m$};
\draw[st] (3,12) -- (3,13);
\draw[st] (5,12) -- (5,13);
\draw[st] (7,12) -- (7,13);
\end{tikzpicture}\;.
\end{equation*}
Expanding $\jwlpz_x$, we may rewrite the above as
\begin{align}
\label{expandedtrace}
\begin{tikzpicture}[scale=0.25,mylabels,centered]
\draw[st] (3,1) -- (3,2);
\draw[st] (5,1) -- (5,2);
\draw[st] (7,1) -- (7,2);
\draw[jwlpz] (2,2) rectangle (8,4);
\node at (5,3) {$m$};
\tlcup{5,6}{6,5}{7,6};
\node at (7.6,7) {$y$};
\draw[st] (3,4)..controls{(3,5.2) and (1,4.8)}..(1,6);
\draw[st] (5,4)..controls{(5,5.2) and (3,4.8)}..(3,6);
\draw[st] (7,4)..controls{(7,5.2) and (9,4.8)}..(9,6);
\draw[jw] (0,6) rectangle (6,8);
\node at (3,7) {$x$};
\draw[st] (7,6) -- (7,8);
\draw[st] (9,6) -- (9,8);
\begin{scope}[yscale=-1,shift={(0,-14)}]
\tlcup{5,6}{6,5}{7,6};
\draw[st] (3,4)..controls{(3,5.2) and (1,4.8)}..(1,6);
\draw[st] (5,4)..controls{(5,5.2) and (3,4.8)}..(3,6);
\draw[st] (7,4)..controls{(7,5.2) and (9,4.8)}..(9,6);
\end{scope}
\node at (9.6,7) {$z$};
\draw[jwlpz] (2,10) rectangle (8,12);
\node at (5,11) {$m$};
\draw[st] (3,12) -- (3,13);
\draw[st] (5,12) -- (5,13);
\draw[st] (7,12) -- (7,13);
\end{tikzpicture}\;
+\;
\begin{tikzpicture}[scale=0.25,mylabels,centered]
\draw[st] (3,1) -- (3,2);
\draw[st] (5,1) -- (5,2);
\draw[st] (7,1) -- (7,2);
\draw[jwlpz] (2,2) rectangle (8,4);
\node at (5,3) {$m$};
\tlcup{5,6}{6,5}{7,6};
\node at (7.6,7) {$y$};
\draw[st] (3,4)..controls{(3,5.2) and (1,4.8)}..(1,6);
\draw[st] (5,4)..controls{(5,5.2) and (3,4.8)}..(3,6);
\draw[st] (7,4)..controls{(7,5.2) and (9,4.8)}..(9,6);
\draw[empty] (0,6) rectangle (6,8);
\node at (3,7) {rest};
\draw[st] (7,6) -- (7,8);
\draw[st] (9,6) -- (9,8);
\begin{scope}[yscale=-1,shift={(0,-14)}]
\tlcup{5,6}{6,5}{7,6};
\draw[st] (3,4)..controls{(3,5.2) and (1,4.8)}..(1,6);
\draw[st] (5,4)..controls{(5,5.2) and (3,4.8)}..(3,6);
\draw[st] (7,4)..controls{(7,5.2) and (9,4.8)}..(9,6);
\end{scope}
\node at (9.6,7) {$z$};
\draw[jwlpz] (2,10) rectangle (8,12);
\node at (5,11) {$m$};
\draw[st] (3,12) -- (3,13);
\draw[st] (5,12) -- (5,13);
\draw[st] (7,12) -- (7,13);
\end{tikzpicture}\;
&=\frac{\qnd{x+1}}{\qnd{x+1-y}}\;
\begin{tikzpicture}[scale=0.25,mylabels,centered]
\draw[st] (3,1) -- (3,2);
\draw[st] (5,1) -- (5,2);
\draw[st] (7,1) -- (7,2);
\draw[jwlpz] (2,2) rectangle (8,4);
\node at (5,3) {$m$};
\node at (7.6,7) {$z$};
\draw[st] (3,4) -- (3,10);
\draw[st] (5,4) -- (5,10);
\draw[st] (7,4) -- (7,10);
\draw[jw] (2,6) rectangle (6,8);
\node at (4,7) {$x-y$};
\draw[jwlpz] (2,10) rectangle (8,12);
\node at (5,11) {$m$};
\draw[st] (3,12) -- (3,13);
\draw[st] (5,12) -- (5,13);
\draw[st] (7,12) -- (7,13);
\end{tikzpicture}\;
+\;
\begin{tikzpicture}[scale=0.25,mylabels,centered]
\draw[st] (3,1) -- (3,2);
\draw[st] (5,1) -- (5,2);
\draw[st] (7,1) -- (7,2);
\draw[jwlpz] (2,2) rectangle (8,4);
\node at (5,3) {$m$};
\tlcup{5,6}{6,5}{7,6};
\node at (7.6,7) {$y$};
\draw[st] (3,4)..controls{(3,5.2) and (1,4.8)}..(1,6);
\draw[st] (5,4)..controls{(5,5.2) and (3,4.8)}..(3,6);
\draw[st] (7,4)..controls{(7,5.2) and (9,4.8)}..(9,6);
\draw[empty] (0,6) rectangle (6,8);
\node at (3,7) {rest};
\draw[st] (7,6) -- (7,8);
\draw[st] (9,6) -- (9,8);
\begin{scope}[yscale=-1,shift={(0,-14)}]
\tlcup{5,6}{6,5}{7,6};
\draw[st] (3,4)..controls{(3,5.2) and (1,4.8)}..(1,6);
\draw[st] (5,4)..controls{(5,5.2) and (3,4.8)}..(3,6);
\draw[st] (7,4)..controls{(7,5.2) and (9,4.8)}..(9,6);
\end{scope}
\node at (9.6,7) {$z$};
\draw[jwlpz] (2,10) rectangle (8,12);
\node at (5,11) {$m$};
\draw[st] (3,12) -- (3,13);
\draw[st] (5,12) -- (5,13);
\draw[st] (7,12) -- (7,13);
\end{tikzpicture}\\
\label{endtrace}
&=\frac{\qnd{x+1}}{\qnd{x+1-y}}\;
\begin{tikzpicture}[scale=0.25,mylabels,centered]
\draw[st] (3,1) -- (3,2);
\draw[st] (5,1) -- (5,2);
\draw[st] (7,1) -- (7,2);
\draw[jwlpz] (2,2) rectangle (8,4);
\node at (5,3) {$m$};
\draw[st] (3,4) -- (3,5);
\draw[st] (5,4) -- (5,5);
\draw[st] (7,4) -- (7,5);
\end{tikzpicture}\;
+\td(<m),
\end{align}
where the last equality follows from the fact that the second term consists solely of terms of lower through degree. This is trivial in the case where $x$ is Eve, and otherwise recall from Lemma~\ref{jwlpzexp} that one can write every remaining term of $\jwlpz_x$ in the form
\begin{equation}\label{restjwlp}
\begin{tikzpicture}[scale=0.25,mylabels,centered]
\draw[st] (1,2) -- (1,3);
\draw[st] (3,2) -- (3,3);
\draw[st] (5,2) -- (5,3);
\draw[st] (7,2) -- (7,3);
\draw[st] (9,2) -- (9,3);
\draw[st] (11,2) -- (11,10);
\draw[morg] (0,3) -- (10,3) -- (8,5) -- (2,5) -- cycle;
\node at (5,4) {$\ssdor_{S'}$};
\draw[jw] (2,5) rectangle (12,7);
\node at (7,6) {$x[S']$};
\draw[morg] (2,7) -- (8,7) -- (10,9) -- (0,9) -- cycle;
\node at (5,8) {$\ssupr_{S'}$};
\draw[st] (1,9) -- (1,10);
\draw[st] (3,9) -- (3,10);
\draw[st] (5,9) -- (5,10);
\draw[st] (7,9) -- (7,10);
\draw[st] (9,9) -- (9,10);
\node at (13.2,3) {$a_t\p^{(t)}$};
\end{tikzpicture}
\quad\mbox{or}\quad
\begin{tikzpicture}[scale=0.25,mylabels,centered]
\draw[st] (1,2) -- (1,3);
\draw[st] (3,2) -- (3,11);
\draw[st] (5,2) -- (5,11);
\draw[st] (7,2) -- (7,3);
\draw[st] (9,2) -- (9,3);
\draw[st] (11,2) -- (11,5);
\draw[morg] (0,3) -- (10,3) -- (8,5) -- (2,5) -- cycle;
\node at (5,4) {$\ssdor_{S'}$};
\draw[jw] (2,7) rectangle (6,9);
\node at (4,8) {$x[S'][t]$};
\tlcap{7,5}{9,7}{11,5};
\tlcup{7,11}{9,9}{11,11};
\begin{scope}[shift={(0,4)}]
\draw[morg] (2,7) -- (8,7) -- (10,9) -- (0,9) -- cycle;
\node at (5,8) {$\ssupr_{S'}$};
\draw[st] (1,9) -- (1,10);
\draw[st] (3,9) -- (3,10);
\draw[st] (5,9) -- (5,10);
\draw[st] (7,9) -- (7,10);
\draw[st] (9,9) -- (9,10);
\end{scope}
\draw[st] (11,11) -- (11,14);
\node at (13.2,4) {$a_t\p^{(t)}$};
\end{tikzpicture}\;,
\end{equation}
where $S'$ is a down-admissible set for $m_x$ and is non-empty in the first case, and where $a_t$ is the first non-zero coefficient in the $(\ell,\p)$-expansion of $x$. Tracing off the last $y$ strands of either diagram, we end up with elements of through degree at most $x[S']-y<x[t]-y<x[t]$ in the first case, and $x[S'][t]+y\leq x[t]+y$ in the second case. Therefore, even adding the $z$ additional strands as in \eqref{expandedtrace}, these terms end up having through degree less than $x[t]+y+z=m(S)[t]<m,$ the last inequality holding because $t>S$, and \eqref{endtrace} follows. Finally, a simple computation gives
\begin{equation*}
\frac{\qnd{x+1}}{\qnd{x+1-y}}=\frac{\qnd{a_{m(S),s_j}[S]+1}}{\qnd{a_{m(S),s_1}[S]+1}}=\prod_{k=1}^{j}\frac{\qnd{a_{m(S),s_k}[S]+1}}{\qnd{a_{m(S),s_{k}-1}[S]+1}}\;,
\end{equation*}
as desired.

Let $l\geq 2$ and write $S=\sqcup_{i=1}^{l}S_i=S_l\sqcup S'$, where $S_l<S'$. As in the base case, write $m(S)+1=[a_k,a_{k-1},\ldots,a_0]_{\ell,\p}$, let $x=[a_k,\ldots,a_{s_j+1},0,\ldots,0]_{\ell,\p}-1$, where $s_j$ is the largest element of $S'$, and let $t$ be the index of the first non-zero $(\ell,\p)$-digit of $x$. Then, using the inductive hypothesis for $S'$ and using the same procedure as in the base case for $S_l$ gives
\begin{align*}
\jwlpz_m\zdor_S\zupr_S\jwlpz_m &=\jwlpz_m\zdor_{S_l}\zdor_{S'}\zupr_{S'}\zupr_{S_l}\jwlpz_m=\;
\begin{tikzpicture}[scale=0.25,mylabels,centered]
\draw[st] (3,1) -- (3,2);
\draw[st] (5,1) -- (5,2);
\draw[st] (7,1) -- (7,2);
\draw[jwlpz] (2,2) rectangle (8,4);
\node at (5,3) {$m$};
\draw[morg] (2,4) -- (8,4) -- (10,6) -- (0,6) -- cycle;
\node at (5,5) {$\zupr_{S_l}$};
\draw[jwlpz] (0,6) rectangle (10,8);
\node at (5,7) {$m(S_l)$};
\draw[morg] (0,8) -- (10,8) -- (12,10) -- (-2,10) -- cycle;
\node at (5,9) {$\zupr_{S'}$};
\draw[jwlpz] (-2,10) rectangle (12,12);
\node at (5,11) {$m(S)$};
\draw[morg] (-2,12) -- (12,12) -- (10,14) -- (0,14) -- cycle;
\node at (5,13) {$\zdor_{S'}$};
\draw[jwlpz] (0,14) rectangle (10,16);
\node at (5,15) {$m(S_l)$};
\begin{scope}[shift={(0,8)}]
\draw[morg] (0,8) -- (10,8) -- (8,10) -- (2,10) -- cycle;
\node at (5,9) {$\zdor_{S_l}$};
\draw[jwlpz] (2,10) rectangle (8,12);
\node at (5,11) {$m$};
\draw[st] (3,12) -- (3,13);
\draw[st] (5,12) -- (5,13);
\draw[st] (7,12) -- (7,13);
\end{scope}
\end{tikzpicture}\\
&=\prod_{s\in S'}\frac{\qnd{a_{m(S),s}[S]+1}}{\qnd{a_{m(S),s-1}[S]+1}}\;
\begin{tikzpicture}[scale=0.25,mylabels,centered]
\draw[st] (3,1) -- (3,2);
\draw[st] (5,1) -- (5,2);
\draw[st] (7,1) -- (7,2);
\draw[jwlpz] (2,2) rectangle (8,4);
\node at (5,3) {$m$};
\draw[morg] (2,4) -- (8,4) -- (10,6) -- (0,6) -- cycle;
\node at (5,5) {$\zupr_{S_l}$};
\draw[jwlpz] (0,6) rectangle (10,8);
\node at (5,7) {$m(S_l)$};
\begin{scope}[shift={(0,0)}]
\draw[morg] (0,8) -- (10,8) -- (8,10) -- (2,10) -- cycle;
\node at (5,9) {$\zdor_{S_l}$};
\draw[jwlpz] (2,10) rectangle (8,12);
\node at (5,11) {$m$};
\draw[st] (3,12) -- (3,13);
\draw[st] (5,12) -- (5,13);
\draw[st] (7,12) -- (7,13);
\end{scope}
\end{tikzpicture}\;
+\;
\begin{tikzpicture}[scale=0.25,mylabels,centered]
\draw[st] (3,1) -- (3,2);
\draw[st] (5,1) -- (5,2);
\draw[st] (7,1) -- (7,2);
\draw[jwlpz] (2,2) rectangle (8,4);
\node at (5,3) {$m$};
\draw[morg] (2,4) -- (8,4) -- (10,6) -- (0,6) -- cycle;
\node at (5,5) {$\zupr_{S_l}$};
\draw[jwlpz] (0,6) rectangle (10,8);
\node at (5,7) {$m(S_l)$};
\draw[empty] (0,8) rectangle (10,10);
\node at (5,9) {$\td(<m(S)[t])$};
\draw[jwlpz] (0,10) rectangle (10,12);
\node at (5,11) {$m(S_l)$};
\begin{scope}[shift={(0,4)}]
\draw[morg] (0,8) -- (10,8) -- (8,10) -- (2,10) -- cycle;
\node at (5,9) {$\zdor_{S_l}$};
\draw[jwlpz] (2,10) rectangle (8,12);
\node at (5,11) {$m$};
\draw[st] (3,12) -- (3,13);
\draw[st] (5,12) -- (5,13);
\draw[st] (7,12) -- (7,13);
\end{scope}
\end{tikzpicture}\\
&=\prod_{s\in S}\frac{\qnd{a_{m(S),s}[S]+1}}{\qnd{a_{m(S),s-1}[S]+1}}\;
\begin{tikzpicture}[scale=0.25,mylabels,centered]
\draw[st] (3,1) -- (3,5);
\draw[st] (5,1) -- (5,5);
\draw[st] (7,1) -- (7,5);
\draw[jwlpz] (2,2) rectangle (8,4);
\node at (5,3) {$m$};
\end{tikzpicture}\;
+\td(<m),
\end{align*}
where we used the fact that $m(S)[t]<m$. The proposition follows.
\end{proof}

\begin{rem}
This shows that the value of $\ip{\vv[m(S)]{},\vv[m(S)]{}}$ is equal to the inverse of the coefficient $\lambda_{m(S)}^S$, from Definition~\ref{jwlpzdef}, that appears in $\jwlpz_m$. This is perhaps not so surprising, as these coefficients are defined as the inverses of a particular local intersection form (in the language of \cite{Elias2015}) that would be computed by looking at the coefficient in front of the identity in $\jw_m\ssdor_{S}\ssupr_{S}\jw_m$. What is surprising is that the value does not change when replacing the usual Jones--Wenzl elements by their semisimple $(\ell,\p)$-analogues.
\end{rem}

\subsection{Two-dimensional Jantzen filtrations} 
Of particular interest in the context of the mixed characteristic $(\ell,\p)$ is the two-dimensional grid of ideals given by
\begin{equation*}
\begin{tikzpicture}[centered]
\foreach \i in {1,...,3}
{
    \foreach \j in {1,...,3}
    {
        \ifnum\i=1
            \ifnum\j=1
                \node at (2*\i,-\j) {$(\p,m_\delta)$};
            \else
                \node at (2*\i,-\j) {$(\p,m_\delta^\j)$};
            \fi
        \else
            \ifnum\j=1
                \node at (2*\i,-\j) {$(\p^\i,m_\delta)$};
            \else
                \node at (2*\i,-\j) {$(\p^\i,m_\delta^\j)$};
            \fi
        \fi
    }
}
\foreach \i in {1,...,4}
{
    \foreach \j in {1,...,3}
    {
        \ifnum\i=4
            \node at (2*\i-0.5,-\j) {$\cdots$};
        \else
            \node at (2*\i+1,-\j) {$\supset$};
        \fi
    }
}
\foreach \i in {1,...,3}
{
    \foreach \j in {1,...,4}
    {
        \ifnum\j=4
            \node[rotate=270] at (2*\i,-\j) {$\cdots$};
        \else
            \node[rotate=270] at (2*\i,-\j-0.5) {$\supset$};
        \fi
    }
}
\end{tikzpicture}\;.
\end{equation*}
Hence, from now on, we fix $\mathbb{S}=\{(\p^i,m_\delta^j)\mid i,j\geq 0\}$. Computing two-dimensional Jantzen filtrations cell modules will rely on careful examination of the quantum integers appearing in the corresponding Gram matrices, for which the following will be useful.

\begin{lem}\label{qidentity}
For any integer $k \ge 1$, 
\begin{equation}
 \qnn{\p^{(k+1)}}=\qnn{\p^{(k)}}\left(\qnn{\p}^{p^{k-1}}\qnn{\ell}^{(p-1)p^{k-1}}+\p F\right),
\end{equation} where $F\in\Z[\delta]$ is not divisible by $m_\delta$ nor $\p$.
\end{lem}
\begin{proof}
Even though we are really interested in these expressions as polynomials in $\delta$, the first part of the proof is simpler to write in terms of expressions in $q$ and $q^{-1}$. Using Lemma~\ref{qnum}, we have
\begin{align*}
\qn{\p^{(k+1)}}=\qn{\p^{(k)}}\qn[q^{\p^{(k)}}]{p}&=\qn{\p^{(k)}}\left(\frac{q^{\ell\p^k}-q^{-\ell\p^k}}{q^{\ell\p^{k-1}}-q^{-\ell\p^{k-1}}}\right)\\
&\equiv \qn{\p^{(k)}}\left(\frac{q^{\ell\p}-q^{-\ell\p}}{q^\ell-q^{-\ell}}\right)^{\p^{k-1}}\pmod{p}\\
&= \qn{\p^{(k)}}\left(\frac{\qn{\ell\p}}{\qn{\ell}}\right)^{\p^{k-1}}\\
&\equiv \qn{\p^{(k)}}\left(\frac{\qn{\p}\qn{\ell}^p}{\qn{\ell}}\right)^{\p^{k-1}}\pmod{p}\\
&= \qn{\p^{(k)}}\qn{\p}^{p^{k-1}}\qn{\ell}^{(p-1)p^{k-1}},
\end{align*} and thus the same thing holds for polynomials in $\delta$ and the equality in the statement holds for some $F\in\Z[\delta]$. All that is left to show is that $m_\delta$ and $\p$ do not divide $F$. We have $\qnd{\p^{(k+1)}}=\qnd{\p^{(k)}}\qnd[\qnd{\p^{(k)}+1}-\qnd{\p^{(k)}-1}]{p}$ and $\qnd{\p^{(k)}+1}-\qnd{\p^{(k)}-1}\equiv \pm 2\pmod{m_\delta}$, so
\begin{equation*}
\qnd{\p}^{p^{k-1}}\qnd{\ell}^{(p-1)p^{k-1}}+\p F=\qnd[\qnd{\p^{(k)}+1}-\qnd{\p^{(k)}-1}]{p}\equiv \qnd[\pm 2]{p}=\pm p\not\equiv 0\pmod{m_\delta},
\end{equation*}
showing that $m_\delta$ doesn't divide $\qnd{\p}^{p^{k-1}}\qnd{\ell}^{(p-1)p^{k-1}}+\p F$. Since $m_\delta$ divides the first term, it cannot divide $F$. Finally, writing $$F=\frac{1}{\p}\left(\qnd[\qnd{\p^{(k)}+1}-\qnd{\p^{(k)}-1}]{p}-\qnd{\p}^{\p^{k-1}}\qnd{\ell}^{(\p-1)\p^{k-1}}\right)$$ and substituting $\delta=2$ gives $1-\p^{\p^{k-1}-1}\ell^{(\p-1)\p^{k-1}}$, which is not divisible by $\p$, and the result follows.
\end{proof}

\begin{deff}
For any integer $i\geq 1$ and any element $f\in\Zm$, let $$\xi_i(f)=\max\{j\in\Z_{\geq 0} : f\in (p^i,m_\delta^j)\}.$$
\end{deff}

Let $S$ be an up-admissible set for $m$ and let $\gamma_S:=\ip{\vv[m(S)]{},\vv[m(S)]{}}$, where the image of $\vv[m(S)]{}$ in $\ctlmod[m(S)]{m}$ generates a submodule with head isomorphic to the trivial module $\stlmod[m(S)]{m(S)}$. Then, using Proposition~\ref{propsublattice} and the definition above, we have
\begin{equation*}
\mathbb{S}_{\stlmod{m(S)}}(\ctlmod{m})=\{(\p^i,m_\delta^j) \mid \xi_i(\gamma_S)\geq 1\mbox{ and }j\leq \xi_i(\gamma_S)\},
\end{equation*}
for any $n\geq m(S)$. Therefore, Theorem~\ref{theofilt} completely solves the problem of computing two-dimensional Jantzen filtrations corresponding to $\mathbb{S}$, for any cell module $\ctlmod{m}$ in any mixed characteristic.

\begin{nota}
Write $S^i$ for the set $S$ with its $i$ greatest elements removed, with the convention that $S^0=S$ and $S^i=\varnothing$ if $i\geq \abs{S}$.
\end{nota}

\begin{theo}\label{theofilt}
Suppose that $\p\neq 2$. The coefficients $\gamma_S$ satisfy
\begin{equation*}
\xi_i\left(\gamma_S\right)=
\sum_{\substack{s\in S^{i-1}\\ s\neq 0}}(\p^s-\p^{s-1}) +\delta_{0\in S}.
\end{equation*}
On the other hand, if $\p=2$, then
\begin{equation*}
\xi_i\left(\gamma_S\right)=\begin{cases}
2\left(\sum_{\substack{s\in S^{i-1}\\ s\neq 0}}(\p^s-\p^{s-1}) +\delta_{0\in S}\right)& \mbox{if } i\leq \abs{S},\\
\delta_{0\in S} & \mbox{if } i>\abs{S}.
\end{cases}
\end{equation*}
\end{theo}
\begin{proof}
First, write $S=\sqcup_{k=1}^lS_k=\sqcup_{k=1}^l\{s_{k,1},\ldots,s_{k,j_k}\}$ as a disjoint union of stretches, and note that
\begin{align*}
\gamma_S=\prod_{s\in S}\frac{\qnd{a_{m(S),s}[S]+1}}{\qnd{a_{m(S),s-1}[S]+1}}&=\prod_{k=1}^l\frac{\qnd{\alpha_k\p^{(s_{k,j_k}+1)}}}{\qnd{\beta_k\p^{(s_{k,1})}}}\\
&=\prod_{k=1}^l\frac{\qn[\qnd{\p^{(s_{k,j_k}+1)}+1}-\qnd{\p^{(s_{k,j_k}+1)}-1}]{\alpha_k}}{\qn[\qnd{\p^{(s_{k,1})}+1}-\qnd{\p^{(s_{k,1})}-1}]{\beta_k}}\prod_{k=1}^l\frac{\qnd{\p^{(s_{k,j_k}+1)}}}{\qnd{\p^{(s_{k,1})}}},
\end{align*}
for some integers satisfying $0<\alpha_k,\beta_k<\p$ (or $0<\beta_k<\ell$ if $s_{k,1}=0$) since $S$ is up-admissible for $m$. This implies that
\begin{equation*}
\qn[\qnd{\p^{(s_{k,j_k}+1)}+1}-\qnd{\p^{(s_{k,j_k}+1)}-1}]{\alpha_k}\equiv\pm\alpha_k\not\equiv 0\pmod{(\p,m_\delta)},
\end{equation*}
and similarly for $\qn[\qnd{\p^{(s_{k,1})}+1}-\qnd{\p^{(s_{k,1})}-1}]{\beta_k}$, so that these elements are invertible in $\Zm$ and can be ignored for the rest of this proof. That is, it is sufficient to compute the value of $\xi_i$ on expressions of the form
\begin{equation*}
\gamma_S=\prod_{k=1}^l\frac{\qnd{\p^{(s_{k,j_k}+1)}}}{\qnd{\p^{(s_{k,1})}}}=\prod_{s\in S}\frac{\qnd{\p^{(s+1)}}}{\qnd{\p^{(s)}}},
\end{equation*}
and we record here that this last expression of $\gamma_S$ makes sense even if $S$ is not an up-admissible set for $m$. In particular, all of the $\gamma_{S^i}$, for $i\leq\abs{S}$, are well-defined elements of $\Zm$ and this will be crucial to use induction later.

If $\abs{S}=1$, then $S=\{s\}$. If $s>0$, then by Lemma~\ref{qidentity} there exists some $F\in\Z[\delta]$ such that $\gamma_S=\qnd{\p}^{\p^{s-1}}\qnd{\ell}^{(\p-1)\p^{s-1}}+\p F$ and $\p,m_\delta\nmid F$, so that $\xi_1(\gamma_S)=\p^s-\p^{s-1}$ if $\p\neq 2$ and $\xi_1(\gamma_S)=2(\p^s-\p^{s-1})$ if $\p=2$, while $\xi_i(\gamma_S)=0$ for all $i\geq 2$ and any $\p$. If $s=0$, then $\gamma_S=\qnd{\ell}$, so $\xi_1(\gamma_S)=1$ if $\p\neq 2$ and $\xi_1(\gamma_S)=2$ if $\p=2$, and $\xi_i(\gamma_S)=1$ for any other $i$ and any $\p$, as desired.

If $\abs{S}=k+1$, then $S=S^1\sqcup\{s\}$, where $s$ is the greatest element of $S$. Using Lemma~\ref{qidentity} again,
\begin{equation*}
\gamma_S =\frac{\qnd{\p^{(s+1)}}}{\qnd{\p^{(s)}}}\gamma_{S^1}=
\qnd{\p}^{\p^{s-1}}\qnd{\ell}^{(\p-1)\p^{s-1}}\gamma_{S^1}+\p G\gamma_{S^1},
\end{equation*}
where $G\in\Z[\delta]$ and $\p,m_\delta\nmid G$. By induction, this implies, if $\p\neq 2$,
\begin{equation*}
\xi_1(\gamma_S)=\p^s-\p^{s-1}+\xi_1(\gamma_{S^1})=\sum_{\substack{s\in S\\ s\neq 0}}(\p^s-\p^{s-1}) +\delta_{0\in S},
\end{equation*}
and
\begin{equation*}
\xi_i(\gamma_S)=\xi_{i-1}(\gamma_{S^1})=\sum_{\substack{s\in (S^1)^{i-2}\\ s\neq 0}}(\p^s-\p^{s-1}) +\delta_{0\in S^1}
\end{equation*}
which gives the desired result since $(S^1)^{i-2}=S^{i-1}$ and $0\in S$ if and only if $0\in S^1$. On the other hand, if $\p=2$,
\begin{equation*}
\xi_1(\gamma_S)=2(\p^s-\p^{s-1})+\xi_1(\gamma_{S^1})=2\left(\sum_{\substack{s\in S\\s\neq 0}} (\p^s-\p^{s-1})+\delta_{0\in S}\right),
\end{equation*}
and
\begin{equation*}
\xi_i(\gamma_S)=\xi_{i-1}(\gamma_{S^1})=\begin{cases}
2\left(\sum_{\substack{s\in (S^1)^{i-2}\\ s\neq 0}}(\p^s-\p^{s-1}) +\delta_{0\in S^1}\right)& \mbox{if } i-1\leq \abs{S^1},\\
\delta_{0\in S} & \mbox{if } i-1>\abs{S^1},
\end{cases}
\end{equation*}
which is again the desired result, noting that $\abs{S^1}=\abs{S}-1$. The theorem follows.
\end{proof}

\begin{ex}
Taking $\kk=\Z[\delta]/(\p,m_\delta)=\Z[\delta]/(2,\delta+1)$ so that $\ell=3$ and $\p=2$, consider the cell module $\ctlmod[200]{18}$ for $\TL_{200}$. Then
\begin{equation*}
18+1=[1,1,0,1]_{3,2},
\end{equation*}
so that the Alperin diagram of the cell module is the following, where every composition factor is listed with its corresponding up-admissible set:
\begin{equation*}
\begin{tikzpicture}[mylabels,centered]

\node[inner sep=1pt] (18) at (-1,0) {\begin{tabular}{c}$18$\\ $\{\}$\end{tabular}};
\node[inner sep=1pt] (22) at (-1,-1) {\begin{tabular}{c}$22$\\ $\{0\}$\end{tabular}};
\node[inner sep=1pt] (42) at (1,-1) {\begin{tabular}{c}$42$\\ $\{3\}$\end{tabular}};
\node[inner sep=1pt] (46) at (-1,-2) {\begin{tabular}{c}$46$\\ $\{0,3\}$\end{tabular}};
\node[inner sep=1pt] (30) at (1,-2) {\begin{tabular}{c}$30$\\ $\{2,3\}$\end{tabular}};
\node[inner sep=1pt] (90) at (3,-2) {\begin{tabular}{c}$90$\\ $\{3,4\}$\end{tabular}};
\node[inner sep=1pt] (34) at (-1,-3) {\begin{tabular}{c}$34$\\ $\{0,2,3\}$\end{tabular}};
\node[inner sep=1pt] (94) at (1,-3) {\begin{tabular}{c}$94$\\ $\{0,3,4\}$\end{tabular}};
\node[inner sep=1pt] (78) at (3,-3) {\begin{tabular}{c}$78$\\ $\{2,3,4\}$\end{tabular}};
\node[inner sep=1pt] (186) at (5,-3) {\begin{tabular}{c}$186$\\ $\{3,4,5\}$\end{tabular}};
\node[inner sep=1pt] (28) at (-1,-4) {\begin{tabular}{c}$28$\\ $\{0,1,2,3\}$\end{tabular}};
\node[inner sep=1pt] (82) at (1,-4) {\begin{tabular}{c}$82$\\ $\{0,2,3,4\}$\end{tabular}};
\node[inner sep=1pt] (190) at (3,-4) {\begin{tabular}{c}$190$\\ $\{0,3,4,5\}$\end{tabular}};
\node[inner sep=1pt] (174) at (5,-4) {\begin{tabular}{c}$174$\\ $\{2,3,4,5\}$\end{tabular}};
\node[inner sep=1pt] (76) at (1,-5) {\begin{tabular}{c}$76$\\ $\{0,1,2,3,4\}$\end{tabular}};
\node[inner sep=1pt] (178) at (3,-5) {\begin{tabular}{c}$178$\\ $\{0,2,3,4,5\}$\end{tabular}};
\node[inner sep=1pt] (172) at (1,-6) {\begin{tabular}{c}$172$\\ $\{0,1,2,3,4,5\}$\end{tabular}};

\draw[thick] (18) -- (22);
\draw[thick] (18) -- (42);
\draw[thick] (22) -- (46);
\draw[thick] (42) -- (30);
\draw[thick] (42) -- (46);
\draw[thick] (42) -- (90);
\draw[thick] (30) -- (34);
\draw[thick] (30) -- (78);
\draw[thick] (46) -- (34);
\draw[thick] (46) -- (94);
\draw[thick] (90) -- (78);
\draw[thick] (90) -- (94);
\draw[thick] (90) -- (186);
\draw[thick] (34) -- (28);
\draw[thick] (34) -- (82);
\draw[thick] (78) -- (82);
\draw[thick] (78) -- (174);
\draw[thick] (94) -- (82);
\draw[thick] (94) -- (190);
\draw[thick] (186) -- (174);
\draw[thick] (186) -- (190);
\draw[thick] (28) -- (76);
\draw[thick] (82) -- (76);
\draw[thick] (82) -- (178);
\draw[thick] (174) -- (178);
\draw[thick] (190) -- (178);
\draw[thick] (178) -- (172);
\draw[thick] (76) -- (172);

\draw[line width=10pt,color=red,opacity=0.2,rounded corners] (18) -- (22) -- (46) -- (34) -- (28) -- (76) -- (172);
\draw[line width=10pt,color=red,opacity=0.2,rounded corners] (18) -- (42) -- (90) -- (78) -- (174) -- (178) -- (172);

\end{tikzpicture}\; .
\end{equation*}
For each composition factor in the two highlighted chains, we want to draw the boundary of the ideals containing the corresponding $\gamma_S$ in the two-dimensional grid of ideals $\{(\p^i,m_\delta^j)\}_{i,j\geq 1}$, which is done by computing the value of $\xi_i(\gamma_S)$ for every $i$. We will sometimes abuse notation by writing $\xi_i(a)$ to denote $\xi_i(\gamma_{S_a})$, where $S_a$ is the up-admissible set corresponding to the simple composition factor $\stlmod[200]{a}$. We get the following pictures for the left chain and the right chain respectively:
\begin{equation*}
\begin{tikzpicture}[scale=0.4,centered]
\draw[step=1,dashed] (0,0) grid (9.9,-32.9);
\draw[color=\colzero,fill=\colzero,fill opacity=\jantzenopacity] (0,0) -- (0,-33) -- (1,-33) -- (1,-32) -- (2,-32) -- (2,-16) -- (3,-16) -- (3,-8) -- (4,-8) -- (4,-4) -- (5,-4) -- (5,-2) -- (6,-2) -- (6,-1) -- (10,-1) -- (10,0);
\draw[line width=1pt,color=black,fill=none] (1,-33) -- (1,-32) -- (2,-32) -- (2,-16) -- (3,-16) -- (3,-8) -- (4,-8) -- (4,-4) -- (5,-4) -- (5,-2) -- (6,-2) -- (6,-1) -- (10,-1);
\draw[color=\colone,fill=\colone,fill opacity=\jantzenopacity] (0,0) -- (0,-32) -- (1,-32) -- (1,-16) -- (2,-16) -- (2,-8) -- (3,-8) -- (3,-4) -- (4,-4) -- (4,-2) -- (5,-2) -- (5,-1) -- (10,-1) -- (10,0);
\draw[line width=1pt,color=black,fill=none] (0,-32) -- (1,-32) -- (1,-16) -- (2,-16) -- (2,-8) -- (3,-8) -- (3,-4) -- (4,-4) -- (4,-2) -- (5,-2) -- (5,-1) -- (10,-1);
\draw[color=\coltwo,fill=\coltwo,fill opacity=\jantzenopacity] (0,0) -- (0,-16) -- (1,-16) -- (1,-8) -- (2,-8) -- (2,-4) -- (3,-4) -- (3,-2) -- (4,-2) -- (4,-1) -- (10,-1) -- (10,0);
\draw[line width=1pt,color=black,fill=none] (0,-16) -- (1,-16) -- (1,-8) -- (2,-8) -- (2,-4) -- (3,-4) -- (3,-2) -- (4,-2) -- (4,-1) -- (10,-1);
\draw[color=\colthree,fill=\colthree,fill opacity=\jantzenopacity] (0,0) -- (0,-14) -- (1,-14) -- (1,-6) -- (2,-6) -- (2,-2) -- (3,-2) -- (3,-1) -- (4,-1) -- (4,-1) -- (10,-1) -- (10,0);
\draw[line width=1pt,color=black,fill=none] (0,-14) -- (1,-14) -- (1,-6) -- (2,-6) -- (2,-2) -- (3,-2) -- (3,-1) -- (4,-1) -- (4,-1) -- (10,-1);
\draw[color=\colfour,fill=\colfour,fill opacity=\jantzenopacity] (0,0) -- (0,-10) -- (1,-10) -- (1,-2) -- (2,-2) -- (2,-1) -- (3,-1) -- (3,-1) -- (4,-1) -- (4,-1) -- (10,-1) -- (10,0);
\draw[line width=1pt,color=black,fill=none] (0,-10) -- (1,-10) -- (1,-2) -- (2,-2) -- (2,-1) -- (3,-1) -- (3,-1) -- (4,-1) -- (4,-1) -- (10,-1);
\draw[color=\colfive,fill=\colfive,fill opacity=\jantzenopacity] (0,0) -- (0,-2) -- (1,-2) -- (1,-1) -- (2,-1) -- (2,-1) -- (3,-1) -- (3,-1) -- (4,-1) -- (4,-1) -- (10,-1) -- (10,0);
\draw[line width=1pt,color=black,fill=none] (0,-2) -- (1,-2) -- (1,-1) -- (2,-1) -- (2,-1) -- (3,-1) -- (3,-1) -- (4,-1) -- (4,-1) -- (10,-1);
\node[color=\colzero] (172) at (-2,-33) {$172$};
\node[color=\colone] (76) at (-2,-24) {$76$};
\node[color=\coltwo] (28) at (-2,-15) {$28$};
\node[color=\colthree] (34) at (-2,-12) {$34$};
\node[color=\colfour] (46) at (-2,-6) {$46$};
\node[color=\colfive] (22) at (-2,-1) {$22$};
\draw[thick,color=\colzero] (172) -- (0.5,-32.5);
\draw[thick,color=\colone] (76) -- (0.5,-24);
\draw[thick,color=\coltwo] (28) -- (0.5,-15);
\draw[thick,color=\colthree] (34) -- (0.5,-12);
\draw[thick,color=\colfour] (46) -- (0.5,-6);
\draw[thick,color=\colfive] (22) -- (0.5,-1);
\node at (0.5,-33.5) {$\vdots$};
\node at (10.8,-0.5) {$\cdots$};
\foreach \y in {1,...,33}
{
    \node at (-0.5,-\y+0.5) {{\scriptsize $\y$}};
}
\foreach \x in {1,...,10}
{
    \node at (\x-0.5,0.5) {{\scriptsize $\x$}};
}
\end{tikzpicture}
\hspace{10pt}
\mbox{and}
\hspace{10pt}
\begin{tikzpicture}[scale=0.4,centered]
\draw[step=1,dashed] (0,0) grid (9.9,-32.9);
\draw[color=\colzero,fill=\colzero,fill opacity=\jantzenopacity] (0,0) -- (0,-33) -- (1,-33) -- (1,-32) -- (2,-32) -- (2,-16) -- (3,-16) -- (3,-8) -- (4,-8) -- (4,-4) -- (5,-4) -- (5,-2) -- (6,-2) -- (6,-1) -- (10,-1) -- (10,0);
\draw[line width=1pt,color=black,fill=none] (1,-33) -- (1,-32) -- (2,-32) -- (2,-16) -- (3,-16) -- (3,-8) -- (4,-8) -- (4,-4) -- (5,-4) -- (5,-2) -- (6,-2) -- (6,-1) -- (10,-1);
\draw[color=\colone,fill=\colone,fill opacity=\jantzenopacity] (0,0) -- (0,-33) -- (1,-33) -- (1,-30) -- (2,-30) -- (2,-14) -- (3,-14) -- (3,-6) -- (4,-6) -- (4,-2) -- (5,-2) -- (5,-1) -- (10,-1) -- (10,0);
\draw[line width=1pt,color=black,fill=none] (1,-33) -- (1,-30) -- (2,-30) -- (2,-14) -- (3,-14) -- (3,-6) -- (4,-6) -- (4,-2) -- (5,-2) -- (5,-1) -- (10,-1);
\draw[color=\coltwo,fill=\coltwo,fill opacity=\jantzenopacity] (0,0) -- (0,-33) -- (1,-33) -- (1,-28) -- (2,-28) -- (2,-12) -- (3,-12) -- (3,-4) -- (4,-4) -- (4,-0);
\draw[line width=1pt,color=black,fill=none] (1,-33) -- (1,-28) -- (2,-28) -- (2,-12) -- (3,-12) -- (3,-4) -- (4,-4) -- (4,-0);
\draw[color=\colthree,fill=\colthree,fill opacity=\jantzenopacity] (0,0) -- (0,-28) -- (1,-28) -- (1,-12) -- (2,-12) -- (2,-4) -- (3,-4) -- (3,-0);
\draw[line width=1pt,color=black,fill=none] (0,-28) -- (1,-28) -- (1,-12) -- (2,-12) -- (2,-4) -- (3,-4) -- (3,-0);
\draw[color=\colfour,fill=\colfour,fill opacity=\jantzenopacity] (0,0) -- (0,-24) -- (1,-24) -- (1,-8) -- (2,-8) -- (2,0);
\draw[line width=1pt,color=black,fill=none] (0,-24) -- (1,-24) -- (1,-8) -- (2,-8) -- (2,0);
\draw[color=\colfive,fill=\colfive,fill opacity=\jantzenopacity] (0,0) -- (0,-8) -- (1,-8) -- (1,0);
\draw[line width=1pt,color=black,fill=none] (0,-8) -- (1,-8) -- (1,0);
\node[color=\colzero] (172) at (-2,-33.5) {$172$};
\node[color=\colone] (178) at (-2,-31.5) {$178$};
\node[color=\coltwo] (174) at (-2,-29) {$174$};
\node[color=\colthree] (78) at (-2,-26) {$78$};
\node[color=\colfour] (90) at (-2,-16) {$90$};
\node[color=\colfive] (42) at (-2,-4) {$42$};
\draw[thick,color=\colzero] (172) -- (1.5,-31.5);
\draw[thick,color=\colone] (178) -- (1.5,-29.5);
\draw[thick,color=\coltwo] (174) -- (0.5,-29);
\draw[thick,color=\colthree] (78) -- (0.5,-26);
\draw[thick,color=\colfour] (90) -- (0.5,-16);
\draw[thick,color=\colfive] (42) -- (0.5,-4);
\node at (0.5,-33.5) {$\vdots$};
\node at (10.8,-0.5) {$\cdots$};
\foreach \y in {1,...,33}
{
    \node at (-0.5,-\y+0.5) {{\scriptsize $\y$}};
}
\foreach \x in {1,...,10}
{
    \node at (\x-0.5,0.5) {{\scriptsize $\x$}};
}
\end{tikzpicture}\;,
\end{equation*}
where the finite values of $\xi_1(172)=64$, $\xi_1(178)=62$, and $\xi_1(174)=60$ were cut off for space considerations, whereas $\xi_i(m(S))$ stays equal to $1$ indefinitely for every $S$ containing $0$.

Taking $n=10 000$ so that $\ctlmod[10000]{18}$ has $42$ composition factors, we can draw the boundary of ideals as above for each one of them and organise them together as horizontal slices (in increasing order of the index of the composition factor) to get the image in Figure~\ref{fig:3Dlayers} that highlights the fractal nature of the modular representation theory of $\TL_n$ (c.f. \cite{Spencer2023,Sutton2023}).

\begin{figure}
\includegraphics[scale=0.5]{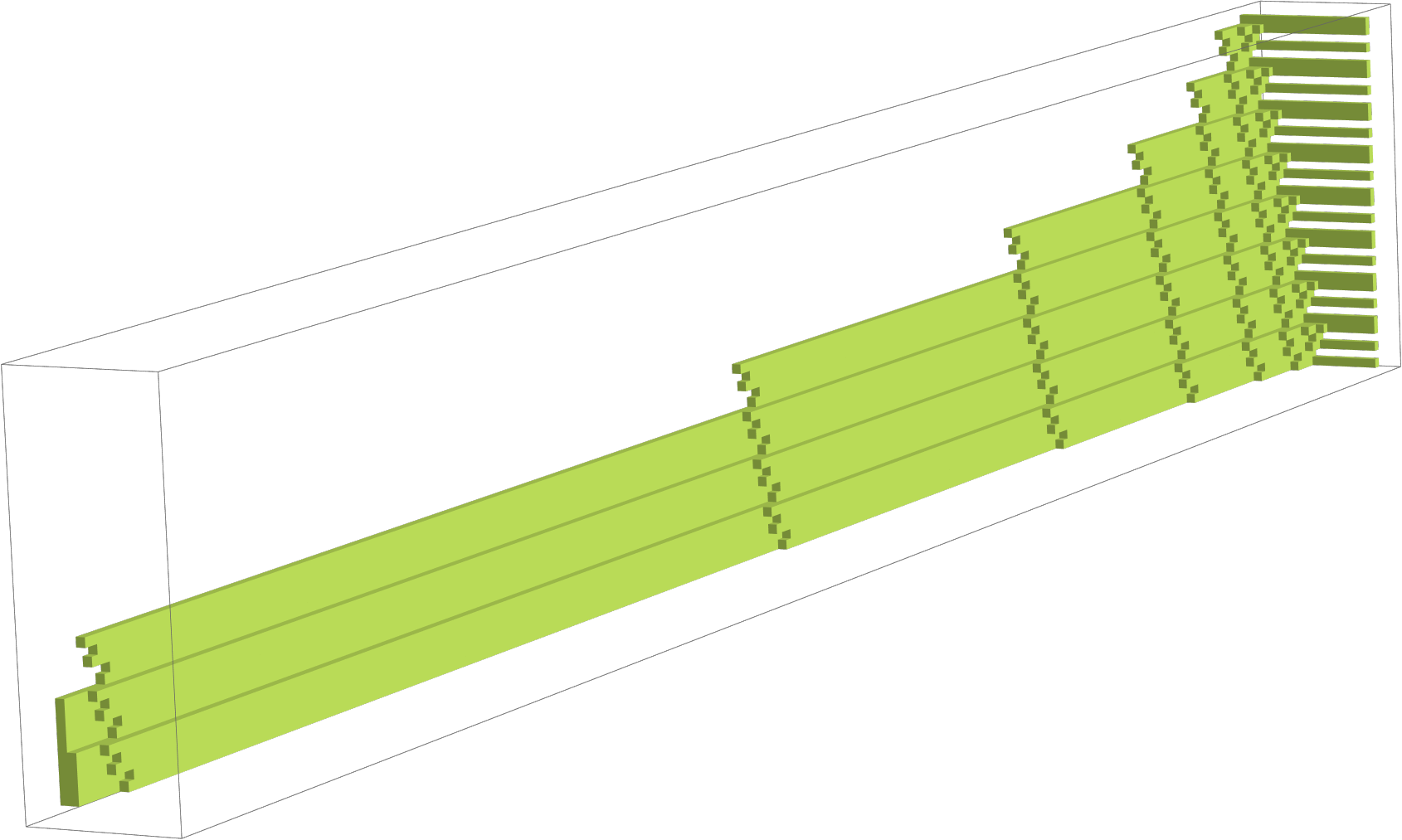}
\caption{The $42$ composition factors of $\ctlmod[10 000]{18}$.}
\label{fig:3Dlayers}
\end{figure}
\end{ex}

%% file: acknowledgements.tex
The authors would like to thank Dani Tubbenhauer and Geordie Williamson for interesting discussions. This work was supported by the Additional Funding Programme for Mathematical Sciences, delivered by EPSRC (EP/V521917/1) and the Heilbronn Institute for Mathematical Research. The second author was also supported by the Natural Sciences and Engineering Research Council of Canada.